\documentclass[12pt]{amsart}
\usepackage{amssymb,latexsym,amsmath,amsthm,amsfonts, enumerate}
\usepackage{color}
\usepackage[all]{xy}
\usepackage{tikz}
\usepackage{tikz-cd}
\usetikzlibrary{positioning,shapes,shadows,arrows,snakes,matrix,patterns,calc}
\usepackage[colorlinks=true,pagebackref,hyperindex]{hyperref}
\usepackage{pgfplots}
\usepgfplotslibrary{fillbetween}
\usepackage{makecell}

\advance \topmargin by -\headheight
\advance \topmargin by -\headsep
\evensidemargin \oddsidemargin
\numberwithin{equation}{section}

\newtheorem{theorem}{Theorem}[section]
\newtheorem{proposition}[theorem]{Proposition}
\newtheorem{conjecture}[theorem]{Conjecture}
\newtheorem{corollary}[theorem]{Corollary}
\newtheorem{lemma}[theorem]{Lemma}

\theoremstyle{definition}
\newtheorem{remark}[theorem]{Remark}
\newtheorem{example}[theorem]{Example}
\newtheorem{definition}[theorem]{Definition}

\let\oldmarginpar\marginpar
\renewcommand\marginpar[1]{\-\oldmarginpar[\raggedleft\small\sf
#1]{\raggedright\small\sf #1}}

\newcommand{\ZZ}{\mathbb{Z}}
\newcommand{\CC}{\mathbb{C}}

\newcommand{\QQ}{\mathbb{Q}}

\newcommand{\cA}{\mathcal{A}}

\newcommand{\cT}{\mathcal{T}}
\newcommand{\D}{\mathrm{Dom}}

\newcommand{\Hom}{\mathrm{Hom}}

\def\M{{\mathfrak M}}

\def \v{\mathbf{v}}

\newcommand{\myE}{\mathbf{E}}
\newcommand{\myF}{\mathcal{F}}
\newcommand{\myG}{\mathcal{G}}
\newcommand{\mytildeF}{\tilde{\mathcal{F}}}
\newcommand{\mytildeG}{\tilde{\mathcal{G}}}
\newcommand{\mytildeH}{\tilde{\mathcal{H}}}
\newcommand{\im}{\mathrm{im}}

\newcommand{\rank}{\mathrm{rank}}

\begin{document}

\title[Nakajima's quiver varieties and triangular bases of rank-2 cluster algebras]
{Nakajima's quiver varieties and triangular bases of rank-2 cluster algebras}
\author{Li Li}
\address{Department of Mathematics
and Statistics,
Oakland University, 
Rochester, MI 48309-4479, USA 
}
\email{li2345@oakland.edu}


\begin{abstract}
Berenstein and Zelevinsky introduced quantum cluster algebras \cite{BZ1} and the triangular bases \cite{BZ2}. The support conjecture in \cite{LLRZ} asserts that the support of a triangular basis element for a rank-2 cluster algebra is bounded by an explicitly described region that is possibly concave. In this paper, we prove the support conjecture for all skew-symmetric rank-2 cluster algebras. 
\end{abstract}

\maketitle
\tableofcontents

\section{Introduction}

Cluster algebras and quantum cluster algebras were introduced by Fomin-Zelevinsky \cite{fz-ClusterI} and Berenstein-Zelevinsky \cite{BZ1}, respectively. A main goal of introducing quantum cluster algebras is to understand good bases arising from the representation theory of certain non-associative algebras. Mimicking the construction of the dual canonical basis from a PBW basis \cite{Leclerc}, Berenstein and Zelevinsky \cite{BZ2} constructed the triangular basis for acyclic quantum cluster algebras. The triangular basis has many nice properties including (conjecturally) the strong positivity, that is, all structure constants are nonnegative. 
In this paper, we will focus on triangular bases for quantum cluster algebras of rank 2, and will not discuss those of higher ranks;  the definition of rank-2 quantum cluster algebras and their triangular bases will be reviewed in \S2.

In  \cite{LLRZ}, Lee, Rupel, Zelevinsky and the author proposed the following conjecture. For $x\in\mathbb{R}$, denote $[x]_+=\max\{x,0\}$.
Let $b,c$ be positive integers, every triangular basis of the (coefficient-free) rank-2 quantum cluster  algebra $\mathcal{A}_\v(b,c)$ is of the form
\begin{equation}\label{eq:C=e}
C[a_1,a_2]=\sum_{p,q} e(p,q) X^{(bp-a_1,cq-a_2)}
\end{equation}
where $a_1,a_2,p,q\in\mathbb{Z}$, and $0\le p\le [a_2]_+$, $0\le q\le [a_1]_+$. 
Define the set of positive imaginary roots as
$$\Phi^{im}_+=\{(a_1,a_2)\in\mathbb{Z}^2_{>0}: ca_1^2-bca_1a_2+ba_2^2\le0\}.$$

\begin{conjecture}\label{conj:triangular support}\cite[Conjecture 11]{LLRZ}
Let $b,c\in\mathbb{Z}_{>0}$,  $(a_1,a_2)\in\Phi^{im}_+$. Let $C[a_1,a_2]$ be a triangular basis element of the (coefficient-free) rank-2 quantum cluster  algebra $\mathcal{A}_\v(b,c)$.  
For integers $0\le p\le a_2$, $0\le q\le a_1$, the coefficient $e(p,q)$ of $C[a_1,a_2]$ is nonzero if and only if $D(p,q)\ge0$, where
\begin{equation}\label{eq:supp}
D(p,q)=ca_1q+ba_2p-bp^2-bcpq-cq^2.
\end{equation}
\end{conjecture}

To illustrate, the region of $\{(p,q) \in\mathbb{R}^2 \ |\  p,q\ge0, D(p,q)\ge0\}$ is the ``curved triangle'' OAC in Figure \ref{fig1} (6). The reader is also referred to \S\ref{examp1} for an example where $b=c=3$, $(a_1,a_2)=(3,4)\in\Phi^{im}_+$, and an example where $b=c=3$, $(a_1,a_2)=(2,8)\not\in\Phi^{im}_+$.

Our main theorem is to confirm the above conjecture in the skew-symmetric case, that is, when $b=c$. Throughout the paper except \S3, we assume that $b=c=r$.

\begin{theorem}\label{main theorem}
Assume $(b,c)=(r,r)$ where $r\in\mathbb{Z}_{>0}$,  $(a_1,a_2)\in\Phi^{im}_+$. For integers $0\le p\le a_2$, $0\le q\le a_1$, let $D=D(p,q)= r(a_1q+a_2p-p^2-rpq-q^2)$. Then the following holds:

{\rm(i)} $e(p,q)\neq0$ if and only if $D\ge0$.

{\rm(ii)} Fix $(p,q)$ such that $D\ge 0$. Then $\displaystyle e(p,q)=\sum_{-D\le i\le D} e_i\v^i$, where $e_i$'s are nonnegative integers satisfying the following conditions:

$\bullet$  symmetry: $e_i=e_{-i}$ for every $i$; 

$\bullet$ $e_D=e_{-D}=1$; 

$\bullet$ $e_i\neq0$ if and only if  $i\in\{-D,-D+2r,-D+4r,\dots,D\}$; 

$\bullet$ unimodality:

\quad -- if $2r|D$, then $e_{-D}\le e_{-D+2r}\le\cdots\le e_{-2r}\le e_0\ge e_{2r}\ge\cdots\ge e_{D-2r}\ge e_D$; 

\quad -- if $2r\nmid D$, then $e_{-D}\le e_{-D+2r}\le \cdots\le e_{-3r}\le e_{-r}=e_r\ge e_{3r}\ge\cdots\ge e_{D-2r}\ge e_D$. 
\end{theorem}

\bigskip

Next, as a corollary of  Theorem \ref{main theorem} and a study of the case of real roots (Theorem \ref{theorem: real root}), we give a complete description for the support of any triangular basis element.  Define ${\bf d}_n=\begin{bmatrix}d'_n\\d''_n\end{bmatrix}$ to be the denominator vector of the cluster variable $X_n$. For $x,y\in\mathbb{R}$, define 
$\begin{bmatrix}x\\y\end{bmatrix}_+=\begin{bmatrix}[x]_+ \\ [y]_+ \end{bmatrix}$. 
Define
$$T_n =\textrm{ the convex hull of } \{(0,0),([d''_n]_+,0),(0,[d'_n]_+)\}.$$ 
For $s\in\mathbb{R}_{\ge0}$ and $T\subseteq\mathbb{R}^2$,  define $sT=\{s z\ | \ z\in T\}$. 


\begin{definition}\label{def:R}
For $(a_1,a_2)\in\mathbb{Z}^2$, define a region $R(a_1,a_2)\subseteq \mathbb{R}^2$ as follows (see Figure \ref{fig1}):

{\rm(a)} If $(a_1,a_2)\notin\Phi^{im}_+$, we can write $(a_1,a_2)=s_1{\bf d}_n+s_2{\bf d}_{n+1}$ for some $n\in\mathbb{Z}$, $s_1,s_2\in\mathbb{Z}_{\ge0}$. Then define $R(a_1,a_2)$ be the following Minkowski sum:\footnote{To verify the shape of $R(a_1,a_2)$ of case (5) in Figure \ref{fig1}: assume $s_1,s_2>0, n\neq0,1,2$. Note that $s_1T_n+s_2T_{n+1}$ is the convex hull of 
$$O=(0,0), A=(s_1d''_n+s_2d''_{n+1},0), B=(s_2d''_{n+1},s_1d'_n), C=(0,s_1d'_n+s_2d'_{n+1}), D=(s_1d''_n,s_2d'_{n+1})$$
We claim that $OABC$ is convex, and $D$ lies inside the triangle $OAC$. To show this, it suffices to show that the vertices  $A,B,C$ lie counterclockwisely, and $A,C,D$ also lie counterclockwisely. The former statement follows from the determinant computation:
$\begin{vmatrix}
s_1d''_n+s_2d''_{n+1}&0&1\\
s_2d''_{n+1}&s_1d'_n&1\\
0&s_1d'_n+s_2d'_{n+1}&1\\
 \end{vmatrix}=s_1s_2
 \begin{vmatrix}
d'_n&d'_{n+1}\\
d''_n&d''_{n+1}\\
 \end{vmatrix}
=s_1s_2>0$. 
The latter statement follows from the fact that $ABCD$ is a parallelogram. 
}
$$ R(a_1,a_2)=s_1T_n+s_2T_{n+1}.$$

{\rm(b)} If $(a_1,a_2)\in\Phi^{im}_+$, then define 
$$R(a_1,a_2)=\{(p,q)\in\mathbb{R}_{\ge0}^2\ | \ D(p,q)\ge0\}. $$ 
\end{definition}

\begin{figure}[h]
\begin{tikzpicture}[scale=.7]
\begin{scope}[shift={(0, 5.5)}]
\usetikzlibrary{patterns}
\draw[->] (0,0) -- (2.5,0)
node[above] {\tiny $p$};
\draw[->] (0,0) -- (0,2.5)
node[left] {\tiny $q$};
\fill [blue]  (0,0) circle (2pt);
\draw (0.5,-.3) node[anchor=north] {\footnotesize (1) $a_1,a_2\le0$};
\draw (1.5,-.7) node[anchor=north] {\footnotesize i.e. $(a_1,a_2)=s_1{\bf d}_1+s_2{\bf d}_{2}$};
\end{scope}
\begin{scope}[shift={(6, 5.5)}]
\usetikzlibrary{patterns}
\draw[->] (0,0) -- (2.5,0)
node[above] {\tiny $p$};
\draw[->] (0,0) -- (0,2.5)
node[left] {\tiny $q$};
\fill (0,0) circle (1.5pt);
\fill (2,0) circle (1.5pt);
 \draw [blue,very thick] (0,0) -- (2,0);
\draw (0.5,-.3) node[anchor=north] {\footnotesize (2) $a_1\le 0< a_2$};
\draw (1.5,-.7) node[anchor=north] {\footnotesize i.e. $(a_1,a_2)=s_1{\bf d}_0+s_2{\bf d}_{1}$};
\draw (2,0) node[anchor=south] {\tiny$A$};
\end{scope}
\begin{scope}[shift={(12,5.5)}]
\usetikzlibrary{patterns}
\draw[->] (0,0) -- (2.5,0)
node[above] {\tiny $p$};
\draw[->] (0,0) -- (0,2.5)
node[left] {\tiny $q$};
\fill (0,0) circle (1.5pt);
\fill (0,2) circle (1.5pt);
\draw [blue, very thick] (0,0) -- (0,2);
\draw (0.5,-.3) node[anchor=north] {\footnotesize (3) $a_2\le 0< a_1$};
\draw (1.5,-.7) node[anchor=north] {\footnotesize i.e. $(a_1,a_2)=s_1{\bf d}_2+s_2{\bf d}_{3}$};
\draw (0,2) node[anchor=east] {\tiny$C$};
\end{scope}
\begin{scope}[shift={(0,0)}]
\usetikzlibrary{patterns}
\fill [blue!10]  (0,3)--(1.2,0)--(0,0)--(0,3);
\draw (0,3)--(1.2,0);
\draw (0,0) node[anchor=south west] {\tiny$sT_n$};
\draw[->] (0,0) -- (3.5,0)
node[above] {\tiny $p$};
\draw[->] (0,0) -- (0,3.5)
node[left] {\tiny $q$};
\draw (1.5,-.5) node[anchor=north] {\footnotesize (4) $(a_1,a_2)=s{\bf d}_n$,};
\draw (1.5,-1) node[anchor=north] {\footnotesize \hspace{10pt} $n\neq0,1,2,3$};
\draw (0,2.7) node[anchor=east] {\tiny$C$};
\draw (1.5,0) node[anchor=south] {\tiny$A$};
\fill (0,3) circle (1.5pt);
\fill (1.2,0) circle (1.5pt);
\end{scope}
\begin{scope}[shift={(6,0)}]
\usetikzlibrary{patterns}
\fill [blue!10]  (0,3)--(1.9,1.2)--(2.5,0)--(0,0)--(0,3);
\draw (0,3)--(1.9,1.2)--(2.5,0);
\draw (0,.3) node[anchor=south west] {\tiny$s_1T_n+s_2T_{n+1}$};
\draw[->] (0,0) -- (3.5,0)
node[above] {\tiny $p$};
\draw[->] (0,0) -- (0,3.5)
node[left] {\tiny $q$};
\fill (0,0) circle (1.5pt);
\fill (1.9,1.2) circle (1.5pt);
\fill (0,3) circle (1.5pt);
\fill (2.5,0) circle (1.5pt);
\draw (1.5,-.5) node[anchor=north] {\footnotesize (5) $(a_1,a_2)=s_1{\bf d}_n+s_2{\bf d}_{n+1}$,};
\draw (1.5,-1) node[anchor=north] {\footnotesize \hspace{10pt} $s_1,s_2>0$, $n\neq0,1,2$};
\draw (0,3) node[anchor=east] {\tiny$C$};
\draw (2.7,-.1) node[anchor=south] {\tiny$A$};
\draw (2,1.2) node[anchor=south] {\tiny$B$};
\end{scope}
\begin{scope}[shift={(12,0)}]
\usetikzlibrary{patterns}
\fill [blue!10] {(0,3) to [out=290, in=160] (2.5,0) to (0,0)};
\draw (0,0) node[anchor=east] {\tiny$O$};
\draw (2.5,0) node[anchor=south] {\tiny$A$};
\draw (.7,.2) node[anchor=south] {\tiny$D\ge0$};
\draw (0,3) node[anchor=east] {\tiny$C$};
\draw[->] (0,0) -- (3.5,0)
node[above] {\tiny $p$};
\draw[->] (0,0) -- (0,3.5)
node[left] {\tiny $q$};
\fill (0,0) circle (1.5pt);
\fill (0,3) circle (1.5pt);
\fill (2.5,0) circle (1.5pt);
\draw [black!70] {(0,3) to [out=290, in=160] (2.5,0)};
\draw (1.5,-.5) node[anchor=north] {\footnotesize (6) $(a_1,a_2)\in\Phi^{im}_+$};
\end{scope}
\end{tikzpicture}
\caption{The region $R(a_1,a_2)$. The cases (1)--(5) are for real roots $(a_1,a_2)\notin\Phi^{im}_+$, and $A=(a_2,0)$, $C=(0,a_1)$, $B=(s_2d''_{n+1},s_1d'_n)$.} 
\label{fig1}
\end{figure}
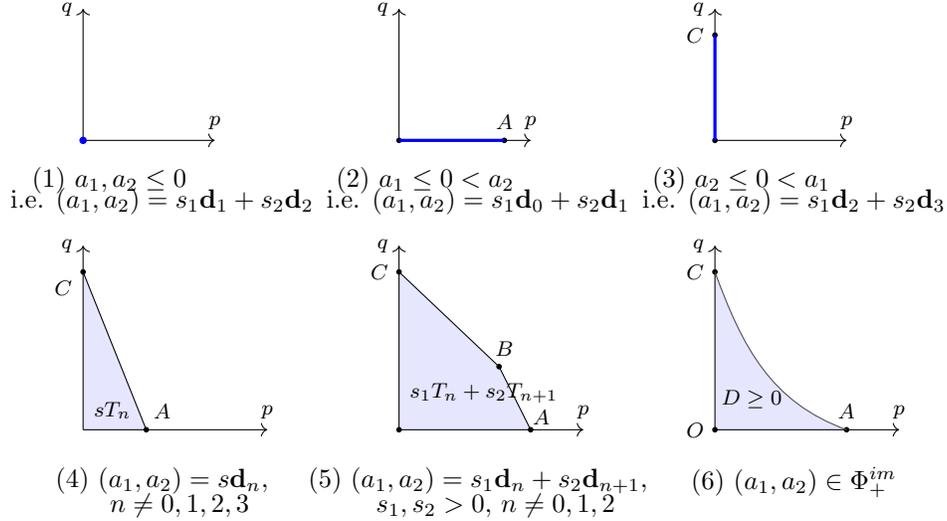

\begin{corollary}\label{cor:precise condition for e(p,q)neq0}
Let $(a_1,a_2)\in\mathbb{Z}^2$, let $\{e(p,q)\}$ be the coefficients of the triangular basis $C[a_1,a_2]$ as in \eqref{eq:C=e}. Then $e(p,q)\neq0$ if and only if $(p,q)\in R(a_1,a_2)$, and in this case we have
 $$\deg e(p,q)=D(p,q). $$ 
\end{corollary}


\bigskip

Now we would like to put the theorem into a geometric context. 
For an irreducible variety $X$, a local system $L$ on (an open dense subset of) $X$ and an integer $d$,  denote by $IC_X(L)[d]$ the intersection cohomology complex on $X$, with coefficients in $L$ shifted by $d$; for $d=0$,  $IC_X(L)[0]$ is just $IC_X(L)$; if $L=\mathbb{Q}$ is trivial, then we just denote  $IC_X(L)[d]$ by $IC_X[d]$. (Note that $IC_X(L)=\mathcal{IC}_X(L)[\dim V]$ and in particular, $IC_X=\mathbb{Q}[\dim X]$ when $X$ is nonsingular; please see \cite[\S1.6]{deCataldo2015} for the definition of  $\mathcal{IC}$ and the difference between $IC$ and $\mathcal{IC}$.)
The famous BBDG Decomposition theorem \cite{BBD} asserts that, for a proper map of algebraic varieties $f:X\to Y$, the pushforward of the intersection cohomology complex $IC_X$ can be decomposed (non-canonically) into a direct sum of $IC_V(L)[d]$ for a collection of irreducible subvarieties $V\subseteq Y$, local systems $L$ on (an open dense subset of) $V$, and integers $d$. The support problem (that is, to determine which subvarieties $V$ actually occur in the BBDG decomposition) can be quite difficult, as stated in a comment in Ngo's proof of the Fundamental Lemma of Langlands theory:  ``...Their supports constitute an important topological invariant of $f$. In general, it is difficult to explicitly determine this invariant'' (translated from French).
Several recent studies are focusing on the support problem of particular morphisms; the readers are referred to the references \cite{MS,Migliorini2015,deCataldo2017,MSV,deCHM}.

In the above context, we have a corollary of the main theorem, which gives a complete solution to the support problem  for the natural projection from a class of Nakajima's nonsingular graded quiver varieties to the affine  graded quiver varieties. 

First we fix some notations.
Let $v=(v_1,v_2)\in\mathbb{Z}_{\ge0}^2, w=(0,w'_1,w_2,0)\in\mathbb{Z}_{\ge0}^4$. 
Let ${\mytildeF_{v,w}}, \mathbf{E}_{v,w}$ be Nakajima's nonsingular graded quiver variety and affine graded quiver variety, respectively. (See \S3 for more related definitions and notations).  The condition $0\le v_1\le w_1' \textrm{ and }0\le v_2\le rv_1$ guarantees that that ${\mytildeF_{v,w}}$ is nonempty. 
There is a natural projection $\pi: {\mytildeF_{v,w}}\to \mathbf{E}_{v,w}$, and the BBDG decomposition takes the following form (and we will show that the local systems are all trivial):
\begin{equation}\label{eq:BBDG for pi}
\pi_*(IC_{\mytildeF_{v,w}})=\bigoplus_{v'=(v'_1,v'_2)}\bigoplus_{d}a^d_{v,v';w}IC_{\mathbf{E}_{v',w}}[d]. 
\end{equation}

\begin{corollary}\label{cor:BBDG support Nakajima}
Let $v=(v_1,v_2)\in\mathbb{Z}_{\ge0}^2, w=(0,w'_1,w_2,0)\in\mathbb{Z}_{\ge0}^4$ such that $0\le v_1\le w_1' \textrm{ and }0\le v_2\le rv_1$. The coefficients $a^d_{v,v';w}$ in \eqref{eq:BBDG for pi} are related to the triangular basis as follows: let
\begin{equation}\label{a1a2pq}
a_1= w'_1-rv'_2,\; a_2=r[a_1]_++rv'_1-w_2,\; p=v_2-v'_2,\; q=[a_1]_+-v_1+v'_1,
\end{equation}
and let $e(p,q)$ be determined by \eqref{eq:C=e}, then
$e(p,q)=\sum_d a_{v,v';w}^d\v^{rd}.$
\smallskip

As a consequence, a subvariety $\mathbf{E}_{v',w}$ appears in the BBDG decomposition \eqref{eq:BBDG for pi} if and only if `` $0\le v'_1\le v_1,\  0\le v'_2\le v_2$,\  and $(p,q)\in R(a_1,a_2)$ for $a_1,a_2,p,q$ defined in \eqref{a1a2pq}''.
\end{corollary}

The reader is referred to \S\ref{example for 1.5} for an example of the above corollary.

\smallskip

The paper is organized as follows. 
In \S2 we recall the definition and properties of quantum cluster algebra and triangular basis, and that triangular basis coincides with dual canonical basis.
In \S3 we study the support condition for cluster monomials.
In \S4 we recall some facts on Nakajima's quiver varieties.
In \S5 we prove an algebraic version of the transversal slice theorem, and show that the local systems appeared in the BBDG decomposition theorem are trivial.
In \S6 we introduce two other desingularizations of $\mathbf{E}_{v,w}$. 
In \S7 we study the triangular bases.
In \S8 we prove the above main theorem (Theorem \ref{main theorem}) and the two corollaries (Corolloaries \ref{cor:precise condition for e(p,q)neq0}, \ref{cor:BBDG support Nakajima}). 
In \S9 we discuss an analogue of the Kazhdan-Lusztig polynomial that comes from the geometry of this paper. 
We end the paper with some examples in \S10. 

\smallskip

{\bf Acknowledgments.} he author learned a lot from the collaborated work with Kyungyong Lee, Dylan Rupel, Andrei Zelevinsky, based on which the current work is possible; he also gratefully acknowledges discussions with Mark de Cataldo and Fan Qin; he thanks Pramod Achar for sharing a draft of his book, which is now available in print \cite{Achar}; he greatly appreciates the referee for carefully reading through the paper and providing many helpful comments and suggestions.
Computer calculations were performed using SageMath \cite{Sage}.

\section{Quantum cluster algebras and triangular bases}
In this section we recall the definition and some fundamental facts of rank-2 quantum cluster algebras. See \cite{BZ1,BZ2} for details. 

\subsection{Rank-2 quantum cluster algebras}
First, let $\cT$ be the quantum torus: 
 $$\cT=\ZZ[\v^{\pm 1}]\langle X_1^{\pm1},X_2^{\pm1}: X_2 X_1=\v^2 X_1 X_2\rangle, $$
let $\mathcal{F}$ be the skew-field of fractions of $\cT$.

The \emph{bar-involution} is the $\mathbb{Z}$-linear anti-automorphism of $\cT$ determined by $\overline{f}(\v)=f(\v^{-1})$ for  $f\in\mathbb{Z}[\v^{\pm1}]$ and
\[\overline{fX_1^{a_1}X_2^{a_2}}=\overline{f}X_2^{a_2}X_1^{a_1}=\v^{2a_1a_2}\overline{f}X_1^{a_1}X_2^{a_2}\quad\text{($a_1,a_2\in\ZZ$)}.\]
An element in $\cT$ which is invariant under the bar-involution is said to be \emph{bar-invariant}.

The \emph{rank-2 quantum cluster algebra} $\cA_\v(b,c)$  is the $\mathbb{Z}[\v^{\pm1}]$-subalgebra of $\mathcal{F}$ generated by the \emph{cluster variables} $X_m$ ($m \in\ZZ$) defined recursively from the \emph{exchange relations}
$$X_{m+1}X_{m-1}=\begin{cases}
&\v^{b}X_m^b+1, \textrm{ if $m$ is odd}; \\
&\v^{c}X_m^c+1, \textrm{ if $m$ is even}; \\
\end{cases}
$$
(In notations of \cite{BZ2}, we take $\Lambda=\begin{bmatrix}0&-b\\c&0\end{bmatrix}$, $B=\begin{bmatrix}0&-b\\ c&0\end{bmatrix}$, $d_1=c, d_2=b$, with a linear order $1\vartriangleleft2$.)

One can easily check that
 \begin{equation}\label{eq:commutation-A11}
  X_{m+1} X_m = \v^2 X_m X_{m+1} \quad (m \in \ZZ),
 \end{equation}
and that all cluster variables are bar-invariant.  Equation \eqref{eq:commutation-A11} implies that each \emph{cluster} $\{X_m, X_{m+1}\}$ generates a quantum torus
 $\cT_m=\ZZ[\v^{\pm1}]\langle X_m^{\pm1},X_{m+1}^{\pm1}: X_{m+1} X_m = \v^2 X_m X_{m+1}\rangle$.  
The (bar-invariant) \emph{quantum cluster monomials} in the quantum torus $\cT_m$ as
 $$X^{(a_1,a_2)}_m = \v^{a_1 a_2} X_m^{a_1} X_{m+1}^{a_2} \quad (a_1, a_2 \in \ZZ_{\geq 0}, \,\, m \in \ZZ).$$
Denote $X^{(a_1,a_2)}=X^{(a_1,a_2)}_1$. 
Note that $X^{(a_1,a_2)}X^{(b_1,b_2)}=\v^{a_2b_1-a_1b_2}X^{(a_1+b_1,a_2+b_2)}$.

\subsection{Triangular bases}

The construction of the triangular basis starts with the standard monomial basis $\{M[{\bf a}]\,:\,{\bf a}\in\ZZ^2\}$.
 For every ${\bf a}=(a_1,a_2) \in \ZZ^2$, the \emph{standard monomial} $M[{\bf a}]$ (which is denoted $E_{(-a_1,-a_2)}$ in \cite{BZ2}) is defined as follows, where $X'_1=X_3$, $X'_2=X_0$, and $v({\bf a})\in\mathbb{Z}$ is determined by the condition that the leading term of $M({\bf a})$ is bar-invariant):
 \begin{equation}\label{eq:Ma}
  M[{\bf a}] = \v^{v({\bf a})}X^{([-a_1]_+,[-a_2]_+)}(X_1')^{[a_1]_+}(X_2')^{[a_2]_+}= \v^{a_1 a_2} X_3^{[a_1]_+} X_1^{[-a_1]_+} X_2^{[-a_2]_+} X_0^{[a_2]_+} \ .
 \end{equation}
 It is known 
 that the elements $M[{\bf a}]$ form a $\ZZ[\v^{\pm 1}]$-basis of the cluster algebra $\cA_\v(b,c)$.

The standard monomials are not bar-invariant and do not contain all the cluster monomials, moreover they are inherently dependent on the choice of an initial cluster. These drawbacks provide a motivation to consider the triangular basis
(which is introduced in \cite{BZ2} and recalled below) constructed from the standard monomial basis with a
built-in bar-invariance property.

In this paper, for ${\bf a}'=(a_1',a_2'), {\bf a}=(a_1,a_2)\in\mathbb{Z}^2$, we denote 
$$\aligned
&{\bf a}'\le {\bf a}\  (\textrm{or } {\bf a}\ge{\bf a}'),  \textrm{  if $a_i'\le a_i$ for $i=1,2$;}\\ 
& {\bf a}'< {\bf a}\  (\textrm{or } {\bf a}>{\bf a}'), \textrm{ if ${\bf a}'\le {\bf a}$ and ${\bf a'}\neq {\bf a}$.}\\
\endaligned
$$ 
\begin{definition}\cite{BZ2}
The triangular basis $\{C[{\bf a}]\,:\,{\bf a}\in\ZZ^2\}$ is the unique collection of elements in $\cA_\v(b,c)$ satisfying:

 \begin{enumerate}
  \item[] (P1) Each $C[{\bf a}]$ is bar-invariant.
  \item[] (P2) For each ${\bf a}\in \ZZ^2$,
  $\displaystyle  C[{\bf a}] - M[{\bf a}] \in \bigoplus_{{\bf a}'<{\bf a}} \mathbf{v}\ZZ[\mathbf{v}]M[{\bf a}'].$
  \end{enumerate}
\end{definition}

 \smallskip
The triangular basis has the following nice property:
\begin{theorem}[{\cite[Theorem 1.6]{BZ2}}]
 The triangular basis does not depend on the choice of initial cluster and it contains all cluster monomials.
 \end{theorem}

\section{The case of real roots}
In this section, we study the support of a cluster monomial in a skew-symmetrizable quantum cluster algebra $\mathcal{A}_\v(b,c)$, that is, without assuming $b=c$. The denominator vectors $\{{\bf d}_n\}$ can be determined recursively by
$$
{\bf d}_1=\begin{bmatrix}-1\\0\end{bmatrix},\quad
{\bf d}_2=\begin{bmatrix}0\\-1\end{bmatrix},\quad
{\bf d}_{n+1}+{\bf d}_{n-1}
=\begin{cases}
&b[{\bf d}_n]_+, \textrm{ if $n$ is odd},\\
&c[{\bf d}_n]_+, \textrm{ if $n$ is even}.\\
\end{cases} 
$$
Below are a few values of ${\bf d}_n$:
\begin{center}
 \begin{tabular}{|c|c|c|c|c|c|c|c|c|c|c|c|c|} 
 \hline
 $n$ &$\cdots$ &$-3$&$-2$&$-1$&$0$&$1$&$2$&$3$&$4$&$5$ & $6$&$\cdots$\\ 
 \hline
         $d'_n$ &$\cdots$ & $bc-1$ & $b$  &$1$& $0$  & $-1$ &  $0$ & $1$ & $b$ & $bc-1$ & $b^2c-2b$ & $\cdots$\\
$d''_n$ & $\cdots$ & $bc^2-2c$ & $bc-1$ &$c$ & $1$ & $0$ & $-1$ & $0$ & $1$ & $c$ & $bc-1$&$\cdots$\\
  \hline
\end{tabular}
\end{center}

\begin{theorem}\label{theorem: real root}
Consider quantum cluster algebra $\mathcal{A}_\v(b,c)$ for arbitrary $b,c\in\mathbb{Z}_{>0}$. 
Let $(a_1,a_2)\in\mathbb{Z}^2\setminus \Phi^{im}_+$. 
(Thus $(a_1,a_2)=s_1{\bf d}_n+s_2{\bf d}_{n+1}$ for some $n\in\mathbb{Z}$ and $s_1,s_2\in\mathbb{Z}_{\ge0}$,  and $C[a_1,a_2]=\v^{s_1s_2}X_n^{s_1}X_{n+1}^{s_2}$ is a cluster monomial.)
Define $\{e(p,q)\}$ as in \eqref{eq:C=e}. Then $e(p,q)\neq0$ if and only if $(p,q)\in s_1T_n+s_2T_{n+1}=R(a_1,a_2)$, and for those $(p,q)$, we have $\deg e(p,q)=D(p,q)$. 
\end{theorem}
\begin{proof}

We prove the statement by induction on $|n-1|$. 

The base cases are $n=0,1,2$. 

For $n=1$: $C[a_1,a_2]=C[-s_1,-s_2]=X^{(s_1,s_2)}$; $e(p,q)=1$ if $(p,q)=(0,0)$, and $e(p,q)=0$ otherwise; $s_1T_1+s_2T_2=\{(0,0)\}$; for $(p,q)=0$, $\deg e(0,0)=0=D(0,0)$. So the statement is true in this case.

For $n=2$: $C[a_1,a_2]=C[s_2,-s_1]=X_2^{(s_1,s_2)}=\sum_{q=0}^{s_2}{s_2\brack q}_{\v^c}X^{(-s_2,s_1+cq)}$ (this computation appeared in the proof of \cite[Corollary 3.6]{LLRZ2}); $e(p,q)={s_2\brack q}_{\v^c}$ if $p=0$ and $0\le q\le s_2$, and $e(p,q)=0$ otherwise; $s_1T_2+s_2T_3=\{(0,y)\ |\ 0\le y\le s_2\}$; for $p=0$ and $0\le q\le s_2$, $\deg e(0,q)=(s_2-q)qc=D(0,q)$. So the statement is true in this case.

For $n=0$: similar as the $n=2$ case so we skip.

\smallskip

Recall in \cite{LLRZ2} we defined automorphisms $\sigma_\ell$ ($\ell=1,2$) on the quantum torus such that $\sigma_\ell(X_m)=X_{2\ell-m}$ and $\sigma_\ell(\v)=\v^{-1}$. 

For the inductive step, we shall prove the case $n\ge3$. 

Without loss of generality we assume $s_1\neq0$ (otherwise we can replace $(n,s_1,s_2)$ by $(n-1,s_2,0)$).  Denote $C[a'_1,a'_2]=\sigma_2(C[a_1,a_2])=v^{s_2s_1}X_{3-n}^{s_2}X_{4-n}^{s_1}$ (which is also a cluster monomial) and define $e'(p',q')$ by
$$C[a_1',a_2']=\sum_{p',q'} e'(p',q') X^{(bp'-a_1',cq'-a_2')}.$$ 
Then $a_1'=ba_2-a_1$, $a_2'=a_2$, and by inductive hypothesis, the statement is true for $C[a'_1,a'_2]$; that is, $e'(p',q')\neq0$ if and only if $(p,q)\in R(a'_1,a'_2)\ (=s_2T_{3-n}+s_1T_{4-n})$, and for those $(p',q')$, $\deg e'(p',q')=D'(p',q')$ where $D'(p',q')=ca'_1q'+ba'_2p'-bp'^2-bcp'q'-cq'^2$.

 For a fixed $p$, denote $p'=a_2-p$,  there is the following relation by \cite[Proof of Lemma 2.2]{LLRZ2}:
\begin{equation}\label{eq:22}
\sum_{q'}e'(p',q')t^{q'}=\left(\sum_{\ell\ge0}{{-a_1+bp}\brack{\ell}}_{\v^c}t^\ell\right)\left(\sum_{q}e(p,q)t^q\right)
\end{equation}
which we can also rewrite as
\begin{equation}\label{eq:22'}
\left(\sum_{\ell\ge0}{{a_1-bp}\brack{\ell}}_{\v^c}t^\ell\right)\left(\sum_{q'}e'(p',q')t^{q'}\right)=\sum_{q}e(p,q)t^q
\end{equation}
By inductive hypothesis, $e'(p',q')\neq0\Rightarrow (p',q')\in R(a_1',a_2')$. Considering the highest $t$-degree terms on both sides of \eqref{eq:22} (if $a_1\le bp$) or \eqref{eq:22'} (if $a_1\ge bp$), we see that
\begin{equation}\label{eq:maxq'q}
a_1-bp+\max\{q'\ | \ e'(p',q')\neq0\}=\max\{q\ | \ e(p,q)\neq0\}
\end{equation}

We claim that the convex hull of $\{(p,q)\ | \ e(p,q)\neq0\}$ is $R(a_1,a_2)=s_1T_n+s_2T_{n+1}$. 
Indeed, by inductive hypothesis, the convex hull of $\{(p',q')\ | \ e'(p',q')\neq0\}$ is $R(a'_1,a'_2)=s_2T_{3-n}+s_1T_{4-n}$, a polygon with vertices $(0,0)$ and $(a'_2,0)=(s_2d''_{3-n}+s_1d''_{4-n},0)$, $(s_1d''_{4-n},s_2d'_{3-n})$, $(0,a'_1)=(s_2d'_{3-n}+s_1d'_{4-n})$. Thus by \eqref{eq:maxq'q}, 
the convex hull of $\{(p,q)\ | \ e(p,q)\neq0\}$ has vertices $(0,0)$ and $(0,a_1), (a_2-s_1d''_{4-n}, a_1-b(a_2-s_1d''_{4-n})+s_2d'_{3-n})=(s_2d''_{n+1},s_1d'_n), (a_2,a_1-b(a_2-0)+a'_1)=(a_2,0)$; so the convex hull is exactly $s_1T_n+s_2T_{n+1}$. Another way to see it is to consider the map 
\begin{equation}\label{RR'}
\varphi_R:(x,y)\mapsto (a_2-x,y-(a_1-bp))
\end{equation} 
which sends the top bound of $R(a_1,a_2)$ to the top bound of $R(a'_1,a'_2)$.
\smallskip

Now we prove the theorem by considering two cases separately:

Case 1: if $a_1\ge bp$.  Fix $q$ in \eqref{eq:22'}, we get
$$e(p,q)=\sum_{\ell+q'=q}{{a_1-bp}\brack{\ell}}_{\v^c}e'(p',q').$$
Assume $(p,q)\in R(a_1,a_2)$, we shall show that $\deg e(p,q)=D(p,q)$. It suffices to show that 
$\deg{{a_1-bp}\brack{\ell}}_{\v^c}e'(p',q')\le D(p,q)$ and the equality holds exactly once. Indeed, using $a'_1=ba_2-a_1,a'_2=a_2, p'=a_2-p, q'=q-\ell$, we have
$$\aligned
&D(p,q)-\Big(\deg{{a_1-bp}\brack{\ell}}_{\v^c}\Big)e'(p',q')=D(p,q)-(a_1-bp-\ell)\ell c-D'(p',q')\\
&\quad=ca_1q+ba_2p-bp^2-bcpq-cq^2-(a_1-bp-\ell)\ell c\\
&\quad\quad -c(b_2-a_1)(q-\ell)-ba_2(a_2-p)+b(a_2-p)^2+bc(a_2-p)(q-\ell)+c(q-\ell)^2\\
&\quad=2c(q-\ell)(a_1-bp-\ell)\ge0
\endaligned
$$ 
since $\ell\le \min\{q,a_1-bp)$. Moreover, the equality can be obtained exactly once, when $\ell=\min\{q,a_1-bp\}$; for this we need to verify that the corresponding $(p',q')$ lies in $R(a'_1,a'_2)$: 

-- if $q\le a_1-bp$, then $\ell=q$, $q'=0$, in which case $(p',q')=(p',0)\in R(a'_1,a'_2)$ (note that $(p,q)\in R(a_1,a_2)$ guarantees $0\le p'\le a'_2$);

-- if $q\ge a_1-bp$, then $\ell=a_1-bp$, $(p',q')=(a_2-p,q-(a_1-bp))=\varphi(p,q)$ which is in $R(a'_1,a'_2)$. 

\smallskip

Case 2: if $a_1\le bp$. By \eqref{eq:22}, for any fixed  $q'$ we have
\begin{equation}\label{eq:ep'q'=sum epq}
e'(p',q')=\sum_{\ell+q=q'}{{-a_1+bp}\brack{\ell}}_{\v^c}e(p,q).
\end{equation}
By a computation similar to Case 1, when $\ell+q=q'$, we have
\begin{equation}\label{eq:D'-D}
D(p',q')-(bp-a_1-\ell)\ell c-D(p,q)=2c(q'-\ell)(bp-a_1-\ell).
\end{equation}
We will prove that $\deg e(p,q)=D(p,q)$ for all $q$ with $(p,q)\in R(a_1,a_2)$ by a downward induction.

As the base case, let $q_{\max}$ be the maximal $q$ satisfying  $(p,q)\in R(a_1,a_2)$ and let $\ell=bp-a_1, q'=\ell+q_{\max}$. Then $(p',q')=(a_2-p,\ell+q_{\max})=\varphi_R(p,q_{\max})$ lies in $R(a'_1,a'_2)$, thus $\deg e'(p',q')=D'(p',q')$. Since both $\ell$ and $q_{\max}$ are maximal, there is only one term in the right side of \eqref{eq:ep'q'=sum epq}, thus 
$\deg e(p,q_{\max})=(bp-a_1-\ell)\ell c+\deg e(p,q_{\max})=D'(p',q')=D(p,q_{\max})$.

For the inductive step, assume $\deg e(p,q^*)=D(p,q^*)$ for all $q^*>q$ with $(p,q^*)\in R(a_1,a_2)$. Recall that we already proved $e(p,q^*)=0$ for $(p,q^*)\notin R(a_1,a_2)$. Let $\ell=bp-a_1, q'=\ell+q$. Then $(p',q')=(a_2-p,\ell+q)=\varphi_R(p,q)$ lies in $R(a'_1,a'_2)$, thus $\deg e'(p',q')=D'(p',q')$. Comparing the degrees of the two side of \eqref{eq:ep'q'=sum epq}, we have 
$$D'(p',q')\le \max_{\stackrel{q\le q^*\le q_{\max}}{\ell^*+q^*=q'}}\{(bp-a_1-\ell^*)\ell^* c+\deg e(p,q^*)\}$$ 
However, for $q^*$ satisfying $q+1\le q^*\le q_{\max}$, we already know $\deg e(p,q^*)=D(p,q^*)$, so 
$$\aligned
&(bp-a_1-\ell^*)\ell^* c+\deg e(p,q^*)=(bp-a_1-\ell^*)\ell^* c+D(p,q^*)\\
&\stackrel{\text{by } \eqref{eq:D'-D}}{=}D'(p',q')-2c(q'-\ell^*)(bp-a_1-\ell^*)
=D'(p',q')-2cq^*(q^*-q)< D'(p',q').
\endaligned$$
So for $q^*=q$, $\ell^*=q'-q^*=bp-a_1=\ell$, we must have 
$$D'(p',q')=(bp-a_1-\ell^*)\ell^* c+\deg e(p,q^*)$$ 
It follows that
$$\deg e(p,q)=D'(p',q')-(bp-a_1-\ell^*)\ell^* c\stackrel{\text{by } \eqref{eq:D'-D}}{=}
D(p,q)+2c(q'-\ell^*)(bp-a_1-\ell^*)=D(p,q).$$

This completes the proof for the case $n\ge 3$.  The case $n\le-1$ can be proved similarly by considering $\sigma_1$ instead of $\sigma_2$.
\end{proof}

In the rest of the paper we focus on skew-symmetric rank-2 cluster algebra $\mathcal{A}_\v(r,r)$. 

\section{Nakajima's graded quiver varieties}

In this section we recall Nakajima's graded quiver varieties. For simplicity, we focus to the special case corresponding to rank-2 cluster algebras. 
The references for this section are \cite{Nakajima, Qin}. 
Many statements in this section are known to hold in a much more general setting; but with a focus on the special case, we can present an almost self-contained introduction by proving most of the geometric facts without referring to the aforementioned papers.

Denote nonnegative integer tuples $w=(w_1,w_1',w_2,w_2')\in\mathbb{Z}_{\ge0}^4$ and $v=(v_1,v_2)\in\mathbb{Z}_{\ge0}^2$.
Fix $\mathbb{C}$-vector spaces $W_1,W'_1,W_2,W'_2, V_1,V_2$ with dimensions 
$w_1,w'_1,w_2,w'_2,v_1,v_2$, respectively. Define
$$W=W_1\oplus W_1'\oplus W_2\oplus W_2', \quad V=V_1\oplus V_2.$$
Consider $(x_1,x_2,y_1,\dots,y_r)$, where $x_i:W_i\to W'_i$ ($i=1,2$), $y_h:W_2\to W'_1$ ($h=1,\dots,r$) are  linear maps.
\begin{equation}\label{WWWW}
\xymatrix{
W_2'&& W_2\ar[ll]_-{x_2} \ar@<-.5ex>[lld] \ar@<.5ex>[lld]^(.35){\phantom{XX} y_h\; (h=1,\dots,r)}\\
W_1'&& W_1\ar[ll]^-{x_1}\\
}
\end{equation}
Denote ${\bf x}=(x_1,x_2)$ and ${\bf y}=(y_1,\dots,y_r)$. 
Define two linear maps as below:
$$
\aligned
A=x_1+y_1+\cdots+y_r: &\; W_1\oplus W_2^{\oplus r}\to W_1', \quad (a_1,b_1,\dots,b_r)\mapsto x_1(a_1)+y_1(b_1)+\cdots+y_r(b_r)\\
B=x_2\oplus y_1\oplus \cdots\oplus y_r: &\; W_2\to W_2'\oplus (W_1')^{\oplus r}, \quad b\mapsto (x_2(b), y_1(b),\dots,y_r(b))
\endaligned
$$
Alternatively, we can express these maps by matrices. Fix bases for $W_i$ and $W'_i$ , and use \fbox{$x_1$}, or simply $x_1$ by abuse of notation, to denote the matrix representing the linear map $x_1$. Similar notations are used for $x_2$, $y_h$ ($h=1,\dots,r$). 
 Then the above two maps $A$ and $B$ are represented by the following block matrices of sizes $w_1'\times (w_1+rw_2)$ and $(w_2'+rw_1')\times w_2$, respectively (denote $p=({\bf x},{\bf y})$): 
\begin{equation}\label{AB}
\aligned
&
A=A_p=A({\bf x, y})=
[x_1,y_1,\dots,y_r]_\text{hor}=
\begin{bmatrix} 
\fbox{$x_1$}&\fbox{$y_1$}&\cdots&\fbox{$y_r$}
\end{bmatrix},\quad\\
&B=B_p=B({\bf x, y})=[x_2,y_1,\dots,y_r]_\text{vert}=
\begin{bmatrix} \,\fbox{$x_2$}\, \\ \fbox{$y_1$}\\ \vdots\\ \fbox{$y_r$}
\end{bmatrix}\\
\endaligned
\end{equation}
where we use the following notation for horizontally and vertically stacked matrices:
$$
[M_1,\dots,M_n]_\text{hor}=\begin{bmatrix}M_1 & \cdots & M_n\\
\end{bmatrix},\quad 
[M_1,\dots,M_n]_\text{vert}=\begin{bmatrix}M_1\\ \vdots\\ M_n\\
\end{bmatrix}$$
Below is the list of sizes of various matrices:

\begin{center}
 \begin{tabular}{|c|c|c|c|c|c|} 
 \hline
matrix & $x_1$ & $x_2$ & $y_h$ & $A$ & $B$\\
 \hline
size   & $w_1'\times w_1$ & $w_2'\times w_2$ & $w_1'\times w_2$ & $w'_1\times(w_1+rw_2)$ & 
$(w'_2+rw'_1)\times w_2$\\
 \hline
\end{tabular}
\end{center}
For an $m\times n$ matrix $A$, we denote 
$$A_{[I;J]}=[a_{ij}]_{i\in I, j\in J}=\text{ the submatrix of $A$ with row index set $I$ and column index set $J$},$$
denote
$$A_{[-;J]}=A_{[1,\dots,m;J]}, \quad A_{[I;-]}=A_{[I;1,\dots,n]}$$
%
For a $\mathbb{C}$-vector space $E$ of dimension $e$, denote by $Gr(d,E)$ or $Gr(d,e)$ the Grassmannian space parametrizing all $d$-dimensional linear subspaces of $E$.
\begin{definition}\cite[\S4]{Nakajima} 
Let $w=(w_1,w_1',w_2,w_2')\in\mathbb{Z}_{\ge0}^4$ and $v=(v_1,v_2)\in\mathbb{Z}_{\ge0}^2$.

(a) The variety $\mathbf{E}_w$ is defined as the space of quiver representations of the quiver in \eqref{WWWW}:
$$\mathbf{E}_w=\Hom(W_1, W_1') \oplus\Hom(W_2, W_2') \oplus \Hom(W_2,W_1')^{\oplus r}\cong \CC^{w_1w_1'+w_2w_2'+rw_1'w_2}.$$ 
We denote elements of $\mathbf{E}_w$ in the form $({\bf x}, {\bf y})=(x_1,x_2,y_1,\dots,y_r)$.

(b) The projective variety $\mathcal{F}_{v,w}$ is a subvariety of $Gr(v_1,W'_1)\times Gr(v_2,W'_2\oplus {W'_1}^{\oplus r})$ defined as
$$\mathcal{F}_{v,w}=\{(X_1,X_2)\; |\; \dim X_i=v_i\, (i=1,2), \quad X_1\subseteq W'_1,\quad   X_2\subseteq W_2'\oplus X_1^{\oplus r}\}.$$

(c) \emph{Nakajima's nonsingular graded quiver variety} $\tilde{\mathcal{F}}_{v,w}$ is given by
$$
\tilde{\mathcal{F}}_{v,w}
=\{({\bf x}, {\bf y},X_1,X_2)\;|\;  ({\bf x},{\bf y})\in \mathbf{E}_w, (X_1,X_2)\in\mathcal{F}_{v,w},\; 
 {\im} A({\bf x},{\bf y})\subseteq X_1,\;  {\im} B({\bf x},{\bf y})\subseteq X_2\}.
$$
where ${\im} A({\bf x},{\bf y})$ and ${\im} B({\bf x},{\bf y})$ mean the image of the corresponding linear maps.

(d) \emph{Nakajima's affine graded quiver variety} $\mathbf{E}_{v,w}$ is given by
$$
\aligned
\mathbf{E}_{v,w}
&=\{({\bf x},{\bf y})\in\mathbf{E}_w\, |\, 
\mathrm{rank}A({\bf x,y})\le v_1,\; \mathrm{rank} B({\bf x, y})\le v_2
\}.
\endaligned
$$
\end{definition}

\begin{remark}
For readers who are used to notations in Nakajima's paper \cite{Nakajima}, his notations are related to ours in the following way: $W_1(q^2)$,  $W_1(1)$, $W_2(q^3)$, ,  $W_2(q)$, $\mathcal{F}(v,W)$, $\tilde{\mathcal{F}}(v,W)$, $E_W$, $V_1(q)$, $V_2(q^2)$ are our $W_1$, $W_1'$, $W_2$, $W_2'$, $\mathcal{F}_{v,w}$, $\tilde{\mathcal{F}}_{v,w}$, $\mathbf{E}_w$, $V_1$, $V_2$, respectively. The Cartan matrix  $\mathbf{C}=(c_{ij})=\begin{bmatrix}2&-r\\-r&2\end{bmatrix}$. The formula \cite[(3.2)]{Nakajima} 
$C_q:(v_i(a))\mapsto (v'_i(a)), \text{ where } v'_i(a)=v_i(aq)+v_i(aq^{-1})+\sum_{j\neq i} c_{ij}v_j(a)$
gives 
$$v'_1(q^2)=v_1-rv_2,\; v'_1(1)=v_1,\;   v'_2(q^3)=v_2,\; v'_2(q)=v_2-rv_1.$$
In our notation, we write
\begin{equation}\label{C_qv}
C_qv=(v_1-rv_2,v_1,v_2,v_2-rv_1)
\end{equation}
The affine and graded quiver varieties are defined as algebro-geometric quotient and GIT quotient, respectively:

(a) Affine graded quiver variety
$\M_0^\bullet(V,W)=\{(\alpha_i,\beta_i,b_j)\}// \prod GL(V_i)$ 

(b) Nonsingular graded quiver variety
$\M^\bullet(V,W)=\{(\alpha_i,\beta_i,b_j)\}^s/ \prod GL(V_i)$ 

\noindent where $(\alpha_i,\beta_i,b_j)$ are linear maps with domain and codomain indicated in the following diagram:
\begin{equation}\label{figure:WV}
\xymatrix{
W_2'&V_2\ar[l]_-{\beta_2}\ar@<-.5ex>[d] \ar@<.5ex>[d]^{b_1,\dots,b_r}& W_2\ar[l]_-{\alpha_2}\\
W_1'&V_1\ar[l]_-{\beta_1}& W_1\ar[l]_-{\alpha_1}\\
}
\end{equation}
The relation between $(\alpha_i,\beta_i,b_j)$ and the maps $(x_1,x_2,y_h)$ in \eqref{WWWW} is given by
$x_1=\beta_1\alpha_1$, $x_2=\beta_2\alpha_2$,
$y_h=\beta_1b_h\alpha_2$ (for $h=1,\dots,r$).
\end{remark}

The following statement is proved by Nakajima \cite{Nakajima}. For the reader's convenience, we give a proof here.
\begin{lemma}\label{EF}
Let $w=(w_1,w_1',w_2,w_2')\in\mathbb{Z}_{\ge0}^4$ and $v=(v_1,v_2)\in\mathbb{Z}_{\ge0}^2$.

{\rm(a)} The variety $\myE_{v,w}\neq\emptyset$.   
The variety $\mathcal{F}_{v,w}$ (thus $\tilde{\mathcal{F}}_{v,w}$)  is nonempty if and only if 
\begin{equation}\label{nonempty F}
0\le v_1\le w_1' \textrm{ and }0\le v_2\le w_2'+rv_1.
\end{equation}

{\rm(b)}  If \eqref{nonempty F} holds, then  $\mytildeF_{v,w}$ is (the total space of) a vector bundle over $\mathcal{F}_{v,w}$ of rank  $w_1v_1+w_2v_2$. Meanwhile, the natural projection  
\begin{equation}\label{pi}
\pi: \tilde{\mathcal{F}}_{v,w}\to \mathbf{E}_{v,w},\quad ({\bf x},{\bf y},X_1,X_2) \mapsto ({\bf x},{\bf y})
\end{equation}
has the zero fiber $\pi^{-1}(0)\cong \mathcal{F}_{v,w}$.
\end{lemma}
\begin{proof}
(a) and the second statement of (b) are obvious. For the first statement of (b),  consider the natural projection 
$$\mytildeF_{v,w}\to \mathcal{F}_{v,w}, \quad ({\bf x},{\bf y},X_1,X_2)\mapsto(X_1,X_2).$$ We prove that it gives a vector bundle by identifying it with the pullback of a vector bundle. Indeed, for fixed $(X_1,X_2)$, a tuple $({\bf x},{\bf y},X_1,X_2)$ is in $\mytildeF_{v,w}$ if and only if the $w_1$ column vectors of the matrix $x_1$ are in $X_1$, and the $w_2$ column vectors of the matrix $[x_2,y_1,\dots,y_r]_\text{vert}$ are in $X_2$. (To show that these  conditions implies $\im A({\bf x},{\bf y})\subseteq X_1$, note that $X_2\subseteq W'_2\oplus X_1^{\oplus r}$, thus each ${\rm im}y_i$ is in $X_1$.)  So 
\begin{equation}\label{eq:tildeF is vector bundle}
\mytildeF_{v,w} \cong \iota^*\Big(S_{Gr(v_1,W'_1)}^{\oplus w_1}\bigoplus S_{Gr(v_2,W'_2\oplus {W'_1}^{\oplus r})}^{\oplus w_2}\Big)
\end{equation}
 where $\iota$ is the natural embedding
$\iota:\mathcal{F}_{v,w}\to Gr(v_1,W'_1)\times Gr(v_2,W'_2\oplus {W'_1}^{\oplus r})$,  and $S_{Gr(d,e)}$ stands for the tautological subbundle (i.e., the universal subbundle) on the Grassmannian variety $Gr(d,e)$. Since the pullback of a vector bundle is a vector bundle of the same rank, it follows from \eqref{eq:tildeF is vector bundle} that $\mytildeF_{v,w}$ is a vector bundle over $\mathcal{F}_{v,w}$ of rank $v_1w_1+v_2w_2$. (In particular, it is locally trivial.)
\end{proof}

Now we study a stratification of $\pi$ given in \eqref{pi}. 
Consider a stratification $\myE_w=\bigcup_{v} \myE^\circ_{v,w}$ where
\begin{equation}\label{Ust}
\quad \myE^\circ_{v,w} =\{({\bf x},{\bf y})\in\mathbf{E}_w\, |\, 
(\mathrm{rank}A({\bf x,y}), \mathrm{rank} B({\bf x, y}))=(v_1,v_2)=v \}.
\end{equation}

For $v'=(v'_1,v'_2)$ and $v=(v_1,v_2)$, recall that we denote $v'\le v$ if $v'_1\le v_1$ and $v'_2\le v_2$. (Note that in \cite{Nakajima} this is denoted oppositely as $v'\ge v$.)
Recall the following definition of $l$-dominant condition given in \cite{Nakajima}.

\begin{definition}\label{def:l-dominant}
Let $w=(w_1,w_1',w_2,w_2')\in\mathbb{Z}_{\ge0}^4$ and $v=(v_1,v_2)\in\mathbb{Z}_{\ge0}^2$. We say that
$(v,w)$ is $l$-\emph{dominant} if $w-C_qv\ge0$; or equivalently, by \eqref{C_qv}:
\begin{equation}\label{eq:l-dominant}
w_1-v_1+rv_2\ge0,\quad
w_1'-v_1\ge0,\quad
w_2-v_2\ge0,\quad 
w_2'-v_2+rv_1\ge0. 
\end{equation}
For a fixed $w$, we define
$$\D(w)=\{v\in\mathbb{Z}_{\ge0}^2 \ | \ (v,w) \textrm{ is $l$-dominant} \}.$$
We define 
$$\D=\{(v,w)\in\mathbb{Z}_{\ge0}^2\times\mathbb{Z}_{\ge0}^4 \ | \ (v,w) \textrm{ is $l$-dominant} \}.$$
 
\end{definition}
Note that the condition \eqref{eq:l-dominant} implies \eqref{nonempty F}.

In the following lemma and proposition, we define $\bar{v}$ with the property that $(\bar{v},w)$ is $l$-dominant and $\mathbf{E}_{v,w}=\mathbf{E}_{\bar{v},w}$ (which will be proved in Proposition \ref{prop:fiber}).

\begin{lemma}\label{lemma:vbar}
Let $w=(w_1,w_1',w_2,w_2')\in\mathbb{Z}_{\ge0}^4$, $v=(v_1,v_2)\in\mathbb{Z}_{\ge0}^2$,  $V=\{v'\in\mathbb{Z}_{\ge0}^2 \ | \ v'\le v \}$. 
Then the intersection $V\cap \D(w)=\{v'\in\mathbb{Z}_{\ge0}^2 \ | \ v'\le v \textrm{ and $(v,w)$  is $l$-dominant} \}$ has a unique maximal element $\bar{v}$ in the sense that every $v'\in V\cap \D(w)$ must satisfy $v'\le \bar{v}$. 
More explicitly, 
\begin{equation}\label{eq:barv}
\bar{v}=\big(\min(v_1,w'_1,w_1+rv_2,w_1+rw_2), \min(v_2,w_2,w'_2+rv_1,w'_2+rw'_1)\big)
\end{equation}
In particular, if $(v,w)$ satisfies \eqref{nonempty F}, then
$$\bar{v}=\big(\min(v_1,w_1+rv_2,w_1+rw_2), \min(v_2,w_2)\big);$$
if $v\in\D(w)$, then $\bar{v}=v$. 
\end{lemma}
\begin{proof}
It is obvious that if $v\in\D(w)$, then $\bar{v}=v$. So in the rest of the proof we assume $v\in\mathbb{Z}_{\ge0}^2\setminus \D(w)$.

We want to show that the $\bar{v}$ given in \eqref{eq:barv} is the unique maximal element in the intersection $V\cap \D(w)$.

We consider three cases, and only give a proof of Case 1, because the other two cases can be proved similarly. See Figure \ref{figure:l-dominant} for the corresponding figures. 

Case 1. $w_1-w_1'+rw_2\ge0$ and $w_2'-w_2+rw_1'\ge0$. 
In this case, \eqref{eq:barv} becomes
$$\bar{v}=
\begin{cases}
(w_1',w_2), &\textrm{ if $v_1\ge w_1'$ and $v_2\ge w_2$};\\ 
\big(v_1,\min(w_2'+rv_1,w_2)\big), &\textrm{ if $v_1\le w_1'$ and $v_2>\min(w_2'+rv_1,w_2)$};\\
\big(\min(w_1+rv_2,w_1'),v_2\big), &\textrm{ if $v_1> \min(w_1+rv_2,w_1')$ and $v_2\le w_2$}.\\
\end{cases}
$$ 

For simplicity, we assume that $w_1\le w'_1$ and $w'_2\le w_2$. (Other cases are similar). The set $\D(w)$ is shown in Figure \ref{figure:l-dominant} Left. The line $w_1-x+ry=0$ has slope $1/r$ whose $x$-intercept is $w_1\ge0$; the line $w'_2-y+rx=0$ has slope $r$ whose $y$-intercept is $w'_2$. The convex hull of $\D(w)$ is a (possibly degenerated) hexagon $H$ whose upper-right corner is $(w'_1,w_2)$. 

The three regions of $v$ are illustrated in Figure \ref{figure:l-dominant} Left by different filling patterns (diagonal stripes, vertical stripes, horizontal stripes), and there are overlaps on the boundary. It is clear from the figure that  each $v\in\mathbb{Z}_{\ge0}^2\setminus \D(w)$ belongs to at least one of the three regions.

If $v_1\ge w_1'$ and $v_2\ge w_2$,  then $v$ is weakly to the northeast of $(w'_1,w_2)$, which implies $\D(w)\subseteq V$, $V\cap \D(w)=\D(w)$, thus the unique maximal element in   $V\cap \D(w)$ is $(w'_1,w_2)$. 

If $v_1\le w_1'$ and $v_2>\min(w_2'+rv_1,w_2)$, then $v$ is above the hexagon $H$ and lies between the vertical lines $x=0$ and $x=w'_1$. In this case, $\bar{v}$ is the intersection of the vertical line $x=v_1$ with the upper boundary of $H$ (which consists of two sides), which is the point $(v_1,\min(w_2'+rv_1,w_2))$. 

If $v_1>\min(x_1+rv_2,w'_1)$ and $v_2<w_2$, then $v$ is to the right the hexagon $H$ and lies between the horizontal lines $y=0$ and $y=w_2$. In this case, $\bar{v}$ is the intersection of the horizontal line $y=v_2$ with the right boundary of $H$ (which consists of two sides), which is the point $(\min(w_1+rv_2,w'_1),v_2)$. 

This proves the formula \eqref{eq:barv} in this case.

The other two cases are proved similarly:

Case 2. If $w_1-w_1'+rw_2<0$:
$$\bar{v}=
\begin{cases}
(w_1+rw_2,w_2), &\textrm{ if $v_1\ge w_1+rw_2$ and $v_2\ge w_2$};\\ 
\big(v_1,\min(w_2'+rv_1,w_2)\big), &\textrm{ if  $v_2>\min(w_2'+rv_1,w_2)$ and $v_1\le w_1+rw_2$};\\
(w_1+rv_2,v_2), &\textrm{ if $v_1>w_1+rv_2$ and $v_2\le w_2$}.\\
\end{cases}
$$

Case 3. If $w_2'-w_2+rw_1'<0$:
$$\bar{v}=
\begin{cases}
(w_1',w_2'+rw_1'), &\textrm{ if $v_1\ge w_1'$ and $v_2\ge w_2'+rw_1'$};\\ 
(v_1,w_2'+rv_1), &\textrm{ if  $v_1\le w_1$ and $v_2>w_2'+rv_1$};\\
\big(\min(w_1',w_1+rv_2), v_2\big), &\textrm{ if $v_1>\min(w_1',w_1+rv_2)$ and $v_2\le w_2'+rw_1'$}.\\
\end{cases}
$$ 

Thus \eqref{eq:barv} holds for all cases.
\end{proof}

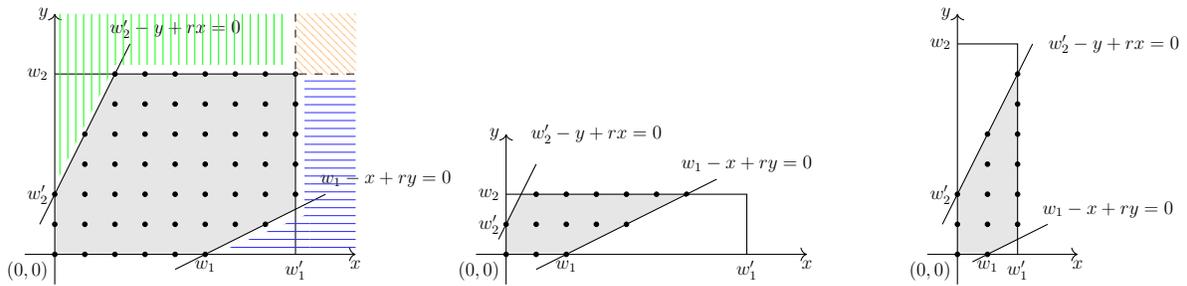
\begin{figure}[h]
\begin{tikzpicture}[scale=.4,every node/.style={scale=.6}]
\fill[gray!20] (0,0)--(5,0)--(8,1.5)--(8,6)--(2,6)--(0,2);
\begin{scope}
\clip (-.2,-.2)--(5,-.2)--(8.5,1.5)--(8.5,6.5)--(2,6.2)--(-.2,2.2);
\foreach \x in {0,...,8} {
  \foreach \y in {0,...,6} {
    \node[draw,circle, inner sep=1pt,fill] at (\x,\y) {};
  }
}
\end{scope}
\draw[->] (-1,0) -- (10,0);
\draw[->] (0,-1) -- (0,8);
\draw(8,0)--(8,6)--(0,6);
\draw(4,-.5)--(9,2);
\draw(-.5,1)--(2.5,7);
\node at (5,0) [below] {$w_1$};
\node at (8,0) [below] {$w'_1$};
\node at (10,0) [below] {$x$};
\node at (0,0) [below left] {$(0,0)$};
\node at (0,2) [left] {$w'_2$};
\node at (0,6) [left] {$w_2$};
\node at (0,8) [left] {$y$};
\node at (11,2.5) {$w_1-x+ry=0$};
\node at (4,7.5) {$w'_2-y+rx=0$};

\draw[dashed] (8,8) -- (8,6)--(10,6);
\fill [ fill opacity=.5,line width=1mm,pattern=north west lines, domain=0:1,pattern color=orange!80]
(8,8)--(8,6)--(10,6)--(10,8)--cycle;
\fill [fill opacity=.5,line width=1mm,pattern=vertical lines, domain=0:1,pattern color=green!90]
(0,2.3)--(2,6.3)--(7.7,6.3)--(7.7,8)--(0,8)--cycle;
\fill [fill opacity=.5,line width=1mm,pattern=horizontal lines, domain=0:1,pattern color=blue!80]
(5.3,0)--(8.3,1.3)--(8.3,5.8)--(10,5.8)--(10,0)--cycle;

\begin{scope}[shift={(15,0)}]
\fill[gray!20] (0,0)--(2,0)--(6,2)--(.5,2)--(0,1);
\begin{scope}
\clip (-.2,-.2)--(2.2,-.2)--(6.5,2.2)--(0.5,2.2)--(-.2,1.2);
\foreach \x in {0,...,8} {
  \foreach \y in {0,...,6} {
    \node[draw,circle, inner sep=1pt,fill] at (\x,\y) {};
  }
}
\end{scope}
\draw[->] (-1,0) -- (10,0);
\draw[->] (0,-1) -- (0,4);
\draw(8,0)--(8,2)--(0,2);
\draw(1,-.5)--(7,2.5);
\draw(-.25,0.5)--(1,3);
\node at (2,0) [below] {$w_1$};
\node at (8,0) [below] {$w'_1$};
\node at (10,0) [below] {$x$};
\node at (0,0) [below left] {$(0,0)$};
\node at (0,1) [left] {$w'_2$};
\node at (0,2) [left] {$w_2$};
\node at (0,4) [left] {$y$};
\node at (8,3) {$w_1-x+ry=0$};
\node at (3,4) {$w'_2-y+rx=0$};
\end{scope}

\begin{scope}[shift={(30,0)}]
\fill[gray!20] (0,0)--(1,0)--(2,.5)--(2,6)--(0,2);
\begin{scope}
\clip (-.2,-.2)--(1,-.2)--(2.2,0.5)--(2.2,6.2)--(1.8,6.2)--(-.2,2.2);
\foreach \x in {0,...,2} {
  \foreach \y in {0,...,6} {
    \node[draw,circle, inner sep=1pt,fill] at (\x,\y) {};
  }
}
\end{scope}
\draw[->] (-1,0) -- (4,0);
\draw[->] (0,-1) -- (0,8);
\draw(2,0)--(2,7)--(0,7);
\draw(.5,-.25)--(3,1);
\draw(-.5,1)--(2.5,7);
\node at (1,0) [below] {$w_1$};
\node at (2,0) [below] {$w'_1$};
\node at (4,0) [below] {$x$};
\node at (0,0) [below left] {$(0,0)$};
\node at (0,2) [left] {$w'_2$};
\node at (0,7) [left] {$w_2$};
\node at (0,8) [left] {$y$};
\node at (5,1.5) {$w_1-x+ry=0$};
\node at (5.2,7) {$w'_2-y+rx=0$};
\end{scope}
\end{tikzpicture}
\caption{$l$-dominant lattices for Cases 1 (Left), 2 (Middle), and 3 (Right).}
\label{figure:l-dominant} 
\end{figure}

\begin{lemma}\label{lem:ab}
Fix $v\in\mathbb{Z}_{\ge0}^2$ and $w\in\mathbb{Z}_{\ge0}^4$. Let
$\pi: \tilde{\mathcal{F}}_{v,w}\to \mathbf{E}_{v,w}$ be the map defined in \eqref{pi}.

{\rm(a)}  For each $v'\le v$,  $\myE^\circ_{v',w}$, defined in \eqref{Ust}, is nonempty if and only if  $(v',w)$ is $l$-dominant. In which case, it is nonsingular and locally closed in $\myE_w$ (so is also locally closed in $\myE_{v,w}$), and is irreducible and rational. 
Thus the variety $\mathbf{E}_{v,w}$ has a stratification 
$$\mathbf{E}_{v,w}=\bigcup_{v'} \myE^\circ_{v',w}\, ,\quad
\textrm{where $v'\le v$ and $(v',w)$ are $l$-dominant}.$$ 
 
{\rm(b)} For each $v'\le v$ with $(v',w)$ being $l$-dominant, the restricted projection $\pi^{-1}(\myE^\circ_{v',w})\to \myE^\circ_{v',w}$ is a Zariski locally trivial $\mathcal{M}$-bundle, 
where $\mathcal{M}$ itself is a 
$Gr(v_2-v'_2,w_2'+rv_1-v'_2)$-bundle over $Gr(v_1-v'_1,w_1'-v'_1)$ defined as
$$\mathcal{M}=\Big\{(X_1',X_2') \Big|
X_1'\in Gr(v_1-v_1',w_1'-v_1'), 
X_2'\in Gr\big(v_2-v_2',\mathbb{C}^{w_2'+rv_1'-v_2'}\oplus(X_1')^{\oplus r}\big)\Big\}.$$
In particular, $\mathcal{F}_{v,w}$ is a 
$Gr(v_2,w_2'+rv_1)$-bundle over $Gr(v_1,w_1')$. 
\end{lemma}

We postpone its technical proof to the end of the paper (\S11) so that it is easier to follow the flow of arguments of the main results. 

\begin{proposition}\label{prop:fiber}
Fix $v\in\mathbb{Z}_{\ge0}^2$ and $w\in\mathbb{Z}_{\ge0}^4$. Let
$\pi: \tilde{\mathcal{F}}_{v,w}\to \mathbf{E}_{v,w}$ be the map defined in \eqref{pi}.
Let $\bar{v}$ be defined as in Lemma \ref{lemma:vbar}. 

{\rm(a)} $\mathbf{E}_{v,w}=\mathbf{E}_{\bar{v},w}$ is irreducible. Moreover, $\mathbf{E}^\circ_{\bar{v},w}$ is its largest stratum; so  $\mathbf{E}_{v,w}=\overline{\mathbf{E}^\circ_{\bar{v},w}}$, the Zariski closure of $\mathbf{E}^\circ_{\bar{v},w}$.

{\rm(b)} Assume \eqref{nonempty F}. Then  $\tilde{\mathcal{F}}_{v,w}$ and $\mathcal{F}_{v,w}$ are irreducible and nonsingular and $\pi$ is surjective. Moreover,
$$\aligned
&\dim \mathcal{F}_{v,w}=d(v,w):= -v_1^2+rv_1v_2-v_2^2+v_1w_1'+v_2w_2'\\
&\dim \tilde{\mathcal{F}}_{v,w}=\tilde{d}(v,w):=\dim \mathcal{F}_{v,w}+w_1v_1+w_2v_2
=  -v_1^2+rv_1v_2-v_2^2+v_1(w_1+w_1')+v_2(w_2+w_2') \\  
&\dim \mathbf{E}_{v,w}
=\dim \mathbf{E}_{\bar{v},w}
=\tilde{d}(\bar{v},w)
=  -\bar{v}_1^2+r\bar{v}_1\bar{v}_2-\bar{v}_2^2+\bar{v}_1(w_1+w_1')+\bar{v}_2(w_2+w_2') \\  
\endaligned$$

{\rm(c)} Further assume that $(v,w)$ is $l$-dominant. Then $\pi$ is birational. In this case, it restricts to an isomorphism $\pi^{-1}({\bf E}^\circ_{v,w})\stackrel{\cong}{\longrightarrow} {\bf E}^\circ_{v,w}$. 

\end{proposition}

\begin{proof}
First, we prove the first statement of (b).
The special case $v'=(0,0)$ of Lemma \ref{lem:ab} (b) asserts that  $\mathcal{F}_{v,w}$ is a $Gr(v_2,w_2'+rv_1)$-bundle over $Gr(v_1,w_1')$ and therefore is irreducible and nonsingular. By Lemma \ref{EF} (b), $\tilde{\mathcal{F}}_{v,w}$ is a vector bundle over $\mathcal{F}_{v,w}$, so is also irreducible and nonsingular.
 
For any point in $\mathbf{E}_{v,w}$, it must be in a stratum $\mathbf{E}^\circ_{v',w}$ for some $v'\le v$ such that $v'\in\D(w)$ according to Lemma \ref{lem:ab} (a). By Lemma \ref{lem:ab} (b), the preimage is $\mathcal{M}$, which is nonempty if $0\le v_2-v'_2\le w'_2+rv_1-v'_2$ and $0\le v_1-v'_1\le w'_1-v'_1$; this condition is equivalent to \eqref{nonempty F}. Therefore $\pi$ is surjective since all fibers are nonempty. 

\medskip

(a) The equality $\mathbf{E}_{v,w}=\mathbf{E}_{\bar{v},w}$ follows from the stratifications of both sides, and the fact that and $l$-dominant $v'$ satisfying $v'\le v$ must also satisfy $v'\le \bar{v}$, according to Lemma \ref{lemma:vbar}. 

The projection  $\tilde{\mathcal{F}}_{\bar{v},w}\to \mathbf{E}_{\bar{v},w}$ is surjective and the source variety $\mytildeF_{\bar{v},w}$ is irreducible by the first statement of (b). Therefore the target variety $\mathbf{E}_{\bar{v},w}$ is also irreducible.

Since $\mathbf{E}_{\bar{v},w}$ is irreducible, it has a unique largest stratum. Note $\mathbf{E}_{\bar{v},w}=\bigcup_{v'\le\bar{v},v'\in\D(w)} \mathbf{E}^\circ_{v',w}$. For each $v'< v$ with $v'\in\D(w)$, we have $\mathbf{E}_{v',w}=\bigcup_{v''\le v', v''\in\D(w)} \mathbf{E}^\circ_{v'',w}$  
which is a proper subset of $\mathbf{E}_{v,w}$ because $\mathbf{E}^\circ_{\bar{v},w}\neq\emptyset$,
therefore 
$$\overline{\mathbf{E}^\circ_{v',w}}\subseteq \mathbf{E}_{v',w}\subsetneq \mathbf{E}_{v,w}$$
so $\mathbf{E}^\circ_{v',w}$ is not the largest stratum. Then $\mathbf{E}^\circ_{\bar{v},w}$ is the largest stratum.

\medskip

(c) It suffices to show that $\pi^{-1}({\bf E}^\circ_{v,w})\stackrel{\cong}{\longrightarrow} {\bf E}^\circ_{v,w}$ is bijective, since a bijective morphism between nonsingular complex algebraic varieties is an isomorphism. By Lemma \ref{lem:ab} (b), each fiber is a $Gr(v_2-v_2,w'_2+rv_1-v_2)$-bundle over $Gr(v_1-v_1,w'_1-v_1)$, which is a point since both Grassmannians are points. This implies the bijectivity.

\medskip

Finally, we prove the second statement of (b). Since $\myF_{v,w}$ is isomorphic to the fiber over $0\in \myE_{v,w}$,  it is a $Gr(v_2,w_2'+rv_1)$-bundle over 
$Gr(v_1,w_1')$ by Lemma \ref{lem:ab} (b). Therefore $\dim \myF_{v,w}=v_2(w_2'+rv_1-v_2)+v_1(w_1'-v_1)=
-v_1^2+rv_1v_2-v_2^2+v_1w_1'+v_2w_2'$. 

Then we obtain the formula of $\dim\mytildeF_{v,w}$ because $\mytildeF_{v,w}\to \myF_{v,w}$ is a vector bundle of rank $w_1v_1+w_2v_2$ as observed in Lemma \ref{EF} (b).  
  
The equality $\dim \mathbf{E}_{v,w}=\dim \mathbf{E}_{\bar{v},w}$ follows from (a). The equality 
$\dim \mathbf{E}_{\bar{v},w}=\dim \tilde{\mathcal{F}}_{\bar{v},w}=\tilde{d}(\bar{v},w)$
follows from (c). 
\end{proof}

\section{Decomposition theorem}

The following fact is proved by Nakajima in \cite[Theorem 14.3.2]{Nakajima_JAMS} and generalized by Qin in \cite[Theorem 5.2.1]{Qin}).
Nakajima's proof uses the representation theory of quantum affine algebras in an essential way, and focuses on type $ADE$; Qin's proof uses a similar idea of Nakajima's proof, particularly the analytic transversal slice theorem, but does not provide detailed explanations of why Nakajima's representation-theoretic argument still works in the more general setting. Instead, we will give an elementary proof in our special case, without referring to representation theory. 

\begin{theorem}\label{thm:trivial local system}
Assume ${\mytildeF_{v,w}}\neq\emptyset$. The local system appeared in the BBDG decomposition for  $\pi_*IC_{\mytildeF_{v,w}}$ are all trivial. Thus,
\begin{equation}\label{piICtildeF}
\pi_*(IC_{\mytildeF_{v,w}})=\bigoplus_{v'}\bigoplus_{d}a^d_{v,v';w}IC_{\mathbf{E}_{v',w}}[d]
\end{equation}
where $v'\le v$ satisfies the condition that $(v',w)$ is $l$-dominant, and $a^d_{v,v';w}\in\mathbb{Z}_{\ge0}$. 
\end{theorem}

First, recall the Beilinson--Bernstein--Deligne--Gabber decomposition theorem (also referred as BBDG or BBD decomposition Theorem). 
\begin{theorem}[BBDG Decomposition Theorem]\cite{BBD}\label{BBDthm}
Let $f:X\to Y$ be a proper algebraic morphism between complex algebraic varieties. 
Then there is a finite list of triples $(Y_a,L_a,n_a)$, where for each $a$, $Y_a$ is locally closed smooth irreducible algebraic subvariety of $Y$, $L_a$ is a semisimple local system on $Y_a$, $n_a$ is an integer,  such that:
\begin{equation}\label{eq:decomp thm}
f_* IC_X \cong \bigoplus_{a} IC_{\overline{Y_a}}(L_a)[n_a]
\end{equation}
Moreover, even though the isomorphism is not necessarily canonical, the direct summands appeared on the right hand side are canonical. 
\end{theorem}

In the next two subsections, we shall provide two proofs of Theorem \ref{thm:trivial local system}. The first proof is based on the algebraic version of the Transversal Slice Theorem (Lemma \ref{lem:transversal slice}). The second proof is longer, but only based on the (weaker) analytic version of the Transversal Slice Theorem proved by Nakajima \cite{Nakajima_JAMS}.

\subsection{Algebraic Transversal Slice Theorem}
Nakajima proves an analytic transversal slice theorem \cite[\S3]{Nakajima_JAMS},  and comments that the technique used there ``is based on a work of Sjamaar-Lerman in the symplectic geometry'', so the transversal slice is analytic in nature and is not algebraic; he further comments that ``It is desirable to have a purely algebraic construction of a transversal slice''. We shall prove an Algebraic Transversal Slice Theorem for the varieties studied in this paper. In contrast, by ``Analytic Transversal Slice Theorem'' we mean a weaker statement than Lemma \ref{lem:transversal slice} where we replace ``Zariski open'' by ``open in analytic topology'', and replace the algebraic morphisms/isomorphisms by the analytic ones.

\begin{lemma}[Algebraic Transversal Slice Theorem]\label{lem:transversal slice}
Let $p\in \myE_{v,w}$ be a point in the stratum $\myE^\circ_{v^0,w}$. So $v^0\le v$. Define 
$$w^\perp=w-C_qv^0=(w_1-v_1^0+rv_2^0,w_1'-v_1^0,w_2-v_2^0,w_2'-v_2^0+rv_1^0),\quad
v^\perp=v-v^0=(v_1-v_1^0,v_2-v_2^0).$$ 
Then there exist Zariski open neighborhoods $U\subseteq\myE_{v,w}$ of $p$, $U^0\subseteq \myE^\circ_{v^0,w}$ of $p$, $U^\perp\subseteq\myE^\circ(v^\perp,w^\perp)$ of $0$, and isomorphisms $\varphi,\psi$ making the following diagram commute: 
$$\xymatrix{\mytildeF_{v,w} \ar@{}[r]|{\supseteq} &\pi^{-1}U\ar[d]^\pi\ar[r]^-\varphi_-\cong&U^0\times\pi^{-1}(U^\perp)\ar[d]^{1\times\pi}\ar@{}[r]|{\subseteq}&\myE^\circ_{v^0,w}\times \mytildeF_{v^\perp,w^\perp}\\
\mathbf{E}_{v,w}\ar@{}[r]|{\supseteq}&U\ar[r]^-\psi_-\cong & U^0\times U^\perp\ar@{}[r]|{\subseteq} & \myE^\circ_{v^0,w}\times \mathbf{E}_{v^\perp,w^\perp}}$$
Moreover, $\psi(p)=(p,0)$, the diagram is compatible with the stratifications, in the sense that 
$\psi({\bf E}^\circ_{u,w}\cap U)=U_0\times ({\bf E}^\circ_{u^\perp,w^\perp}\cap U^\perp)$ 
for each $u$ satisfying $v^0\le  u\le v$, where we denote $u^\perp=u-v^0=(u_1-v_1^0,u_2-v_2^0)$. 
\end{lemma}
Its proof is delayed to \S11.

\subsection{The first proof of trivial local systems}
In this subsection we give the first proof of Theorem \ref{thm:trivial local system}. Another proof is given in \S\ref{subsection:2nd proof}.

We need the following lemma describing the BBDG decomposition under pullpack. It will be needed in the proof of Theorem \ref{thm:trivial local system}.

\begin{lemma}\label{lem:VXX}
Given $f:X\to Y$, $Y_a,L_a, n_a$ and the decomposition in Theorem \ref{BBDthm}:
$$f_* IC_X \cong \bigoplus_{a} IC_{\overline{Y_a}}(L_a)[n_a]$$
 Let $V$ be a nonsingular variety. Consider the following Cartesian diagram
$$\xymatrix{ V\times X\ar[r]^-{p'}\ar[d]_{1\times f}& X\ar[d]^f\\ V\times Y\ar[r]^-p&Y }$$
where $p$ and $p'$ are the natural projections. Define the pullback $\tilde{L}_a=p^*L_a$, which is a semisimple local system on $V\times Y_a$. Then
$$(1\times f)_* IC_{V\times X} \cong \bigoplus_{a} IC_{V\times\overline{Y_a}}(\tilde{L}_a)[n_a]$$

\end{lemma}
\begin{proof}
Denote by $d$ the dimension of the fiber of $f$. Then $$(p')^*[d]IC_X=IC_{V\times X},\quad
p^*[d]IC_{\overline{Y_a}}(L_a)=IC_{V\times\overline{Y_a}}(\tilde{L}_a)$$ (for example, see \cite[Lemma 3.6.3 and Corollary 3.6.9]{Achar}).
By proper base change \cite[\S5.8]{deCM}, we have the isomorphism $p^*f_*IC_X\stackrel{\sim}{\to} (1\times f)_*(p')^*IC_X$. Thus,
$$\aligned
(1\times f)_* IC_{V\times X} 
&\cong (1\times f)_* (p')^*[d]IC_{X} 
\cong p^*[d](f_*IC_X)
\cong p^*[d](\bigoplus_{a} IC_{\overline{Y_a}}(L_a)[n_a])\\
&\cong \bigoplus_{a} p^*[d] IC_{\overline{Y_a}}(L_a)[n_a]
\cong \bigoplus_{a} IC_{V\times\overline{Y_a}}(\tilde{L}_a)[n_a]. 
\endaligned
$$
\end{proof}

\begin{proof}[The First Proof of Theorem \ref{thm:trivial local system}]
Let $IC_Z(L)[n]$ be a direct summand that appears in the decomposition of 
$\pi_*(IC_{\tilde{\mathcal{F}}_{v,w}})$. 
Assume a general point $p$ of $Z$ is in ${\bf E}^\circ_{v^0,w}$ (thus $Z\subseteq {\bf E}_{v^0,w}$), and we
adopt the notation in Lemma \ref{lem:transversal slice}. In particular, $U^0$ is a Zariski open subset of ${\bf E}^\circ_{v^0,w}$. 

By uniqueness of the BBDG Decomposition, the restriction $IC_Z(L)[n]|_U$ must coincide with a direct summand of the decomposition of $\pi_*(IC_{\pi^{-1}U})$. Compare with Lemma \ref{lem:VXX} by setting $V=U^0$, $X=\pi^{-1}(U^\perp)$, $Y=U^\perp$, we conclude that $IC_Z(L)[n]|_U\cong IC_{U^0\times\overline{Y_a}}(\tilde{L}_a)[n_a]$ for some $a$.   Thus $Z\cap U=U^0\times \overline{Y_a}$, $L\cong \tilde{L}_a$ on $Z\cap U$, and $n=n_a$. But  $Z\subseteq {\bf E}_{v^0,w}$ implies $Z\cap U\subseteq U^0\times\{0\}$. So we must have $Y_a=\{0\}$ and $Z\cap U=U^0\times\{0\}$, thus $Z={\bf E}_{v^0,w}$.  

Moreover, $L_a$ is the trivial local system $\mathbb{Q}_0$. So the local system $\tilde{L}_a$ on $U^0$ is also trivial since it is the pullback of $L_a$ under the map $U^0\times \{0\}\to \{0\}$. Then $L$ is trivial on $Z\cap U$, and we see that
$$IC_Z(L)[n]\cong IC_{{\bf E}_{v^0,w}}[n]$$
Let $v'=v^0$. Then $v' \le v$ as seen in Lemma \ref{lem:transversal slice}, and $(v',w)$ is $l$-dominant because 
${\bf E}^\circ_{v^0,w}$ is nonempty. 
\end{proof}

\begin{remark} A key fact used in the above proof is that a local system on an irreducible nonsingular variety $X$ is uniquely determined by its restriction on a Zariski open dense subset $U$. This is true because a local system is determined by the action of fundamental group on a stalk   \cite[Lemma 1.7.9]{Achar}, and the following lemma modified from \cite[Lemma 2.1.22]{Achar}. \end{remark}
\begin{lemma}\label{lem:surjective pi1}
Let $X$ be a smooth connected complex variety. For any Zariski open dense subset $U\subseteq X$ and any point $x_0\in U$, the natural map $\pi_1(U,x_0)\to \pi_1(X,x_0)$ is surjective. Moreover, the map is an isomorphism if $X\setminus U$ has complex codimension at least 2. 
\end{lemma}
\begin{proof}
The first statement is proved in \cite[Lemma 1.7.9]{Achar}. The second statement can be proved similarly, with the details given below. Recall the following fact.
\smallskip

\noindent {\bf Fact}(\cite[p146, Theorem 2.3]{Godbillon}): \emph{if $M$ is a connected real smooth manifold without boundary, and $N$ is a closed submanifold, $x$ be a point in $M\setminus N$, then $\pi_1(M\setminus N,x)\to \pi_1(M,x)$ is surjective if the real codimension of $N$ is at least 2, and the map is an isomorphism if the real codimension of $N$ is at least 3.}
\smallskip

\noindent Now assume $Z=X\setminus U$ has complex codimension at least 2. So its real codimension is at least 4.  Stratify $Z$ into locally closed smooth subvarieties $Z=\cup_{i=1}^k Z_i$, indexed in decreasing dimensions. Then $\cup_{i=1}^{j} Z_i$ is open in $Z$ for every $1\le j\le k$. Apply the Fact to $(M,N)=(U\cup(\cup_{i=1}^j Z_i),Z_j)$, we conclude that $\pi_1(U\cup(\cup_{i=1}^{j-1} Z_i),x_0)\to \pi_1(U\cup(\cup_{i=1}^j Z_i),x_0)$ is an isomorphism. 
Thus the composition
$$\pi_1(U,x_0)\to \pi_1(U\cup Z_1,x_0)\to \pi_1(U\cup(\cup_{i=1}^2 Z_i),x_0)\to\cdots\to \pi_1(U\cup(\cup_{i=1}^{k} Z_i),x_0)=\pi_1(X,x_0)$$
is an isomorphism.
\end{proof}

\section{Other desingularizations of $\mathbf{E}_{v,w}$}

In this section, we assume $(v,w)$ is $\ell$-dominant. 

\subsection{Variety $\tilde{\mathcal{F}}^{\rm swap}_{v,w}$}\label{subsection:Fswap}
Given $(v,w)$, define 
\begin{equation}\label{eq:swap}
v^{\rm swap}=(v_2,v_1),\ w^{\rm swap}=(w'_2,w_2,w_1',w_1), \   \tilde{\mathcal{F}}^{\rm swap}_{v,w}=\tilde{\mathcal{F}}_{v^{\rm swap},w^{\rm swap}}
\end{equation}
Note that $\iota:\mathbf{E}_{v,w}\to \mathbf{E}_{v^{\rm swap},w^{\rm swap}}$, induced by $(A,B)\mapsto (B^T,A^T)$ is an isomorphism. The composition
$$\pi^{\rm swap}:\ \tilde{\mathcal{F}}^{\rm swap}_{v,w}\to \mathbf{E}_{v^{\rm swap},w^{\rm swap}}\stackrel{\iota^{-1}}{\longrightarrow} \mathbf{E}_{v,w}$$
gives a desingularization of $\mathbf{E}_{v,w}$. Note that $\tilde{\mathcal{F}}^{\rm swap}_{v,w}$ is not isomorphic to  $\tilde{\mathcal{F}}_{v,w}$ in general.

\subsection{Varieties $\tilde{\mathcal{G}}_{v,w}$ and $\tilde{\mathcal{H}}_{v,w}$}\label{subsection:Gvw}

In this section we factor $\pi: \tilde{\mathcal{F}}_{v,w} \to \mathbf{E}_{v,w}$ through a possibly singular variety $\tilde{\mathcal{H}}_{v,w}$, and introduce a nonsingular variety $\tilde{\mathcal{G}}_{v,w}$ such that there is a map  $\tilde{\mathcal{G}}_{v,w}\to \tilde{\mathcal{H}}_{v,w}$. See the following diagram:
\begin{equation}\label{diagram:FGtoE}
\xymatrix{ 
\tilde{\mathcal{F}}_{v,w}\ar[rd]^{p_2}\ar[rdd]_{\pi} && \tilde{\mathcal{G}}_{v,w}\ar[ld]_{p_2'} \ar[ldd]^{\pi'}\\
& \tilde{\mathcal{H}}_{v,w}\ar[d]^{p_1}\\
& \myE_{v,w}\\
}
\end{equation}

Recall that a linear map $f: V\to W$ has a dual (or transpose, or adjoint) $f^T: W^*\to V^*$, $\phi\mapsto \phi\circ f$. If $f$ is represented by the matrix $A$ with respect to fixed bases of $V$ and $W$, then $f^T$ is represented by the transpose matrix $A^T$ with respect to the dual bases.

\begin{definition}\label{myH,myG}
Define the quasi-projective variety 
$$
\tilde{\mathcal{H}}_{v,w}
=\{({\bf x},{\bf y}, X_1)\;|\;  ({\bf x},{\bf y})\in \mathbf{E}_w,\ \dim X_1=v_1,\ X_1\subseteq W'_1,\  {\im} A({\bf x},{\bf y})\subseteq X_1,\  {\rm rank} B({\bf x},{\bf y})\le v_2\}
$$
and define a stratification $\{U'_{t}\subseteq \mytildeH_{v,w}\}_{0\le t\le \min(v_2,w_2)}$ as
 \begin{equation}\label{Us}
 U'_t=\{({\bf x},{\bf y},X_1)\in \tilde{\mathcal{H}}_{v,w} \; | \; {\rm rank} B({\bf x},{\bf y})=t\}.
 \end{equation}
Define the quasi-projective variety 
$$
\aligned
\tilde{\mathcal{G}}_{v,w}
=&\{({\bf x},{\bf y},X_1, X_2')\;|\; \dim X_1=v_1,\dim X_2'=v_2, \\
&{\im}A({\bf x},{\bf y})\subseteq X_1\subseteq W'_1, \; \im (B({\bf x},{\bf y})^T)\subseteq X_2'\subseteq W_2^* \}.
\endaligned
$$
\end{definition}
\begin{remark}
 The strata $U_t'$ is nonempty only if $0\le t\le \min(v_2,w_2)$, because $\rank(B)\le w_2$ (=the number of columns of $B$) and $\rank(B)\le v_2$ by the definition of $\tilde{\mathcal{H}}_{v,w}$.
\end{remark}


\begin{proposition}\label{maps between FGH}
(a) Let $v_1\le w'_1$, $v_2\le w_2$, $\myG_{v,w}=Gr(v_1,w'_1)\times Gr(v_2,w_2)$. The natural projection $$f: \tilde{\mathcal{G}}_{v,w}\to \mathcal{G}_{v,w},\quad ({\bf x},{\bf y},X_1,X'_2)\mapsto(X_1,X'_2)$$ gives a vector bundle of rank 
$w_1v_1+w'_2v_2+rv_1v_2$. So
$$\dim \mytildeG_{v,w}=\dim\myG_{v,w}+(w_1v_1+w'_2v_2+rv_1v_2)=v_1(w_1+w'_1-v_1)+v_2(w_2+w'_2+rv_1-v_2).$$
As a consequence, $\mytildeG_{v,w}$ is nonsingular. 

(b) We can factorize the natural projection $\pi: \tilde{\mathcal{F}}_{v,w} \to \mathbf{E}_{v,w}$ as follows:
$$
\aligned
 &\tilde{\mathcal{F}}_{v,w}\stackrel{p_2}{\longrightarrow} \tilde{\mathcal{H}}_{v,w} \stackrel{p_1}{\longrightarrow} \mathbf{E}_{v,w}\\
&({\bf x},{\bf y},X_1,X_2)\mapsto ({\bf x},{\bf y},X_1)\mapsto ({\bf x},{\bf y})\\
\endaligned
$$
Moreover,  $p_2$ is stratified by $\{U'_{t} \}$, and $p_1$ and $\pi$ are both stratified by $\{ \myE^\circ_{(s,t),w}\}$ with $((s,t),w)$ being $l$-dominant.
The fibers of $p_2$ over $U'_t$ are isomorphic to $Gr(v_2-t,w_2'+rv_1-t)$.
The fibers of $p_1$ over $U_{s,t}$ are isomorphic to $Gr(v_1-s,w_1'-s)$.

(c) We can factorize the natural projection $\pi': \tilde{\mathcal{G}}_{v,w} \to \mathbf{E}_{\bar{v},w}$ as follows:
$$
\aligned
 &\tilde{\mathcal{G}}_{v,w}\stackrel{p_2'}{\longrightarrow} \tilde{\mathcal{H}}_{v,w} \stackrel{p_1}{\longrightarrow} \mathbf{E}_{v,W}\\
&({\bf x},{\bf y},X_1,X_2')\mapsto ({\bf x},{\bf y},X_1)\mapsto ({\bf x},{\bf y})\\
\endaligned
$$
Moreover,  $p_2'$ is stratified by $\{U'_{t} \}$, and $\pi'$ is stratified by $\{U_{s,t}\}$.
The fibers of $p_2'$ over $U'_t$ are isomorphic to $Gr(v_2-t,w_2-t)$.

(d) $U'_t=\bigcup_s\  p_1^{-1}(U_{s,t})$ is nonsingular. 

\end{proposition}

\begin{proof}
(a) Proposition \ref{prop:fiber} proves that $\mytildeF_{v,w}$ is nonsingular, and we shall show that $\mytildeG_{v,w}$ is nonsingular by the same method. We can think of the ambient vector bundle $E$ to be trivial bundle of rank $(w_1w_1'+w_2w_2'+rw_1'w_2)$ over  $\myG_{v,w}$ where each fiber is the vector space of all possible $({\bf x},{\bf y})$ without restriction. Let $E_1$ be the subbundle of $E$ whose fiber over $(X_1,X'_2)$ is
$$\{({\bf x},{\bf y}) \ | \  \textrm{column vectors of  $x_1,  y_1,\dots, y_r$ are in } X_1\}$$
Let $E_2$ be the subbundle of $E$ whose fiber over $(X_1,X'_2)$ is
$$\{({\bf x},{\bf y}) \ | \  \textrm{row vectors of  $x_2,  y_1,\dots, y_r$ are in } X_2'\}$$
The fact that they are indeed (locally free) subbundles follows the isomorphisms 
$$E_1\cong  \pi_1^*S_{Gr(v_1, w'_1)}^{\oplus(w_1+rw_2)}\oplus \mathbb{C}^{w_2w'_2}, \quad E_2\cong  \pi_2^*S_{Gr(v_2,w_2)}^{\oplus(w'_2+rw'_1)}\oplus\mathbb{C}^{w_1w_1'}.$$
where $\pi_1,\pi_2$ are the natural projections from $\myG_{v,w}=Gr(v_1, w'_1)\times Gr(v_2,w_2)$ to the first and second factor, respectively, and  $\mathbb{C}^{w_2w'_2}$ denotes the trivial bundle whose fibers are the set of $x_2$,  $\mathbb{C}^{w_1w_1'}$ denotes the trivial bundle whose fibers are the set of $x_1$.
 It is well-known that the intersection of two subbundles is a subbundle, provided that the dimensions of all fibers of the intersection are equal. So in order to prove that the total space of 
 $\mytildeG_{v,w}=E_1\cap E_2$ is a subbundle, it suffices to show that the dimensions of the fibers of $E_1\cap E_2$ are a constant ${w_1v_1+w_2'{v}_2+rv_1{v}_2}$. 
Indeed, fix $(X_1,X_2')$. For any nonnegative integer $n$, denote $e_1, \dots, e_n$ the standard basis of $\mathbb{C}^n$. Since $\rank X_1=v_1$, there exists $P\in GL(w_1')$ such that $PX_1={\rm span}(e_1,\dots,e_{v_1})$. Similarly, there exists $Q\in GL(w_2)$ such that $QX_2'={\rm span}(e_1,\dots,e_{{v}_2})$.
Define 
$$\aligned
Y=\{(x_1',x_2',y_1',\dots,y_r')\; |\; &\im(x_1'+y_1'+\cdots+y_r')\subseteq {\rm span}(e_1,\dots,e_{v_1}),\\
&\im(x_2'^T+ y_1'^T+\cdots+ y_r'^T)\subseteq {\rm span}(e_1,\dots,e_{{v}_2})\;\}.
\endaligned$$
Then
$Y\cong \{x_1'\}\times \{x_2'\}\times \prod_{h=1}^r \{y_h'\}\cong \mathbb{C}^{w_1v_1}\times\mathbb{C}^{w_2'{v}_2}\times \prod_{h=1}^r\mathbb{C}^{v_1{v}_2}=\mathbb{C}^{w_1v_1+w_2'{v}_2+rv_1{v}_2}$. 
On the other hand, it is easy to see that the linear map
$$f^{-1}(X_1,X_2')\to Y,\quad (x_1,x_2,y_1,\dots,y_r)\mapsto (Px_1, x_2Q^T,Py_1Q^T,\dots,Py_rQ^T)$$
is a well-defined isomorphism. 
This shows that $f^{-1}(X_1,X_2')$ is a vector space of  dimension same as $\dim Y=w_1v_1+w'_2{v}_2+rv_1{v}_2$. 

Now since $\mytildeG_{v,w}$  is the total space of a vector bundle over a nonsingular variety $Gr(v_1, w'_1)\times Gr({v}_2,w_2)$, it must be nonsingular.

(b) It is easy to check $\pi=p_1\circ p_2$. 
The proof for the statement that $p_2$ is stratified by $\{U_t'\}$ is similar to Proposition \ref{prop:fiber}(b).  
To prove that  $p_2^{-1}(U'_t)\to U'_t$ is a locally trivial bundle,  we show that it is a pullback of another locally trivial  bundle $\mathcal{X}\to \mathcal{F}_{(v_1,t),w}$. Here is the diagram:
$$\xymatrix{p_2^{-1}(U_t') \ar[d]^{p_2}\ar[r]& \mathcal{X}\ar[d]^q\\ U'_t \ar[r]^f & \mathcal{F}_{(v_1,t),w} }$$
where 
$$\aligned
&\mathcal{X}=\{(X_1,X_2, Y_2) | (X_1,Y_2)\in \mathcal{F}_{(v_1,t),w}, Y_2\subseteq X_2\subseteq W'_2\oplus X_1^{\oplus r}, \dim X_2=v_2\}\\
& f: ({\bf x},{\bf y}, X_1)\mapsto (X_1,Y_2)=(X_1,\text{im}B({\bf x},{\bf y}))\\
& q: (X_1,X_2,Y_2) \mapsto (X_1,Y_2)
\endaligned
$$ 
We claim that $q$ gives a locally trivial fiber bundle with fibers isomorphic to $\mathcal{M}$. Indeed, $ \mathcal{F}_{(v_1,t),w}$ can be covered by open subsets over which the direct sum of universal subbundles $S_{Gr(v_1,W'_1)}\oplus S_{Gr(t,W'_2\oplus {W'_1}^{\oplus r})}$ is trivial, so we only need to discuss fiberwisely. For fixed $(X_1,Y_2)$, the set $(X_1,X_2, Y_2)$ satisfying $Y_2\subseteq X_2\subseteq W'_2\oplus X_1^{\oplus r}$ obviously form the variety $$Gr(\dim X_2-\dim Y_2, \dim(W'_2\oplus X_1^{\oplus r})-\dim Y_2)=Gr(v_2-t,w'_2+rv_1-t)$$
So $p_2$ is a locally trivial bundle with fibers isomorphic to $Gr(v_2-t,w'_2+rv_1-t)$.

Similarly,  by considering the diagram 
$$\xymatrix{p_1^{-1}(U_{s,t}) \ar[d]^{p_1}\ar[r]& \mathcal{X}\ar[d]^q\\ U_{s,t} \ar[r]^f &Gr(s,W'_1) }$$
where 
$$\aligned
&\mathcal{X}=\{(X_1, Y_1) | Y_1\in Gr(s,W'_1), Y_1\subseteq X_1\}\\
& f: ({\bf x},{\bf y})\mapsto Y_1=\text{im}A({\bf x},{\bf y})\\
& q: (X_1,Y_1) \mapsto X_1
\endaligned
$$ 
we conclude that $p_1$ is a locally trivial bundle with fibers isomorphic to $Gr(v_1-s,w_1'-s).$

(c) can be proved similarly to (b). The diagram is
$$\xymatrix{p_2'^{-1}(U'_t) \ar[d]^{p'_2}\ar[r]& \mathcal{X}\ar[d]^q\\ U'_t \ar[r]^-f & Gr(v_1,W'_1)\times Gr(t,W_2)}$$
where 
$$\aligned
&\mathcal{X}=\{(X_1,X'_2, Y'_2) | (X_1,Y'_2)\in Gr(v_1,W'_1)\times Gr(t,W_2), Y'_2\subseteq X'_2\subseteq W_2, \dim X'_2=v_2\}\\
& f: ({\bf x},{\bf y}, X_1)\mapsto (X_1,Y'_2)=(X_1,\text{im}(B({\bf x},{\bf y})^T))\\
& q: (X_1,X'_2,Y'_2) \mapsto (X_1,Y'_2)
\endaligned
$$ 

(d) To prove that $U'_t$ is nonsingular, we consider the following open covering
$$U'_t=\bigcup U'_J$$
where $J\subseteq\{1,\dots,w_2\}$, $|J|=t$, $U'_J\subseteq U'_T$ consists of those $({\bf x},\bf{y},X_1)$ such that the columns of $B({\bf x},{\bf y})$ indexed by $J$ are linearly independent. We shall show that $U'_J$ is nonsingular. Without loss of generality, we assume $J=\{1,\dots,t\}$. Similar to the proof of Proposition \ref{prop:fiber} (b), $U'_J$ is isomorphic to the open subset of 
$$U=\mathbb{A}^{w'_2t}\times \mathbb{A}^{(w_2-t)t}\times S^{\oplus(w_1+rt)}_{Gr(v_1,W'_1)}$$
whose elements are denoted $(Z'_2,N,x_1,C'_1,\dots,C'_r,X_1)$, where $X_1\in Gr(v_1,W'_1)$, and matrices: 

$Z'_2$ has size $w'_2\times t$,   

$N$ has size $t\times(w_2-t)$, 

$x_1$ has size $w'_1\times w_1$ and all column vectors are in $X_1$,  

$C'_i$ has size $w'_1\times t$ and all column vectors are in $X_1$,

\noindent The isomorphism is sending  $(Z'_2,N,x_1,C'_1,\dots,C'_r,X_1)$ to $(x_1,x_2,y_1,\dots,y_r,X_1)$ with
$$x_2=[Z'_2\;\; Z'_2N], \quad y_i=[C'_i\;\; C'_iN]$$
So
$$A=\begin{bmatrix}x_1&C'_1&C'_1N&\cdots&C'_r&C'_rN\end{bmatrix},
\quad
B=\begin{bmatrix}Z'_2& Z'_2N\\ C'_1&C'_1N \\ \vdots &\vdots\\ C'_r& C'_rN 
\end{bmatrix} 
$$
The open subset of $U'_J$ is defined by the condition that $[Z'_2,C'_1,\dots,C'_r]_\text{vert}$ must have full column rank $t$. 
\end{proof}

To summarize the notation and computations we introduced so far:
$$
\xymatrix{ 
{\mathcal{F}}_{v,w}&\tilde{\mathcal{F}}_{v,w}\ar[d]_{p_2}\ar@/^1.5pc/@{->}@[][dd]^{\pi}\ar[l]\\
U'_t \;\; \ar@{^{(}->}[r] & \tilde{\mathcal{H}}_{v,w}\ar[d]_{p_1}\\
U_{s,t}\; \ar@{^{(}->}[r] & \myE_{\bar{v},w}\\
}
\quad
\quad
\xymatrix{ 
{\mathcal{G}}_{v,w}&\tilde{\mathcal{G}}_{v,w}\ar[d]_{p'_2}\ar@/^1.5pc/@{->}@[][dd]^{\pi'}\ar[l]\\
U'_t \;\; \ar@{^{(}->}[r] & \tilde{\mathcal{H}}_{v,w}\ar[d]_{p_1}\\
U_{s,t}\; \ar@{^{(}->}[r] & \myE_{\bar{v},w}\\
}
$$

{\tiny
\begin{center}
 \begin{tabular}{|c|c|} 
 \hline
 space & dimension \\ 
 \hline
 $\mathcal{F}_{v,w}$ & $d_{v,w}=v_2(w_2'+rv_1-v_2)+v_1(w_1'-v_1)$  \\ 
 $\tilde{\mathcal{F}}_{v,w}$ & $\tilde{d}_{v,w}=v_2(w_2+w_2'+rv_1-v_2)+v_1(w_1+w_1'-v_1)$\\
 $\mathbf{E}_{v,w}$ & $\bar{v}_2(w_2+w_2'+r\bar{v}_1-\bar{v}_2)+\bar{v}_1(w_1+w_1'-\bar{v}_1)$\\
$\myG_{v,w}$ & $v_1(w'_1-v_1)+v_2(w_2-v_2)$\\
$\mytildeG_{v,w}$ & $v_1(w_1+w'_1-v_1)+v_2(w_2+w'_2+rv_1-v_2)$\\
  \hline
\end{tabular}
\;\; 
 \begin{tabular}{|c|c|} 
 \hline
 map & fiber \\ 
 \hline
 $\mytildeF_{v,w}\to\mathcal{F}_{v,w}$ & $\mathbb{C}^{w_1v_1+w_2v_2}$  \\ 
 $\mytildeG_{v,w}\to\mathcal{G}_{v,w}$ & $\mathbb{C}^{w_1v_1+w_2'v_2+rv_1v_2}$  \\ 
  $p_2|_{U'_t}$ & $Gr(v_2-t,w'_2+rv_1-t)$  \\ 
  $p'_2|_{U'_t}$ & $Gr(v_2-t,w_2-t)$  \\ 
  $p_1|_{U_{s,t}}$ & $Gr(v_1-s,w'_1-s)$  \\ 
  \hline
\end{tabular}
\end{center}
}

\section{Definition of dual canonical basis elements $L(w)$}
The identification between dual canonical basis and triangular basis is explained to us by Fan Qin, using results in \cite{Qin, KQ}. We will give a proof as self-contained as possible after introducing the dual canonical basis.

\subsection{Definition of $L(w)$ and $\mathbf{R}^\mathrm{finite}_t$}

First, recall some definitions from \cite{Nakajima,Qin}: Define 
$\mathcal{D}(\mathbf{E}_{w})$ to be the bounded derived category of constructible sheaves of $\mathbb{Q}$-vector spaces on $\mathbf{E}_{w}$. (The two references ) Define a set
$$\mathcal{P}_w=\{IC_w(v)=IC_{\mathbf{E}_{v,w}} \ | \  v\in\D(w)\}$$
Note that this is a finite set because $v$ must satisfy the condition \eqref{eq:l-dominant}.

Define a full subcategory $\mathcal{Q}_w$ of $\mathcal{D}(\mathbf{E}_{w})$ whose objects are
 finite direct sums of $IC_w(v)[k]$  for various $v\in\D(w)$, $k\in\mathbb{Z}$.
  
Define an abelian group $\mathcal{K}_w$ to be generated by isomorphism classes $(L)$ of objects of $\mathcal{Q}_w$ and quotient by the relations $(L)+(L')=(L'')$ whenever $L\oplus L'\cong L''$. By abuse of notation we denote $(L)$ simply as $L$, as done in \cite{Nakajima}. We can view $\mathcal{K}_w$ as a free $\mathbb{Z}[t^\pm]$-module with a basis $\mathcal{P}_w$, by defining $t^iL=L[i]$ for $i\in\mathbb{Z}$. 
The duality on $\mathcal{D}(\mathbf{E}_{w})$ induces the bar involution on $\mathcal{K}_w$ satisfying
$\overline{tL}=t^{-1}\overline{L},\; \overline{IC_w(v)}=IC_w(v)$. By the BBDG decomposition and Theorem \ref{thm:trivial local system}, for arbitrary $v,w$, 
\begin{equation}\label{pi=sum aIC}
\pi_w(v)=\pi_*(IC_{\widetilde{\mathcal{F}}_{v,w}})
=\sum_{v'\in \D(w)}\sum_{d}a^d_{v,v';w}IC_w(v')[d]
=\sum_{v'\in \D(w)}a_{v,v';w}IC_w(v')
\end{equation}
where $a_{v,v';w}=\sum_d a^d_{v,v';w}t^d$. It follows from the (Algebraic or Analytic) Transversal Slice Theorem that
\begin{equation}\label{eq:a=a}
a_{v,v';w}=a_{v^\perp,v'^\perp;w^\perp} \textrm{ for any $v^0\in\D(w)$}.
\end{equation} 
Note that  $\{ \pi_w(v)\ | \  v\in\D(w)\}$ is also a  $\mathbb{Z}[t^\pm]$-basis for $\mathcal{K}_w$, and for $(v,w)\in \D$ we have $a_{v,v;w}=1$.

Define the dual $\mathcal{K}_w^*={\rm Hom}_{\mathbb{Z}[t^\pm]}(\mathcal{K}_w,\mathbb{Z}[t^\pm])$, which is a free  $\mathbb{Z}[t^\pm]$-module with a basis
$$\{L_w(v)\ |\ v\in\D(w)\}=\textrm{the dual basis to }\mathcal{P}_w,$$
that is, 
$\langle L_w(v), IC_w(v')\rangle = \delta_{v,v'}$. The pairing $\langle-,-\rangle:\mathcal{K}_w^*\times \mathcal{K}_w\to \mathbb{Z}[t^\pm]$ satisfies the condition 
$$\langle L,C[1]\rangle=\langle L[-1],C\rangle=t\langle L,C\rangle, \quad \textrm{ for all } L\in \mathcal{K}_w^*, C\in \mathcal{K}_w.$$

Note that $L_w(v)$ satisfies the property that
$$\langle L_w(v), IC_w(v')\rangle = \delta_{v,v'}=\delta_{v^\perp,v'^\perp}=\langle f_{w^\perp}(v^\perp),IC_{w^\perp}(v'^{\perp})\rangle  \textrm{ for any $v^0\le v$},$$
where $v^0$ is used to define $w^\perp=w-C_qv^0, v^\perp=v-v^0, v'^\perp=v'-v^0$. 
Define
$$\mathbf{R}_t=\{(f_w)\in \prod_w \mathcal{K}_w^* \ | \ \langle f_w,IC_w(v)\rangle=\langle f_{w^\perp}, IC_{w^\perp}(v^\perp)\rangle \textrm{ whenever }v\in\D(w), v^0\le v\}.$$
For $(f_w)\in \prod_w \mathcal{K}_w^*$,
 write $f_w=\sum_{v\in \D(w)} c_{wv}L_w(v)$ where $c_{wv}\in\mathbb{Z}[t^\pm]$. Then $c_{wv}=\langle f_w,IC_w(v)\rangle$, and 
$$ (f_w)\in \mathbf{R}_t \textrm{ if and only if }c_{wv}=c_{w-C_qv,0} \textrm{ for every $l$-dominant pair $(w,v)$}.$$
So an element in $\mathbf{R}_t$ is uniquely determined by $\{c_{w0}\}_w$.

Define $L(w)\in\mathbf{R}_t$ to be induced by $L_w(0)$,
that is, for any $(v',w')\in\D$, 
$$\langle L(w),IC_{w'}(v')\rangle
=\delta_{w,w'-C_qv'}
$$ 

Define $\mathbf{R}^\mathrm{finite}_t$ to be the $\mathbb{Z}[t^\pm]$-submodule of $\mathbf{R}_t$ spanned by $\{L(w)\}_w$:
\footnote{Note that $\{L(w)\}$ is not a basis for $\mathbf{R}_t$: for example, let all $c_{w0}=1$, that is, $f_w=\sum_{v\in\D(w)} L_w(v)$; this gives an element in $\mathbf{R}_t$ corresponding to the infinite sum $\sum_w L(w)$, which is not in the span of $\{L(w)\}$. Nevertheless, each element in $\mathbf{R}_t$ can be uniquely written as a possibly infinite linear combination of $\{L(w)\}$. 

We can think of the projective limit $\lim\limits_{\leftarrow} \mathcal{K}_w$ which identify $IC_w(v)$ with $IC_{w^\perp}(v^\perp)$ 
whenever $v\in\D(w), v^0\le v$. Then $\mathbf{R}_t$ is the dual of $\lim\limits_{\leftarrow} \mathcal{K}_w$, and $\mathbf{R}^\mathrm{finite}_t$ consists of those linear functionals induced from ones defined on finite rank. 
}{}
$$\aligned
\mathbf{R}^\mathrm{finite}_t
&=\textrm{ the free $\mathbb{Z}[t^{\pm}]$-module with the basis $\{L(w)\}_w$ }\\
&=\{(f_w)\in \mathbf{R}_t \ | \ \langle f_w, IC_w(0)\rangle=0  \textrm{ for all but finitely many $w$} \} \\
\endaligned
$$

\subsection{Definition of $M(w)$}

Next, define $\{M_w(v)\ | \ v\in\D(w)\}\in \mathcal{K}_w^*$ to be the functional 
$$(L)\mapsto \sum_k t^{\dim {\bf E}^\circ_{v,w}-k} \dim H^k(i^!_{x_{v,w}}L)$$
where $x_{v,w}$ is a point in ${\bf E}^\circ_{v,w}$, and $i_{x_{v,w}}: x_{v,w}\to {\bf E}_{w}$ is the natural embedding. 

Define $M(w)\in \mathbf{R}_t$ to be induced by $M_w(0)$, that is, for any $(v',w')\in \D$, 
{\small
\begin{equation}\label{MIC}
\aligned
&\langle M(w),IC_{w'}(v')\rangle
\\
&=\begin{cases}
&\langle M_w(0),IC_{w}(v)\rangle
=\sum\limits_k t^{-k} \dim H^k(i^!_0IC_w(v)), 
\textrm{ if $\exists v\in\mathbb{Z}_{\ge0}^2: w-w'=C_q(v-v')$;}\\
&0, \textrm{ otherwise}.\\
\end{cases}
\endaligned
\end{equation}
}
where $i_0: \{0\}\to {\bf E}_w$ is the natural embedding. 
We claim that $M(w)\in \mathbf{R}^\mathrm{finite}_t$. Write $M(w)=\sum_{w''} b_{ww''}L(w'')$, then for any $(w',v')\in\D$: 
$$\langle M(w),IC_{w'}(v')\rangle =\sum_{w''} b_{ww''}\delta_{w'',w'-C_qv'}=b_{w,w'-C_qv'}$$
Together with \eqref{MIC}, we have 
$$b_{ww''}=\begin{cases}
& b_{w,w-C_qv}=\sum\limits_k t^{-k} \dim H^k(i^!_{0}IC_w(v)), \textrm{ if $w''=w-C_qv$ for some $v\in\mathbb{Z}_{\ge0}^2$}; \\
& 0, \textrm{ otherwise}.\\
\end{cases}$$

As a consequence, $b_{w,w-C_qv}\neq0 \ \Rightarrow \  w-C_qv\ge0 \ \Rightarrow \ 0\le v_1\le w'_1$ and $0\le v_2\le w_2$.
So for each fixed $w$, $b_{ww''}=0$ for all but finitely many $w''$, thus $M(w)\in \mathbf{R}^\mathrm{finite}_t$. 

We then claim that $\{M(w)\}$ is a basis of $\mathbf{R}^\mathrm{finite}_t$.  Indeed, denote a partial order  
$$\textrm{ $w'\le_{\rm w} w\Longleftrightarrow w-w'=C_qu$ for some $u\ge0$}$$ 
and denote $w'<_{\rm w}w$ if $w'\le_{\rm w}w$ and $w'\neq w$.
Consider the $\mathbb{Z}[t^{\pm}]$-submodule $V$ of  $\mathbf{R}^\mathrm{finite}_t$ spanned by $\{ L(w')\}_{w'\le_{\rm w} w}$. By the previous paragraph,  the transition matrix $B_w=[b_{w'w''}]$ (where $w',w''\le_{\rm w}w$)  from $\{L(w')\}_{w'\le_{\rm w}w}$ to $\{M(w')\}_{w'\le_{\rm w}w}$ is a (finite) triangular matrix. The diagonal entries $b_{w'w'}=\sum_k t^{-k} \dim H^k(i^!_0IC_{w'}(0))=1$ since $IC_{w'}(0)=\mathbb{Q}_{\{0\}}$ and 
$$H^k(i^!_0IC_{w'}(0))=H^k_{\{0\}}({\bf E}_{w'}, \mathbb{Q}_{\{0\}})=H^k({\bf E}_{w'},\mathbb{Q}_{\{0\}})
=\begin{cases}
&\mathbb{Q}, \textrm{ if $k=0$};\\
&0, \textrm{ otherwise}. \\
\end{cases}$$
where the first equality uses the fact that if $f:X\to Y$ is a closed embedding then $H^k(X,f^!K)=H_X(Y,K)$ (\cite[4.1.12]{deCM}), the second equality uses the long exact sequence of cohomology with supports, \cite[III, Ex 2.3]{Hartshorne}. Thus the matrix $B_w$ is invertible (and we denote its inverse by $B_w^{-1}=[b'_{w'w''}]$) and $L(w)$ is in the span of $\{M(w')\}_{w'\le_{\rm w}w}$. So $\{M(w)\}$ is a basis of $\mathbf{R}^\mathrm{finite}_t$.

Moreover, by definition of the intersection cohomology complex $IC_w(v)=IC_{{\bf E}_{v,w}}(\mathbb{Q})$, if $v>(0,0)$ and $\dim H^k(i^!_0IC_w(v))\neq0$, then $k\ge1$ (\cite[p142, Lemma 3.3.11]{Achar}). So $b_{ww''}\in t^{-1}\mathbb{Z}[t^{-1}]$, and 
$$M(w)=\sum_{w''\le_{\rm w} w}b_{ww''}L(w'')\in L(w)+\sum_{w''<_{\rm w} w}t^{-1}\mathbb{Z}[t^{-1}]L(w''), $$ 
which implies
\begin{equation}\label{L=sum M}
L(w)=\sum_{w''\le_{\rm w} w}b'_{ww''}M(w'')\in M(w)+\sum_{w''<_{\rm w} w}t^{-1}\mathbb{Z}[t^{-1}]M(w'').
\end{equation}

\medskip

\subsection{Definition of $\chi$}
Define a map $\phi: \mathbb{Z}^4\mapsto\mathbb{Z}^2:\  (u_1,u_1',u_2,u'_2)\mapsto (u_1-u_1',u_2-u_2')$.
Define a map $\chi: \mathbf{R}^\mathrm{finite}_t\to \mathcal{T}$ by defining it on the basis $\{M(w)\}$ (note that we do not require $v\in\D(w)$):
\begin{equation}\label{chiMw}
\aligned
\chi(M(w))
&=\sum_{v\in\mathbb{Z}_{\ge0}^2}   \big(\langle M(w),\pi_w(v)\rangle_{t\to\v^r}\big)   X^{\phi(w-C_qv)}\\
&=\sum_{v\in\mathbb{Z}_{\ge0}^2}  \sum_k \v^{-kr}\dim H^k(i_0^! \pi_w(v))   X^{(w_1-w'_1+rv_2,w_2-w'_2-rv_1)}
\endaligned
\end{equation}
and extend it by the rules $\chi(L_1+L_2)=\chi(L_1)+\chi(L_2)$ and $\chi(tL)=\v^r\chi(L)$, 

We claim the following is true: (again, we do not require $v\in\D(w)$)
\begin{equation}\label{chiLw}
\chi(L(w))=\sum_{v\in\mathbb{Z}_{\ge0}^2} a_{v,0;w}(\v^r)X^{(w_1-w'_1+rv_2,w_2-w'_2-rv_1)}
\end{equation}
Indeed,
$$\aligned
\chi(L(w))
&=\sum_{w''\le_{\rm w} w}b'_{ww''}\chi(M(w''))
=\sum_{w''\le_{\rm w} w}b'_{ww''}\sum_{v''}   \big(\langle M(w''),\pi_{w''}(v'')\rangle_{t\to\v^r}\big)   X^{\phi(w''-C_qv'')}\\
&=\sum_{w''\le_{\rm w} w} \sum_{v'\in \D(w'')}\sum_{v''} (a_{v'',v';w''} b'_{ww''}b_{w'',w''-C_qv'})_{t\to\v^r} X^{\phi(w''-C_qv'')}\\
\endaligned
$$
Now change the variables: let $u\in\mathbb{Z}_{\ge0}^2$ satisfy $w=w''+C_qu$, and let $v=v''+u$, $u'=v'+u$. Then $w-C_q=w''-C_qv''$, $a_{v,u';w}=a_{v'',v';w''}$, and
$$
\chi(L(w))=\sum_{u\in \D(w)} \sum_{\stackrel{u'\in \D(w)}{ u'\ge u}}\sum_{v\ge u} (a_{v,u';w} b'_{w,w-C_qu}b_{w-C_qu,w-C_qu'})_{t\to\v^r} X^{\phi(w-C_qv)}
$$
In the above expression, all the conditions under the summation symbols can be dropped: if $v\not\ge u$, then $v\not\ge u'$, and thus $a_{v,u';w}=0$; if $u\notin \D(w)$ or $u'\notin \D(w)$ or $u'\not\ge u$, then $b_{w-C_qu,w-C_qu'}=0$. So we drop the conditions and swap the orders of the variables $u,u',v$ in the iterated sums:
$$\aligned
\chi(L(w))&=\sum_{u} \sum_{u'}\sum_{v} (a_{v,u';w} b'_{w,w-C_qu}b_{w-C_qu,w-C_qu'})_{t\to\v^r} X^{\phi(w-C_qv)}\\
&=\sum_{v} X^{\phi(w-C_qv)} \sum_{u'} (a_{v,u';w})_{t\to\v^r}  \sum_{u} (b'_{w,w-C_qu}b_{w-C_qu,w-C_qu'})_{t\to\v^r} 
\\
\endaligned
$$
Since if $w''$ is not of the form $w-C_qu$ then $b_{w-C_qu,w-C_qu'}=0$, the last summation is
$$\sum_{u} (b'_{w,w-C_qu}b_{w-C_qu,w-C_qu'})_{t\to\v^r}=\sum_{w''} (b'_{w,w''}b_{w'',w-C_qu'})_{t\to\v^r}=\delta_{w,w-C_qu'} $$
 thus
$
\chi(L(w))=\sum_{v} a_{v,0;w}(\v^r)X^{\phi(w-C_qv)}  
$, \eqref{chiLw} is proved.

\medskip

\subsection{$\{\overline{\chi(M(w))}\}$ and the standard monomial basis} 

Denote $P_t(X)=\sum_i\dim H^i(X,\mathbb{Q}) t^i$. Denote
$d_{v,w}=\dim \mathcal{F}_{v,w}$, $\tilde{d}_{v,w}=\dim \widetilde{\mathcal{F}}_{v,w}$. 
We shall prove an explicit formula for $\chi(M(w))$: 
\begin{equation}\label{chiMw}
\aligned
\chi(M(w))
&=\sum_{v\in\mathbb{Z}_{\ge0}^2}   \v^{r(d_{v,w}-\tilde{d}_{v,w})} {w'_2+rv_1 \brack v_2}_{\v^{r}} {w'_1\brack v_1}_{\v^{r}} X^{(w_1-w'_1+rv_2,w_2-w'_2-rv_1)}\\
&=\sum_{v\in\mathbb{Z}_{\ge0}^2}   \v^{-r \dim \widetilde{\mathcal{F}}_{v,w} }  P_{t}(\mathcal{F}_{v,w})  X^{(w_1-w'_1+rv_2,w_2-w'_2-rv_1)}\\
\endaligned
\end{equation}
Indeed, apply the base change isomorphism $i_{0}^!\pi_*=\pi'_*j^!$ for the following Cartesian square 
$$
\begin{tikzcd}
\mathcal{F}_{v,w}\ar[r, hook,"j"]\ar[d,"\pi' "'] & \widetilde{\mathcal{F}}_{v,w} \ar[d,"\pi"]\\
\{0\}\ar[r, hook,"i_{0}"] & \myE_{v,w} 
\end{tikzcd}
$$
we have 
$$i_{0}^!\pi_*(IC_{ \widetilde{\mathcal{F}}_{v,w} })
=\pi'_*j^!(\mathbb{Q}_{ \widetilde{\mathcal{F}}_{v,w} }[\tilde{d}_{v,w} ])
=\pi'_*j^*(\mathbb{Q}_{ \widetilde{\mathcal{F}}_{v,w} }[\tilde{d}_{v,w}-2(\tilde{d}_{v,w}-d_{v,w})])
=\pi'_*(\mathbb{Q}_{ {\mathcal{F}}_{v,w} }[2d_{v,w}-\tilde{d}_{v,w}]),
$$ 
where the second equality is because $j^!(C)=j^*[-2c](C) $ if $j$ is a closed embedding of pure codimension $c$ transversal to all strata of a stratification $\Sigma$ and the complex $C$ is $\Sigma$-constructible \cite[\S4.1.12]{deCM}, here we just take the trivial stratification for $ \widetilde{\mathcal{F}}_{v,w} $. 
Then
\begin{equation}\label{Mpi}
\aligned
&\langle M(w),\pi_{w}(v)\rangle=\sum_k t^{-k}H^k(i_{0}^!\pi_*(IC_{ \widetilde{\mathcal{F}}_{v,w} }))\\
&=\sum_k t^{-k}H^{k+2d-\tilde{d}_{v,w}}( \{0\}, \pi'_*(\mathbb{Q}_{ {\mathcal{F}}_{v,w} }))
=\sum_k\dim t^{-k}H^{k+2d_{v,w}-\tilde{d}_{v,w}}( { {\mathcal{F}}_{v,w} }, \mathbb{Q}) \\
&
=t^{2d_{v,w}-\tilde{d}_{v,w}}\sum_k\dim t^{-k}H^{k}( { {\mathcal{F}}_{v,w} }, \mathbb{Q})\\
&
\stackrel{(*)}{=}t^{2d_{v,w}-\tilde{d}_{v,w}}\Big(t^{-\dim Gr(v_2,w_2'+rv_1)}{w'_2+rv_1 \brack v_2}_{t}\Big) 
\Big(t^{-\dim Gr(v_1,w_1')}{w'_1\brack v_1}_{t}\Big)\\
&=t^{d_{v,w}-\tilde{d}_{v,w}}{w'_2+rv_1 \brack v_2}_{t}
{w'_1\brack v_1}_{t} \\
\endaligned
\end{equation}
where (*) is because ${\mathcal{F}}_{v,w}$  is a $Gr(v_2,w_2'+rv_1)$-bundle over $Gr(v_1,w_1')$, 
and in the last equality we use the fact that $\dim Gr(v_2,w_2'+rv_1)+\dim Gr(v_1,w_1')=\dim {\mathcal{F}}_{v,w}=d_{v,w}$. 
Also note that 
$$P_{t}(\mathcal{F}_{v,w})=t^{d_{v,w}}{w'_2+rv_1 \brack v_2}_{t}
{w'_1\brack v_1}_{t}, $$
thus we have proved \eqref{chiMw}.

Now we choose the initial seed to be $({\bf X}',B')=\mu_2({\bf X}, B)$. Thus $B'=\begin{bmatrix}0&r\\-r&0\end{bmatrix}$, the linear order is $2\vartriangleleft 1$, and the initial cluster is $(X_1,X_0)$.  The standard monomial basis elements with respect to this initial seed is
$$
  M'[{\bf b}] = \v^{b_1b_2} X_2^{[-b_2]_+} X_0^{[b_2]_+} X_1^{[b_1]_+} X_{-1}^{[-b_1]_+} 
$$
More generally, define
$$
  M^*[w] = \v^{(w_1-w'_1)(w'_2-w_2)} X_2^{w_2} X_0^{w_2'} X_1^{w_1} X_{-1}^{w'_1} 
$$

Each $w\in\mathbb{Z}_{\ge0}^4$ can be uniquely written as $w={}^fw+{}^\phi w$ where 
$${}^fw=\big(\min(w_1,w'_1),\min(w_1,w'_1), \min(w_2,w'_2),\min(w_2,w'_2) \big)$$
thus ${}^\phi w$ satisfies ${}^\phi w_1{}^\phi w'_1={}^\phi w_2 {}^\phi w'_2=0$. 

We have
\begin{equation}\label{M'=M'}
\textrm{ if $w={}^\phi w$ (that is, $w_1w'_1=w_2w'_2=0$), then  $M^*[w]=  M'[w_1-w'_1, w_2'-w_2]$. }
\end{equation} 

\begin{lemma}\label{chiMbar=M''}
For any $w\in\mathbb{Z}_{\ge0}^4$, we have 
$$M^*[w]=\overline{\chi(M(w))}=\sum_{v\in\mathbb{Z}_{\ge0}^2}   \v^{r(-d_{v,w}+\tilde{d}_{v,w})} {w'_2+rv_1 \brack v_2}_{\v^{r}} {w'_1\brack v_1}_{\v^{r}} X^{(w_1-w'_1+rv_2,w_2-w'_2-rv_1)}.$$ 
In particular, if $w={}^\phi w$, then $\overline{\chi(M(w))}$ is the standard basis element $M'[w_1-w'_1,w'_2-w_2]$ with the initial seed $\mu_2({\bf X},B)$.
\end{lemma}
\begin{proof}
Note that $X_0=X^{(0,-1)}+X^{(r,-1)}$, $X_{-1}=X^{(-1,0)}+\sum_i{r \brack i}_{\v^r} X^{(ri-1,-r)}=X_0^{(r,-1)}+X_0^{(0,-1)}$, $X_{m-1}^n=\sum {n\brack i}_{\v^r} X_m^{(ri,-n)}=\sum {n\brack i}_{\v^r}\v^{-rin} X_m^{ri}X_{m+1}^{-n}$. So
$$\aligned
&M^*[w]
=\v^{(w_1-w'_1)(w'_2-w_2)} X_2^{w_2} X_0^{w_2'} X_1^{w_1} \sum_{v_1}{w'_1\brack v_1} \v^{-rv_1w'_1}X_0^{rv_1}X_1^{-w'_1}\\
&=\sum_{v_1}\v^{(w_1-w'_1)(w'_2-w_2)+2rv_1w_1}X_2^{w_2}  {w'_1\brack v_1} \v^{-rv_1w'_1}X_0^{w'_2+rv_1}X_1^{w_1-w'_1}\\
&=\sum_{v_1, v_2}\v^{(w_1-w'_1)(w'_2-w_2)+2rv_1w_1-rv_1w'_1-rv_2(w'_2+rv_1)}   {w'_1\brack v_1} {w'_2+rv_1\brack v_2} X_2^{w_2}X_1^{rv_2}X_2^{-w'_2-rv_1}X_1^{w_1-w'_1}\\
\endaligned
$$ 
Since 
$$X_2^{w_2}X_1^{rv_2}X_2^{-w'_2-rv_1}X_1^{w_1-w'_1}=\v^{rv_2w_2+rv_2(w'_2+rv_1)+(w_2-w'_2-rv_1)(w_1-w'_1)}X^{(w_1-w'_1+rv_2,w_2-w'_2-rv_1)}$$
we have
$$\aligned
M^*[w]
&=\v^{r(v_1w_1+v_2w_2)}   \sum_{v_1, v_2}{w'_1\brack v_1} {w'_2+rv_1\brack v_2} X_2^{w_2}X_1^{rv_2}X_2^{-w'_2-rv_1}X_1^{w_1-w'_1}=\overline{\chi(M(w))}
\endaligned
$$ 
where the last equality is because $v_1w_1+v_2w_2=\tilde{d}_{v,w}-d_{v,w}$. 
\end{proof}

\begin{lemma}\label{M'' in M'}
$M^*[w]\in M'[w_1-w'_1,w'_2-w_2]+ \sum_{\bf b} \v\mathbb{Z}[\v]M'[{\bf b}]$.  
\end{lemma}
\begin{proof}
If $\min(w_2,w'_2)>0$, then 
$$\aligned
&M^*[w] = \v^{(w_1-w'_1)(w'_2-w_2)} X_2^{w_2-1}(X_2X_0) X_0^{w_2'-1} X_1^{w_1} X_{-1}^{w'_1}\\
&=\v^{(w_1-w'_1)(w'_2-w_2)} X_2^{w_2-1}(1+\v^r X_1^r) X_0^{w_2'-1} X_1^{w_1} X_{-1}^{w'_1}\\
&= \v^{(w_1-w'_1)(w'_2-w_2)} X_2^{w_2-1}X_0^{w_2'-1} X_1^{w_1} X_{-1}^{w'_1} +\v^{(w_1-w'_1)(w'_2-w_2)+r} X_2^{w_2-1}X_1^r X_0^{w_2'-1} X_1^{w_1} X_{-1}^{w'_1}\\
&=  M^*[w_1,w'_1,w_2-1,w'_2-1]\\
&\quad+\v^{(w_1-w'_1)(w'_2-w_2)+r+2r(w'_2-1)-(r+w_1-w'_1)(w'_2-w_2)}M^*[r+w_1,w'_1,w_2-1,w'_2-1]\\
&= M^*[w_1,w'_1,w_2-1,w'_2-1]+\v^{r(w_2+w'_2-1)}M^*[r+w_1,w'_1,w_2-1,w'_2-1] \\
&\in M^*[w_1,w'_1,w_2-1,w'_2-1]+\bigoplus_{w'} \v\mathbb{Z}[\v] M^*[w']\\
\endaligned
$$
Similarly, if $\min(w_1,w'_1)>0$, then 
$$\aligned
&M^*[w] = \v^{(w_1-w'_1)(w'_2-w_2)} X_2^{w_2}X_0^{w_2'} X_1^{w_1-1}(X_1X_{-1}) X_{-1}^{w'_1-1}\\
&= \v^{(w_1-w'_1)(w'_2-w_2)} X_2^{w_2}X_0^{w_2'} X_1^{w_1-1}(1+\v^rX_0^r) X_{-1}^{w'_1-1}\\
&= \v^{(w_1-w'_1)(w'_2-w_2)} X_2^{w_2}X_0^{w_2'} X_1^{w_1-1} X_{-1}^{w'_1-1}+\v^{(w_1-w'_1)(w'_2-w_2)+r} X_2^{w_2}X_0^{w_2'} X_1^{w_1-1} X_0^r X_{-1}^{w'_1-1}\\
&= M^*[w_1-1,w'_1-1,w_2,w'_2]+\v^{r(w_1+w'_1-1)}M^*[w_1-1,w'_1-1,w_2,w'_2+r] \\
&\in M^*[w_1-1,w'_1-1,w_2,w'_2] + \bigoplus_{w'} \v\mathbb{Z}[\v] M^*[w']\\
\endaligned
$$
Apply the above two formulas recursively, we obtain the result.
\end{proof}

\subsection{$\{\chi(L(w))\}$ and the triangular basis}
\begin{lemma}\label{chiL(w)=C}
For any $w\in\mathbb{Z}_{\ge0}^4$, we have ${\chi(L(w))}={C[w'_1-w_1, w_2'-w_2+r[w'_1-w_1]_+]}$, which is a triangular basis element. In particular, $\chi(L({}^\phi w))=\chi(L(w))$ (therefore $a_{v,0;w}=a_{v,0;{}^\phi w }$ for any $v\in\mathbb{Z}_{\ge0}^2$) and $\chi(L({}^f w))=1$. 
\end{lemma}
\begin{proof}
It follows from \eqref{chiLw} that $\chi(L(w))$ is bar-invariant. 
It follows from \eqref{L=sum M}, Lemma \ref{chiMbar=M''} and Lemma \ref{M'' in M'} that 
$$
\aligned
&\chi(L(w))-M'[w_1-w'_1,w_2'-w_2]\in \chi(L(w))-M^*[w]+ \sum_{\bf b} \v\mathbb{Z}[\v]M'[{\bf b}]\quad\text{ (by Lemma  \ref{M'' in M'}) }\\
&\quad =\chi(L(w))-\overline{\chi(M(w))}+ \sum_{\bf b} \v\mathbb{Z}[\v]M'[{\bf b}]\quad\text{ (by Lemma  \ref{chiMbar=M''}) }\\
&\quad \subseteq \overline{\chi(M(w))}+\sum\v\mathbb{Z}[\v]\overline{\chi(M(w''))}-\overline{\chi(M(w))}+ \sum_{\bf b} \v\mathbb{Z}[\v]M'[{\bf b}]\quad\text{ (by \eqref{L=sum M})}\\
&\quad = \sum\v\mathbb{Z}[\v] M^*[w'']+\sum_{\bf b} \v\mathbb{Z}[\v]M'[{\bf b}]\quad\text{ (by  Lemma \ref{chiMbar=M''})}\\
&\quad \subseteq \sum\v\mathbb{Z}[\v] (\sum \mathbb{Z}[\v] M'[{\bf b}'])+\sum_{\bf b} \v\mathbb{Z}[\v]M'[{\bf b}]
=\sum_{\bf b} \v\mathbb{Z}[\v]M'[{\bf b}]
\endaligned
$$ 

By the definition of triangular basis, $\chi(L(w))$ is a triangular basis element for the seed $({\bf X}',B')$. Since the triangular basis does not depend on the chosen acyclic seed (proved in \cite{BZ2}), $\chi(L(w))$ is also a triangular basis element for the seed $({\bf X},B)$. Its denominator vector is the same as the denominator vector ${\bf d}$ of $M'[w_1-w'_1,w'_2-w_2]=M^*[{}^\phi w]$. Since the denominator vectors of $X_2, X_0, X_1, X_{-1}$ are $(0,-1), ((0,1), (-1,0), (1,r)$ respectively, we denote ${\bf b}=(b_1,b_2)=(w_1-w'_1,w'_2-w_2)$ and get
${\bf d}=[-b_2]_+(0,-1)+[b_2]_+(0,1)+[b_1]_+(-1,0)+[-b_1]_+(1,r)=(-b_1,b_2+r[-b_1]_+)= (w'_1-w_1,w'_2-w_2+r[w'_1-w_1]_+)$. So $\chi(L(w))=C[ w'_1-w_1,w'_2-w_2+r[w'_1-w_1]_+]$. 
\end{proof}

\section{The proof of main results}
This section is devoted to the proof of the main theorem and its corollaries.

\subsection{Some facts on the BBDG Decomposition}
We have the following two lemmas.

\begin{lemma}\label{BBDG property}
Let $f: Y\to X$ be a proper morphism between complex algebraic varieties, $Y$ be nonsingular,  let $0$ be a point in $X$. 
Let $d=\dim Y$, $Y_0=f^{-1}(0)$ and $d_0=\dim Y_0$. Write the BBDG decomposition in the form
\begin{equation}\label{decomposition 0}f_*IC_Y=\bigoplus_b IC_{0}^{\oplus s_{0,b}}[b]\oplus\bigoplus_{V,L,b} IC_{V}^{\oplus s_{V,L,b}}(L)[b]
\end{equation}
where $V\neq 0$ are subvarieties of $X$,  each $L$ is a local system on an open dense subset of $V$, $s_{0,b}, s_{V,L,b}\in\mathbb{Z}_{\ge0}$ are multiplicities of the corresponding $IC$-sheaves. Then  $\{s_{0,b}\}$ satisfy the following conditions:

{\rm i)}  $s_{0,b}=s_{0,-b}$ for every $b\in\mathbb{Z}$.

{\rm ii)} $s_{0,b}\ge s_{0,b+2}$ for  every $b\in\mathbb{Z}_{\ge 0}$.

{\rm iii)} $s_{0,b}=0$ if $|b|>2d_0-d$. In particular, if $2d_0<d$, then $s_{0,b}=0$ for all $b$.
\end{lemma}

\begin{proof}
See \cite{deCM} for i) Poincare duality and ii) Relative Hard Lefschez. For iii),  let $b\ge 0$. From \eqref{decomposition 0} we get
$$ R^{d+b}f_*\mathbb{Q}_Y=R^bf_*\mathbb{Q}_Y[d]=\mathbb{Q}_0^{s_{0,-b}} \oplus \textrm{other terms}$$
Since the proper base change theorem implies $(R^{d+b}f_*\mathbb{Q}_Y)_0=H^{d+b}(Y_0)$, the above implies (we denote the cohomology to be $\mathbb{Q}$-coefficient)
$$H^{d+b}(Y_0)=\mathbb{Q}_0^{s_{0,-b}} \oplus \textrm{other terms}$$
If $d+b>2\dim Y_0=2d_0$, then $H^{d+b}(Y_0)=0$, and thus
its direct summand $\mathbb{Q}_0^{s_{0,-b}}$ must vanish. Thus $s_{0,b}=s_{0,-b}=0$ if $b>2d_0-d$.
\end{proof}

\begin{lemma}\label{max-degree IC}
Let $f: Y\to X$ be a birational proper morphism between complex algebraic varieties, $Y$ be nonsingular,   let $0$ be a point in $X$.  
Let $d=\dim Y=\dim X$, $Y_0=f^{-1}(0)$ and $d_0=\dim Y_0$. 
Then 
$$H^k(i_0^*IC_X)=0, \textrm{ if } k>2d_0-d$$
\end{lemma}
\begin{proof}
Since $f$ is birational, we can write the BBDG decomposition in the form
\begin{equation}\label{decomposition V}
f_*IC_Y=IC_{X}\oplus\bigoplus_{V,L,b} IC_{V}^{s_{V,L,b}}(L)[b]
\end{equation}
where $V \subsetneq X$ are proper subvarieties and $L$ are local system on an open dense subset of $V$.
Similar to the proof of Lemma \ref{BBDG property}, we have 
$$ H^{d+k}(Y_0)=H^k(i_0^*IC_X) \oplus \textrm{other terms}$$
If $k>2d_0-d$, then $d+k>2d_0$, thus the left side vanishes, thus the direct summand  
$H^k(i_0^*IC_X)$ of the right side also vanishes. 
\end{proof}

\subsection{$P_-(v,w)$}
We see in  \eqref{pi=sum aIC} that
$$
\pi_*(IC_{\widetilde{\mathcal{F}}_{v,w}})=\sum_{v'\in \D(w)}a_{v,v';w}IC_w(v')
$$
Apply $\sum_k t^{-k} \dim H^k(i^!_0(-))$ to both sides and then apply \eqref{Mpi} to the left side, we obtain
\begin{equation}\label{eq:tbb=atH}
t^{d_{v,w}-\tilde{d}_{v,w}}{w'_2+rv_1 \brack v_2}_{t}{w'_1\brack v_1}_{t}
=\sum_{v'\in \D(w)}a_{v,v';w}\sum_k t^{-k}\dim H^k(i^!_0 IC_w(v')).
\end{equation}

\begin{remark}
Note that  $H^k(i_0^! IC_w(v))=H^{-k}(i_0^* IC_w(v))$. 
Indeed, let $x$ be a point in $X$, \; $i:x\to X$ be the natural injective morphism. 
Let $\mathbb{D}_X$ be the Verdier duality functor (see \cite[\S2.8]{Achar}). 
Let $\mathcal{F}\in D_c^b(X)$ be self-dual: $\mathbb{D}_X(\mathcal{F})=\mathcal{F}$. (For example, 
the intersection cohomology complex $\mathcal{F}=IC_V\in D_c^b(X)$ supported on $V$ with a trivial local system satisfies the property; for a proof of this property, see \cite[Lemma 3.3.13]{Achar}).)
Then by \cite[(2.8.3)]{Achar},  $i^!\mathcal{F}=\mathbb{D}_X(i^*\mathbb{D}_X(\mathcal{F}))=\mathbb{D}_X(i^*\mathcal{F})$. From here we see
$H^k(i^!F)=H^{-k}(i^*F)$, thus 
$$
\sum_k t^{-k}\dim H^k(i^!\mathcal{F})=
\sum_k t^{k}\dim H^k(i^*\mathcal{F})
$$
\end{remark}

Based on this remark, we define a Laurent polynomial  
$$P_-(v,w)=\sum_k t^{k}\dim H^k(i_0^* IC_w(v))\in\mathbb{Z}[t^\pm]$$
and rewrite \eqref{eq:tbb=atH} as 
\begin{equation}\label{eq:sumAP-}
\sum_{v'\in \D(w)} a_{v,v';w} P_-(v',w)=t^{d_{v,w}-\tilde{d}_{v,w}}{w'_2+rv_1 \brack v_2}_{t}{w'_1\brack v_1}_{t}
\end{equation}
In the special case when $v\in\D(w)$, using the fact $a_{v,v;w}=1$ we get 
\begin{equation}\label{eq:variation of h=something-h}
P_-(v,w)=
t^{d_{v,w}-\tilde{d}_{v,w}}{w'_2+rv_1 \brack v_2}_{t}{w'_1\brack v_1}_{t}
-\sum_{v'\in \D(w)\setminus\{v\}} a_{v,v';w} P_-(v',w),\textrm{ if $v\in\D(w)$}.
\end{equation}

Recall 
$v^{\rm swap}=(v_2,v_1), w^{\rm swap}=(w'_2,w_2,w_1',w_1)$. Define
$$\aligned
&f(v,w)=2d(v,w)-\tilde{d}(v,w) = -v_1^2+rv_1v_2-v_2^2+v_1(w'_1-w_1)+v_2(w'_2-w_2)\\
&f^{\rm swap}(v,w)=f(v^{\rm swap},w^{\rm swap})=
-v_1^2+rv_1v_2-v_2^2+v_1(w_1-w'_1)+v_2(w_2-w'_2)\\
&g(v,w)=-v_1^2-rv_1v_2-v_2^2+v_1(w_1'-w_1)+v_2(w_2-w_2')\\
\endaligned
$$

\begin{lemma}\label{lem:upperbound for P}
Assume $v\in\D(w)$.

If $\tilde{d}(v,w)=0$, then $P_-(v,w)=1$. This happens if and only if $v=(0,0)$.

If $\tilde{d}(v,w)>0$, then the min-degree term of $P_-(v,w)$ is $t^{-\tilde{d}(v,w)}$ (with coefficient 1); the max-degree of $P_-(v,w)$ is $\le \min\big(-1,f(v,w),f^{swap}(v,w), g(v,w)\big)$. 

\end{lemma}
\begin{proof}
First note that $\bar{v}=v$ by Lemma \ref{lemma:vbar} and the assumption that $(v,w)$ is $l$-dominant. So $\dim \mathbf{E}_{v,w}=\tilde{d}(v,w)$ by Proposition \ref{prop:fiber}.

If $\tilde{d}(v,w)=0$, then  $\mathbf{E}_{v,w}=\{0\}$ is a point,
$P_-(v,w)=\sum_k t^{k}\dim H^k(i_0^*\mathbb{Q}_0)=1$. Solving $\tilde{d}(v,w)=0$, that is, $v_1(w'_1-v_1)+v_2(w_2+rv_1-v_2)=0$, and using the condition that $(v,w)$ is $l$-dominant, we must have $v=(0,0)$.

For the rest of the proof we assume $\tilde{d}(v,w)>0$.


For the min-degree of $P_-(v,w)$,  \cite[Exercise 3.10.2]{Achar}  asserts that for an irreducible variety $X$ of dimension $d$, $\tau^{\le-d} IC_X\cong \mathbb{Q}_X[d]$; thus $H^{k}(IC_X)=\tau^{\ge k}\tau^{\le k} IC_X=0$ for $k<-d$ and
$$H^{-d}(IC_X)[d]\cong\tau^{\ge-d}\tau^{\le-d} IC_X\cong \tau^{\ge-d}(\mathbb{Q}_X[d]) \cong \mathbb{Q}_X[d]$$
Thus the min-degree of $P_-(v,w)$ is $-\tilde{d}(v,w)$, and 
since $\mathbf{E}_{v,w}$ is irreducible, we get
$$\text{`` The coefficient of $t^{-\tilde{d}(v,w)}$ in $P(v,w)$ '' $ = \dim H^{-d(v,w)}(i_0^*(IC_w(v))=\dim (i_0^*\mathbb{Q}_{\mathbf{E}_{v,w}})=1$}.$$ 

For the max-degree of $P_-(v,w)$, it is $<0$ by the support condition of intersection cohomology complex \cite[\S3.1 (20)]{deCM}.

Applying Lemma \ref{max-degree IC} to the morphism $\pi: \tilde{\mathcal{F}}_{v,w}\to \mathbf{E}_{v,w}$,  we conclude the max-degree is no larger than $2d(v,w)-\tilde{d}(v,w)=f(v,w)$.

Applying Lemma \ref{max-degree IC} to the morphism $\pi^{\rm swap}: \tilde{\mathcal{F}}^{\rm swap}_{v,w}\to \mathbf{E}_{v,w}$  introduced in \S\ref{subsection:Fswap},  we conclude the max-degree is no larger than $2d(v^{\rm swap},w^{\rm swap})-\tilde{d}(v^{\rm swap},w^{\rm swap})=f^{\rm swap}(v,w)$. 

Applying Lemma \ref{max-degree IC} to the morphism  $\pi': \tilde{\mathcal{G}}_{v,w} \to \mathbf{E}_{{v},w}$ introduced in \S\ref{subsection:Gvw},  we conclude the max-degree is no larger than $2\dim\mathcal{G}_{v,w}-\dim\tilde{\mathcal{G}}(v,w)=2(v_1(w'_1-v_1)+v_2(w_2-v_2))-(v_1(w_1+w'_1-v_1)+v_2(w_2+w'_2+rv_1-v_2))=g(v,w)$. 
\end{proof}

\subsection{Proof of the main theorem (Theorem \ref{main theorem}) and its corollaries}
Recall that in \eqref{eq:C=e} we use $e(p,q)$ to denote the coefficients in the triangular basis element $C[a_1,a_2]=\sum e(p,q)X^{rp-a_1,rq-a_2}$.  
We can express $e(p,q)$ in terms of $a_{v,0;w}$ as follows.
\begin{lemma}\label{lem:e=a}
Fix $(a_1,a_2)\in\mathbb{Z}^2$, \  $(p,q)\in\mathbb{Z}_{\ge0}^2$. 
Assume that $w,v$ satisfy 
$$w'_1-w_1=a_1,\;  w'_2-w_2=a_2-r[a_1]_+, \; v=([a_1]_+-q,\, p)$$ 
Then
$$e(p,q)=a_{v,0;w}(\v^r).$$
Moreover, if $a_1\ge0$, then 
$$D(p,q)=rf(v,w)$$
\end{lemma}
\begin{proof}
By Lemma \ref{chiL(w)=C} and \eqref{chiLw}, 
$$\aligned
&C[a_1,a_2]=C[w'_1-w_1,w'_2-w_2+r[w'_1-w_1]_+]\\
&=\chi(L(w))=\sum_{v\in\mathbb{Z}_{\ge0}^2} a_{v,0;w}(\v^r)X^{(w_1-w'_1+rv_2,w_2-w'_2-rv_1)}
\endaligned
$$
Therefore,
$$
\sum_{p,q} e(p,q)X^{(rp-a_{1},rq-a_{2})}=\sum_{v_1,v_2} a_{v,0;w}(\v^r)X^{(rv_2-a_1,r([a_1]_+-v_1)-a_2)}
$$
Comparing the coefficients gives us the first statement.

To prove the second statement: for $a_1\ge0$, $D(p,q)=D(v_2,a_1-v_1)
=r(-v_1^2+rv_1v_2-v_2^2+a_1v_1+(a_2-ra_1)v_2)
=r(-v_1^2+rv_1v_2-v_2^2+v_1(w'_1-w_1)+v_2(w'_2-w_2))=rf(v,w)$. 
\end{proof}

\begin{lemma}\label{lem:v1>a1 then a=0}
If $v_1>[w'_1-w_1]_+$,  then $a_{v,0;w}=0$.
\end{lemma}
\begin{proof}
By Lemma \ref{lem:e=a}, $a_{v,0;w}(\v^r)=e(p,q)$ for $v=([a_1]_+-q,p)=([w'_1-w_1]_+-q,p)$. For $v_1>[w'_1-w_1]_+$, we have $q<0$, thus $e(p,q)=0$.  
\end{proof}

If $r\ge3$, 
we define $$r_{\rm small}=(r-\sqrt{r^2-4})/2,\quad r_{\rm big}=(r+\sqrt{r^2-4})/2$$ to be the two roots of $x^2-rx+1=0$. Note that $1/r<r_{\rm small}<r_{\rm big}<r$. 

\begin{lemma}\label{lem:f=0}
Let $r\ge3$, $w=(0,w'_1,w_2,0)$ where $r_{\rm small}<w_2/w'_1<r_{\rm big}$. Consider the curve $C_f=\{v\in\mathbb{R}_{\ge0}^2\ |\ v_1\le w'_1, f(v,w)=0, v\neq(w'_1,0)\}$. Then $C_f$ is concave upward and with two endpoints $(0,0)$ and $(w'_1,rw'_1-w_2)$. The slope of tangent of $C_f$ is increasing as $v_1$ increases from $0$ to $w'_1$; the slope at $(0,0)$ is $w'_1/w_2 \ (>r_{\rm small})$, and the slope   $(w'_1,rw'_1-w_2)$ is $r-\frac{w'_1}{rw'_1-w_2} \ (<r_{\rm big})$.
\end{lemma}
\begin{proof}
The equation $\{v\in\mathbb{R}^2\ | \ f(v,w)=0\}$ is a hyperbola. Let $x'=(-v_1+v_2)/\sqrt{2}$, $y'=(v_1+v_2)/\sqrt{2}$, $x=x'+\frac{w_2+w'_1}{(2+r)\sqrt{2}}$, $y=y'+\frac{w_2-w'_1}{(2-r)\sqrt{2}}$. We can rewrite the equation $f(v,w)=0$ as
$$(1+\frac{r}{2})x^2+(1-\frac{r}{2})y^2=\frac{1}{4}\bigg( \frac{(w_2+w'_1)^2}{r+2}-\frac{(w_2-w'_1)^2}{r-2}\bigg)$$
It is easy to see that the right side is positive since $r_{\rm small}<w_2/w'_1<r_{\rm big}$. Thus the hyperbola is of the shape in Figure \ref{fig:f=0}, whose northwest branch passes $(0,0)$ and $(w'_1,rw'_1-w_2)$. 

\begin{figure}[h]
\def\tikzscale{0.25}
\begin{tikzpicture}[scale=\tikzscale]

\tikzset{
    elli/.style args={#1:#2and#3}{
        draw,
        shape=ellipse,
        rotate=#1,
        minimum width=2*#2,
        minimum height=2*#3,
        outer sep=0pt,
    }
}

%
%
\newcommand\tikzhyperbola[6][thick]{%
    \draw [#1, rotate around={#2: (0, 0)}, shift=#3]
        plot [variable = \t, samples=1000, domain=-#6:#6] ({#4 / cos( \t )}, {#5 * tan( \t )});
    \draw [#1, rotate around={#2: (0, 0)}, shift=#3]
        plot [variable = \t, samples=1000, domain=-#6:#6] ({-#4 / cos( \t )}, {#5 * tan( \t )});
}

\def\angle{35}
\def\bigaxis{3cm}
\def\smallaxis{1.5cm}
\coordinate (center) at (-6, 2);

\pgfmathsetmacro\axisratio{\smallaxis / \bigaxis}

\def\lengthofasymptote{13}
\draw [color=black!40, line width = 0.4pt, rotate around={\angle + atan( \axisratio ): (center)}]
    ($ (-\lengthofasymptote, 0) + (center) $) -- ++(2*\lengthofasymptote, 0) ;
\draw [color=black!40, line width = 0.4pt, rotate around={\angle - atan( \axisratio ): (center)}]
    ($ (-\lengthofasymptote, 0) + (center) $) -- ++(2*\lengthofasymptote, 0) ;

\tikzhyperbola[line width = 1.2pt, color=blue]{90+\angle}{(center)}{\smallaxis}{\bigaxis}{76}

    \draw [thick] [->] (-18,1.5)--(8,1.5) node[right, below] {$v_1$};
    \draw [thick] [->] (-11.5,-10)--(-11.5,14) node[above, left] {$v_2$};
    \draw [dashed]  (-11.5,7)--(-3.8,7)--(-3.8,1.5);
    \node[below] at (-3.8,1.5) {$w'_1$};
    \node[left] at (-11.5,7) {$rw'_1-w_2$};
    \node[right] at (-2,6) {$f(v,w)\ge0$};
    \node[left] at (-7,4) {$C_f$};
    \fill (-11.5,1.5) circle (8pt);
    \fill (-3.8,7) circle (8pt);
\end{tikzpicture}
  \caption{$f(v,w)\ge0$.}
\label{fig:f=0}
  \end{figure}
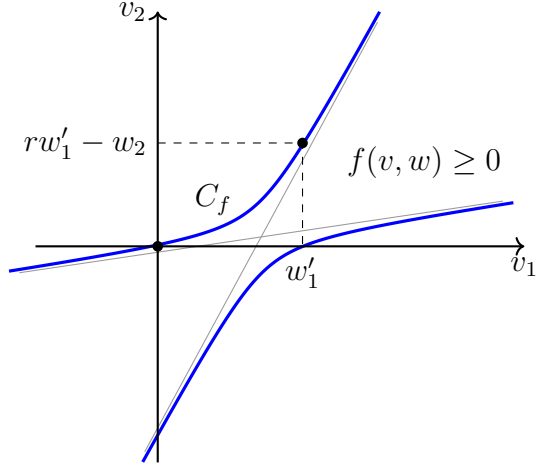

The slope of tangent at point $(v_1,v_2)$ is $(-2v_1+rv_2+w'_1)/(-rv_1+2v_2+w_2)$, so the value 
is $w'_1/w_2$ when $(v_1,v_2)=(0,0)$, 
is $r-\frac{w'_1}{rw'_1-w_2}$ when $(v_1,v_2)=(w'_1,rw'_1-w_2)$.
\end{proof}

\begin{lemma}\label{lemma:deg aP}
Assume $w=(0,w'_1,w_2,0)$, $(w'_1,w_2)\in \Phi^{im}_+$,  $f(v,w)\ge0$, $v^0\in \D(w)\setminus\{{\bf 0}\}$ and $v^0\le v$. Then 
$$\deg(a_{v,v^0;w} P_-(v^0,w)) < f(v,w).$$
\end{lemma}
\begin{proof}
If $r=1$ then $\Phi^{im}_+$ is empty. So $r\ge2$. The case $r=2$ is degenerate (and much simpler than the $r\ge3$ case) so we treat it at the very end.
 
(I) Assume $r\ge3$.

First we show that the inequality is true if $v^0=v$. Because in this case, $a_{v,v'; w}=1$, $\deg P_-(v^0,w)\le 0$, so $\deg(a_{v,v^0;w} P_-(v^0,w)) < f(v,w)$, and the equality holds only if  $\deg P_-(v^0,w)=0$, which is impossible since we assume $v^0\neq{\bf 0}$. 

For the rest of the proof of (I), we assume $v^0<v$. 
As before, denote
$$w^\perp=w-C_qv^0=(w_1-v_1^0+rv_2^0,w_1'-v_1^0,w_2-v_2^0,w_2'-v_2^0+rv_1^0),\quad
v^\perp=v-v^0=(v_1-v_1^0,v_2-v_2^0).$$ 
We can assume $a_{v,v^0;w}(=a_{v^\perp,{\bf 0};w^\perp})$ is nonzero, otherwise there is nothing to prove. Then $\deg a_{v,v^0;w}=f(v^\perp;w^\perp)$ if $a_{v,v^0;w}\neq0$. 
So it suffices to prove that, assuming $f(v^\perp, w^\perp)\ge0$, the following inequality holds:
$$f(v^\perp,w^\perp)+\deg P_-(v^0,w)<f(v,w).$$ 
By Lemma \ref{lem:upperbound for P}, 
it suffices to show the following claim:

\noindent{\bf Claim}. If $0\le v_1\le w'_1$,  $w_2'-v_2+rv_1\ge0$, $v^0\in \D(w)\setminus\{(0,0), (w'_1,rw'_1-w_2)\}$, $v^0< v$, $f(v,w)\ge0, f(v^\perp,w^\perp)\ge0$, then 
$$\min\big(-1,f(v^0,w),f^{swap}(v^0,w), g(v^0,w)\big)<f(v,w)-f(v^\perp,w^\perp).$$

We allow them to be real numbers. Denote $Q(v^0)=Q(v^0_1,v^0_2)=(v^0_1)^2-rv^0_1v^0_2+(v^0_2)^2$. Then 
$$\aligned
&D_1=f(v,w)-f(v^\perp,w^\perp)=Q(v^0)+(w'_1-2v_1)v^0_1+(2rv_1-2v_2-w_2)v^0_2\\
&D_2=f(v,w)-f(v^\perp,w^\perp)-f(v^0,w)=2\big(Q(v^0)-v_1v^0_1+(rv_1-v_2)v^0_2\big)\\
&D_3=f(v,w)-f(v^\perp,w^\perp)-f^{\rm swap}(v^0,w)=2\big(Q(v^0)+(w'_1-v_1)v^0_1+(rv_1-v_2-w_2)v^0_2\big)\\
&D_4=f(v,w)-f(v^\perp,w^\perp)-g(v^0,w)=2\big( (v^0_1)^2+(v^0_2)^2-v_1v^0_1+(rv_1-v_2-w_2)v^0_2\big)\\
\endaligned
$$
To show the claim, it suffices to show that at least one of the following is true: $D_1\ge0$, $D_2>0$, $D_3>0$, or $D_4>0$.  

Now fix $(v,w)$ and let $v^0=(at,bt)$ along a line with a fixed direction $(a,b)\in\mathbb{R}_{\ge0}^2\setminus\{(0,0)\}$ and parameter $t$.

 The condition $v^0\in \D(w)\setminus\{0\}$ implies $rb-a\ge0$, $at\le w'_1$, $bt\le w_2$, $ra-b\ge0$. 
Thus $a,b>0$. (If $a=0$, then $ra-b\ge0$ forces $b=0$, contradicting to the assumption $(a,b)\neq(0,0)$; if $b=0$, then $rb-a\ge0$ forces $a=0$, contradicting to the same assumption.) So we assume for the rest of the proof, that 
$$a=1, b\in\mathbb{Q} \textrm{ and } 1/r\le b\le r.$$ 

Denote 
$$\aligned
&Q_0=Q(a,b)=1-rb+b^2,\\
&A_2=-v_1+(rv_1-v_2)b,\\
&A_3=(w'_1-v_1)+(rv_1-v_2-w_2)b.\\
&A_1=(w'_1-2v_1)+(2rv_1-2v_2-w_2)b=A_2+A_3.\\  
\endaligned
$$
Then 
$$D_1=Q_0t^2+A_1t,\quad D_2=2(Q_0t^2+A_2t),\quad D_3=2(Q_0t^2+A_3t)$$
Note that $A_3=A_2+w'_1-w_2b$. 
Also note that the condition $v^0\le v$ implies $t\le v_1$ and $t\le v_2/b$. 
Then 
$t$ satisfies
$$0<t\le \min(v_1,v_2/b,w'_1, w_2/b)$$

Note that  $Q_0\neq0$, otherwise $b=\frac{r\pm\sqrt{r^2-4}}{2}\notin\mathbb{Q}$ since we assume $r\ge 3$ in part (I). Thus, we only need to consider two cases  $Q_0>0$ (which consists of two subcases 1a and 1b) and $Q_0<0$ (which consists of four subcases 2a--2d) below. 

\medskip

\noindent{\bf Case 1.} If $Q_0>0$, we claim that either $A_2\ge0$, or $A_3\ge0$. It then immediately follows that either $D_2>0$ for all $t>0$, or $D_3>0$ for all $t>0$.  

We consider two subcases: Cases 1a and 1b. See Figure \ref{fig:Q>0,case1}. 

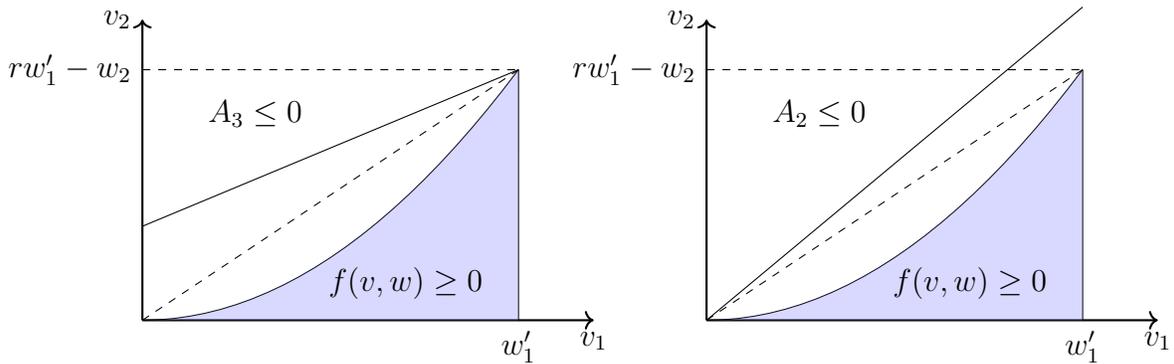
\begin{figure}[h]
\begin{tikzpicture}[scale=5]
    \draw (1,2/3)--(1,0);
    \draw [dashed] (0,0)--(1,2/3) (0,2/3)--(1,2/3);
    \draw [thick] [->] (0,0)--(1.2,0) node[right, below] {$v_1$};
    \draw [thick] [->] (0,0)--(0,.8) node[above, left] {$v_2$};

    \draw [domain=0:1, variable=\x]
      plot ({\x}, {\x*\x*2/3}) node[right] at (1,1){};
    \fill [fill opacity=.3, blue!50, domain=0:1, variable=\x]
      (0, 0)
      -- plot ({\x}, {\x*\x*2/3})
      -- (1, 0)
      -- cycle;
    \node at (0.7,.1) {$f(v,w)\ge0$};
	
    \draw (0,1/4)--(1,2/3);
    \fill [fill opacity=.7,line width=1mm,pattern=north west lines, domain=0:1, path fading = north]
      (0, 1)--(0,1/4)--(1,2/3)--(1,1)--cycle;
    \node at (0.3,.55) {$A_3\le0$};

    \node[below] at (1,0) {$w'_1$};
    \node[left] at (0,2/3) {$rw'_1-w_2$};

  \begin{scope}[shift={(1.5,0)}]
    \draw (1,2/3)--(1,0);
    \draw [dashed] (0,0)--(1,2/3) (0,2/3)--(1,2/3);
    \draw [thick] [->] (0,0)--(1.2,0) node[right, below] {$v_1$};
    \draw [thick] [->] (0,0)--(0,.8) node[above, left] {$v_2$};

    \draw [domain=0:1, variable=\x]
      plot ({\x}, {\x*\x*2/3}) node[right] at (1,1){};
    \fill [fill opacity=.3, blue!50, domain=0:1, variable=\x]
      (0, 0)
      -- plot ({\x}, {\x*\x*2/3})
      -- (1, 0)
      -- cycle;
    \node at (0.7,.1) {$f(v,w)\ge0$};
	
    \draw (0,0)--(1,10/12);
    \fill [fill opacity=.7,line width=1mm,pattern=north west lines, domain=0:1, path fading = north]
      (0, 0)--(0,1)--(1,1)--(1,10/12)--cycle;
    \node at (0.3,.55) {$A_2\le0$};

    \node[below] at (1,0) {$w'_1$};
    \node[left] at (0,2/3) {$rw'_1-w_2$};
  \end{scope}
  
  \end{tikzpicture}
  \caption{Left: Case 1a. Right: Case 1b.}
\label{fig:Q>0,case1}
  \end{figure}
  
\noindent{\bf Case 1a.} $1/r\le b< r_{\rm small}$. We claim that $A_3\ge0$. Indeed, by looking at Figure \ref{fig:Q>0,case1} (Left), it suffices to show that the slope of $A_3=0$ is less than the slope of the diagonal, that is, 
$$r-1/b< r-w_2/w'_1$$
and the above inequality follows from $1/b > 1/r_{\rm small}=r_{\rm big} > w_2/w'_1$. 
So there is no $(v_1,v_2)$ such that $f(v,w)\ge0$ and $A_3<0$ simultaneously.

\noindent{\bf Case 1b.} $r_{\rm big}<b<r$. We claim that $A_2\ge0$. Indeed, by looking at Figure \ref{fig:Q>0,case1} (Right), it suffices to show that the slope of $A_2=0$ is greater than the slope of the diagonal, that is, 
$$r-1/b > (rw'_1-w_2)/w'_1$$
and the above inequality follows from the inequality
$1/b<1/r_{\rm big} = r_{\rm small}< w_2/w'_1$.
So there is no $(v_1,v_2)$ such that $f(v,w)\ge0$ and $A_2<0$ simultaneously. 
\bigskip

\noindent{\bf Case 2.} If $Q_0<0$. Equivalently, if $r_{\rm small}<b<r_{\rm big}$. 
Then for $i=1,2,3$, we have $D_i\ge0$ for $0\le t\le -A_i/Q_0$. We shall consider 4 subcases, Case 2a--2d. 

For convenience, we introduce the following notation 
 (and similar notation for $\ge$ replaced by $\le,>,<$):

For $i=1,2,3$, $H^{\ge0}_i=\{v \ |\  A_i \ge 0\}$, 
$H_{i1}^{\ge0}=\{v\ |\  A_i/(-Q_0)-v_1\ge 0\}$, 
$H_{i2}^{\ge0}=\{v\ |\  A_i/(-Q_0)-v_2/b\ge 0\}$, 
$H_{i3}^{\ge0}=\{v\ |\ A_i/(-Q_0)-(w'_1-v_1)/(rb-1)\ge 0\}$. Define their boundary lines by $\ell_i, \ell_{i1}, \ell_{i2},\ell_{i3}$, respectively. 

Denote $P_{i,jb}=\ell_i\cap\ell_{jb}$.
We collect some easy-to-be-verified observations to be used later. 
For $v=P_{2,33}, P_{3,12}, P_{3,22}, P_{2,13}, P_{12,13}$, we have $f(v,w)=0$, and the coordinates of these points are explicitly computed below: (we use the notation $P=(P^x,P^y)$ to denote the two coordinates for $P$)
$$\aligned
&P_{2,13}=(\frac{b(b(w'_1-rw_2)+w_2)}{b^2-br+1},\frac{(br-1)(b(w'_1-rw_2)+w_2)}{b^2-br+1})\\
&P_{2,33}=(P_{2,33}^x,P_{2,33}^y)=(\frac{b^2w'_1+b(1-rb)w_2}{b^2-rb+1}, \frac{ (rb-1)(bw'_1+(1-rb)w_2) }{b^2-rb+1})\\
&P_{3,12}=(\frac{w'_1-bw_2}{b^2-br+1},\frac{r(w'_1-bw_2)}{b^2-br+1})\\
&P_{3,22}=\big((bw_2-w'_1)/(-Q_0), b(bw_2-w'_1)/(-Q_0)\big) \\
&P_{12,13}=(\frac{b(bw'_1+(br-1)w_2)}{2b^2r^2+b^2-3br+1},\frac{b((2br-1)w'_1-bw_2)}{2b^2r^2+b^2-3br+1}) \\
\endaligned$$ 


$\bullet$ (slope of $\ell_2$) = (slope of $\ell_3=0$) $=r-\frac{1}{b}>b$.

$\bullet$ the line $\ell_2$ passes $O$. 

$\bullet$ the line $\ell_3$ passes through the point $(w'_1,rw'_1-w_2)$, and intersects with the north boundary of $R_{f\ge0}$ at the point $\big(\frac{w'_1-bw_2}{b^2-rb+1}, \frac{b(w'_1-bw_2)}{b^2-rb+1}\big)$.

$\bullet$ (slope of $\ell_{13}$) $=(2r^2b^2-3rb-b^2+1)/(2rb^2-2b)=r+(-b^2-rb+1)/(2rb^2-2b) > $ (slope of $\ell_3$).  So $\ell_{13}$ has positive slope. 

\smallskip

\noindent{\bf Case 2a.} If $A_2\ge A_1,A_3$. We claim that $D_2>0$. See Figure \ref{fig:Q<0,case2a}.

Indeed, $A_2\ge A_3$ implies $w_2b\ge w'_1$; $A_2\ge A_1=A_2+A_3$ implies $A_3\le 0$, 

To prove $D_2\ge0$, we need $Q_0t+A_2\ge0$, that is, $t\le A_2/(-Q_0)$. Since $t\le v_2/b$, it suffices to show $v_2/b\le A_2/(-Q_0)$, or equivalently $v\in H^{\ge0}_{22}$, or equivalently
\begin{equation}\label{eq:2a}
\frac{v_2}{b}\le \frac{-v_1+(rv_1-v_2)b}{-(1-rb+b^2)}
\end{equation}
Note that $b>r_{\rm small}>1/r$, so after simplifying, \eqref{eq:2a} becomes
$$v_2 \le v_1b.$$
This is self-clear from Figure \ref{fig:Q<0,case2a}.  
The inequality $v_2\le v_1b$ is equivalent to saying that the intersection of the curved triangle $OCB$ with the half plane $H^{\le0}_3$, which is the curved triangle $OP_{3,22}F$, is below the line $\ell_{22}$. 

If $P_{3,22}^x<w'_1$, then the intersection of the region $\{A_3<0\}$ with the region 
$$R_{f\ge0}=\{v | f(v,w)\ge0, 0\le v_1\le w'_1, v_2\ge0\}$$ 
is below the line $\ell_{22}$ as shown in the figure.

If $P_{3,22}^x\ge w'_1$, then the whole region $R_{f\ge0}$ is under the line $OE$. 

In both cases, we get the expected inequality $v_2\le v_1b$. Thus $D_2\ge0$. 

Now we need to show that $D_2>0$. If not, then $D_2=0$. We inspect the above computation and get $P_{3,22}=(v_1,v_2)=(v_1,bv_1)$, $t=v_1^0=v_1=A_2/(-Q_0)$, then $v^0_2=bt=bv_1^0=v_2$. Thus $v^0=v$, contradicting to the assumption $v^0<v$.

\begin{figure}[h]
\begin{tikzpicture}[scale=5]
    \draw (1,2/3)--(1,0);
    \draw [dashed] (0,2/3)--(1,2/3);
    \draw [thick] [->] (0,0)--(1.2,0) node[right, above] {$v_1$};
    \draw [thick] [->] (0,0)--(0,.8) node[above, left] {$v_2$};

    \draw [domain=0:1, variable=\x]
      plot ({\x}, {\x*\x*2/3}) node[right] at (1,1){};
    \fill [fill opacity=.3, blue!50, domain=0:1, variable=\x]
      (0, 0)
      -- plot ({\x}, {\x*\x*2/3})
      -- (1, 0)
      -- cycle;
    \node at (0.7,.1) {$R_{f\ge0}$};
	
    \draw (0,-.33)--(1,2/3);
    \fill [fill opacity=.7,line width=1mm,pattern=north west lines, domain=0:1, path fading = north]
      (0,1)--(0,-.33)--(1,2/3)--(1,1)--cycle;

    \draw (0,0)--(1,.34);
    \node[below] at (1,0) {$B=(w'_1,0)$};
    \node[right] at (.8,.35) {$\ell_{22}$};
    \node[left] at (0,2/3) {$rw'_1-w_2$};
   \node[left] at (.5,.4) {$H^{\le0}_3$};
   \node[left] at (0,0) {$O$};
   \node[right] at (0.1,-.25) {$\ell_3$};
   \node[below] at (.35,0) {$F$};
   \node[left] at (.52,.21) {$P_{3,22}$};
   \node[right] at (1,.66) {$C$};

  \end{tikzpicture}
  \caption{Case 2a. If $P_{3,22}^x<w'_1$.}
  \label{fig:Q<0,case2a}
  \end{figure}

\medskip

For the rest, since $a_{v^\perp,{\bf0};w^\perp}\neq0$, by Lemma \ref{lem:v1>a1 then a=0} we have 
\begin{equation}\label{eq:v_1-v^0_1}
v_1-v^0_1=v^\perp_1\le[(w^\perp)'_1-w^\perp_1]_+=[w'_1-rv^0_2]_+
\end{equation}

For Case 2b and 2c below, we assume $w'_1\ge rv^0_2$; for Case 2d, we assume $w'_1<rv^0_2$. 

If $w'_1\ge rv^0_2$, then by \eqref{eq:v_1-v^0_1}, $v^0_1-rv^0_2+w'_1-v_1\ge0$, $(rb-1)t\le w'_1-v_1$, $t\le (w'_1-v_1)/(rb-1)$. 
Therefore 
$t$ satisfies
\begin{equation}\label{eq:case 2,t}
0<t\le \min(v_1,v_2/b,w'_1, w_2/b,(w'_1-v_1)/(rb-1))
\end{equation}

\noindent{\bf Case 2b.} If $w'_1\ge rv^0_2$, $A_3> A_1$, $A_3\ge A_2$. We claim that $D_3>0$. See Figure \ref{fig:Q<0,case2b}.

$A_3>A_1$ implies $A_2<0$; $A_3\ge A_2$ implies $w_2b\le w'_1$. 

To prove $D_3\ge0$, we need $Q_0t+A_3\ge0$, that is, $t\le A_3/(-Q_0)$. Since $t\le (w'_1-v_1)/(rb-1)$, it suffices to show $(w'_1-v_1)/(rb-1)\le A_3/(-Q_0)$, or equivalently $v\in H^{\ge0}_{33}$, or equivalently
\begin{equation}\label{eq:2b}
\frac{w'_1-v_1}{rb-1}\le \frac{(w'_1-v_1)+(rv_1-v_2-w_2)b}{-(1-rb+b^2)}
\end{equation}
After simplifying, \eqref{eq:2b} becomes
$$(r+b-r^2b)(w'_1-v_1)+(rb-1)(rw'_1-w_2-v_2)\ge0$$
Note that $w'_1-v_1,rb-1,rw'_1-w_2-v_2$ are all nonnegative. So if $r+b-r^2b\ge0$, then \eqref{eq:2b} holds. If $r+b-r^2b<0$, then the above defines a half plane $H^{\ge0}_{33}$ below the line 
$$\ell_{33}=\{(v_1,v_2)\ | \ (r+b-r^2b)(w'_1-v_1)+(rb-1)(rw'_1-w_2-v_2)=0\}$$ 
which passes through $C=(w'_1,rw'_1-w_2)$ with slope $(r^2b-r-b)/(rb-1)>0$. 

Note that $A_2<0$ defines the upper half plane $H^{<0}_2$ with boundary line $\ell_2$. 
$$\textrm{( slope of $\ell_{33}$ )} =  r-\frac{b}{rb-1}> r-\frac{1}{b}=\textrm{( slope of $\ell_2$ )}$$

If $P_{2,33}^x\le0$, then $H^{<0}_{33}\cap R_{f\ge0}=\emptyset$. 

If $0\le P_{2,33}^x\le w'_1$: let $F$ be the intersection of $\ell_2$ with $BC$. Then 
$$H^{<0}_2\cap H^{<0}_{33}\cap R_{f\ge0}=\emptyset$$

If $P_{2,33}^x> w'_1$, then 
$$H^{<0}_2\cap R_{f\ge0}=\emptyset\subset H^{\ge0}_{33}.$$ 

In all cases, we get the expected inequality \eqref{eq:2b}. Thus $D_3\ge0$. 

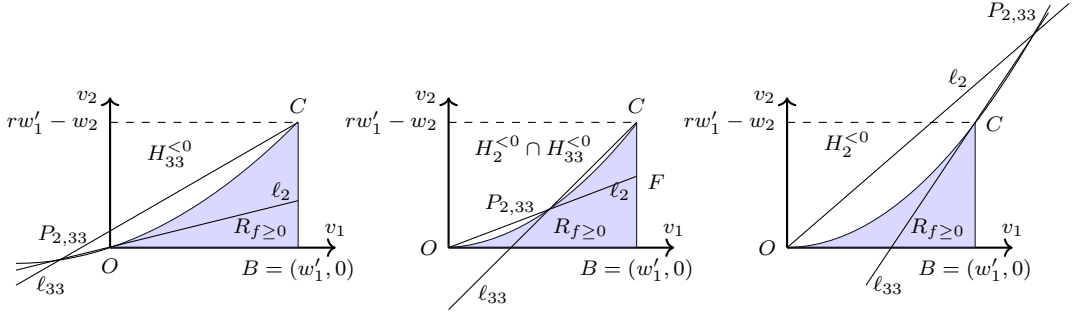
\begin{figure}[h]
\begin{tikzpicture}[scale=2.5]
    \draw (1,2/3)--(1,0);
    \draw [dashed] (0,2/3)--(1,2/3);
    \draw [thick] [->] (0,0)--(1.2,0) node[right, above] {\tiny$v_1$};
    \draw [thick] [->] (0,0)--(0,.8) node[above, left] {\tiny$v_2$};

    \draw [domain=-.5:1, variable=\x]
      plot ({\x}, {\x*(\x+1)*.335}) node[right] at (1,1){};
    \fill [fill opacity=.3, blue!50, domain=0:1, variable=\x]
      (0, 0)
      -- plot ({\x}, {\x*(\x+1)*.335})
      -- (1, 0)
      -- cycle;
    \node at (0.8,.1) {\tiny$R_{f\ge0}$};
	
    \draw (-.5,-.2)--(1,2/3);
    \fill [fill opacity=.7,line width=1mm,pattern=north west lines, domain=0:1, path fading = north]
      (-.5,1)--(-.5,-.2)--(1,2/3)--(1,1)--cycle;

    \draw (-.5,-.12)--(1,.25);
    \node[below] at (1,0) {\tiny$B=(w'_1,0)$};
    \node[right] at (.8,.3) {\tiny$\ell_{2}$};
    \node[left] at (0,2/3) {\tiny$rw'_1-w_2$};
   \node[left] at (.5,.5) {\tiny$H^{<0}_{33}$};
   \node[below] at (0,0) {\tiny$O$};
   \node[right] at (-.45,-.2) {\tiny$\ell_{33}$};
   \node[above] at (-.25,-.05) {\tiny$P_{2,33}$};
   \node[above] at (1,.66) {\tiny$C$};

  \begin{scope}[shift={(1.8,0)}]
    \draw (1,2/3)--(1,0);
    \draw [dashed] (0,2/3)--(1,2/3);
    \draw [thick] [->] (0,0)--(1.2,0) node[right, above] {\tiny$v_1$};
    \draw [thick] [->] (0,0)--(0,.8) node[above, left] {\tiny$v_2$};

    \draw [domain=0:1, variable=\x]
      plot ({\x}, {\x*(\x+.1)*.6}) node[right] at (1,1){};
    \fill [fill opacity=.3, blue!50, domain=0:1, variable=\x]
      (0, 0)
      -- plot ({\x}, {\x*(\x+.1)*.6})
      -- (1, 0)
      -- cycle;
    \node at (0.7,.1) {\tiny$R_{f\ge0}$};
	
    \draw (0,-.33)--(1,2/3);
    \fill [fill opacity=.7,line width=1mm,pattern=north west lines, domain=0:1, path fading = north]
      (0,1)--(0,0)--(0.52,0.2)--(1,.66)--(1,1)--cycle;

    \draw (0,0)--(1,.38);
    \node[below] at (1,0) {\tiny$B=(w'_1,0)$};
    \node[right] at (.8,.3) {\tiny$\ell_{2}$};
    \node[left] at (0,2/3) {\tiny$rw'_1-w_2$};
   \node[above] at (.45,.4) {\tiny$H^{<0}_2\cap H^{<0}_{33}$};
   \node[left] at (0,0) {\tiny$O$};
   \node[right] at (0.1,-.25) {\tiny$\ell_{33}$};
   \node[right] at (1,0.35) {\tiny$F$};
   \node[left] at (.52,.23) {\tiny$P_{2,33}$};
   \node[above] at (1,.66) {\tiny$C$};
  \end{scope}
  
  \begin{scope}[shift={(3.6,0)}]
    \draw (1,2/3)--(1,0);
    \draw [dashed] (0,2/3)--(1,2/3);
    \draw [thick] [->] (0,0)--(1.2,0) node[right, above] {\tiny$v_1$};
    \draw [thick] [->] (0,0)--(0,.8) node[above, left] {\tiny$v_2$};

    \draw [domain=0:1.4, variable=\x]
      plot ({\x}, {\x*(\x+.02)*.65}) node[right] at (1,1){};
    \fill [fill opacity=.3, blue!50, domain=0:1, variable=\x]
      (0, 0)
      -- plot ({\x}, {\x*(\x+.02)*.65})
      -- (1, 0)
      -- cycle;
    \node at (0.8,.1) {\tiny$R_{f\ge0}$};
	
    \draw (0.42, -.2)--(1.39,1.25);
    \fill [fill opacity=.7,line width=1mm,pattern=north west lines, domain=0:1, path fading = north]
      (0,1.25)--(0,0)--(1.39,1.25)--cycle;

    \draw (0,0)--(1.5,1.3);
    \node[below] at (1,0) {\tiny$B=(w'_1,0)$};
    \node[right] at (.8,.9) {\tiny$\ell_{2}$};
    \node[left] at (0,2/3) {\tiny$rw'_1-w_2$};
   \node[left] at (.5,.55) {\tiny$H^{<0}_{2}$};
   \node[left] at (0,0) {\tiny$O$};
   \node[right] at (0.4,-.2) {\tiny$\ell_{33}$};
   \node[left] at (1.39,1.25) {\tiny$P_{2,33}$};
   \node[right] at (1,.66) {\tiny$C$};
  \end{scope}
  
   \end{tikzpicture}
  \caption{Case 2b. Left: $P_{2,33}^x1\le0$. Middle: $0\le P_{2,33}^x\le w'_1$. Right: $P_{2,33}^x>w'_1$.}
  \label{fig:Q<0,case2b}
  \end{figure}

Now we need to show that $D_3>0$. If not, then $D_3=0$. We inspect the above computation and get that $v=C$ (if $P_{2,33}^x\le w'_1$) or $v=P_{2,33}$ (if $0<P^x_{2,33}<w'_1$). 

If $v=C$, then by \eqref{eq:case 2,t}, $v^0_1=t\le (w'_1-v_1)/(rb-1)=0$, thus $v^0_1=0$, $v^0_2=0$, contradicting to the assumption $v^0\neq{\bf0}$.

If $v=P_{2,33}$, then $A_2=0$, $A_3=A_1\ge0$. This is not Case 2b but rather Case 2c, which will be discussed in the next case.

\smallskip

\noindent{\bf Case 2c.} If $w'_1\ge rv^0_2$ and $A_1\ge A_2, A_3$. We claim that $D_1\ge0$.

In this case, $A_2\ge0$, $A_3\ge0$. 

To prove $D_1\ge0$, we need $Q_0t+A_1\ge0$, that is, $t\le A_1/(-Q_0)$. 

Since $t\le \min(v_2/b,(w'_1-v_1)/(rb-1))$, it suffices to show ``$v_2/b\le A_1/(-Q_0)$ or $(w'_1-v_1)/(rb-1)\le A_1/(-Q_0)$'',  or equivalently ``$v\in H^{\ge0}_{12}\cup H^{\ge0}_{13}$'', or equivalently,
$$
\frac{v_2}{b}\le \frac{(w'_1-2v_1)+(2rv_1-2v_2-w_2)b}{-(1-rb+b^2)} \textrm{ or }
\frac{w'_1-v_1}{rb-1}\le \frac{(w'_1-2v_1)+(2rv_1-2v_2-w_2)b}{-(1-rb+b^2)}. 
$$

Note that $\ell_{13}$ intersects the curve $f=0$ at $P_{2,13}$ and $P_{12,13}$, both of which are in $\mathbb{R}_{>0}^2$. 
\medskip

If $w_1'\le w_2b$, then $A_3\le A_2$. The point $C$ satisfies $A_{13}\ge0$ (because $A_{13}(C)=(w_2b-w'_1)/(-b^2+rb-1)\ge0$). See Figure \ref{fig:Q<0,case2c}. 

(i) If $P_{3,12}$ is to the southwest of $P_{12,13}$,  then 
$$H^{<0}_{12}\cap H^{<0}_{13}\cap H^{\ge0}_3=\textrm{(curved triangle)}P_{3,12}P_{12,13}P_{3,13},$$ which has no intersection with $R_{f\ge0}$ except $P_{3,12},P_{12,13}$. 

(ii) If $P_{3,12}$ is to the northeast of $P_{12,13}$, then similarly, $H^{<0}_{12}\cap H^{<0}_{13}\cap H^{\ge0}_3=\emptyset$. 

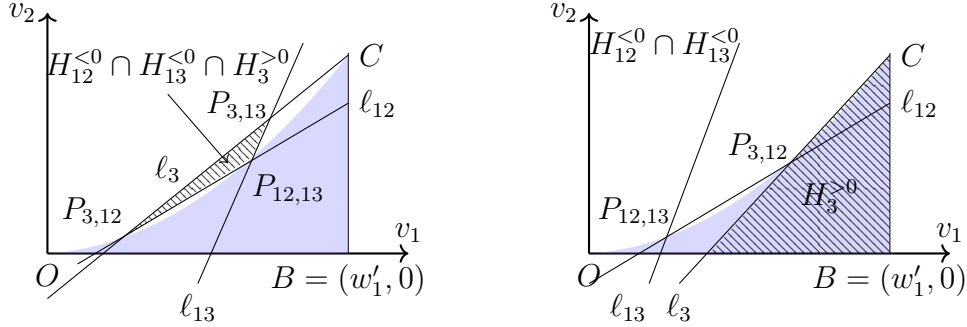
\begin{figure}[h]
\begin{tikzpicture}[scale=4]
    \draw (1,2/3)--(1,0);
    \draw [thick] [->] (0,0)--(1.2,0) node[right, above] {$v_1$};
    \draw [thick] [->] (0,0)--(0,.8) node[above, left] {$v_2$};

    \fill [fill opacity=.3, blue!50, domain=0:1, variable=\x]
      (0, 0)
      -- plot ({\x}, {\x*(\x+.1)*.6})
      -- (1, 0)
      -- cycle;

    \fill [fill opacity=.7,line width=1mm,pattern=north west lines, domain=0:1]
      (.25,.05)--(.68,.32)--(.73,.44)--cycle;
    \draw[<-] (.6,.3)--(.4,.53) node[right, above] {$H^{<0}_{12}\cap H^{<0}_{13}\cap H^{>0}_3$};
	
    \draw (.1,-.035)--(1,.5);
    \node[right] at (1,.5) {$\ell_{12}$};

    \draw (0,-.15)--(1,.66);
    \node[above] at (.4,.2) {$\ell_{3}$};

    \draw (0.5,-.1)--(.85,.7);
    \node[below] at (.5,-.1) {$\ell_{13}$};

    \node[below] at (1,0) {$B=(w'_1,0)$};
   \node[below] at (0,0) {$O$};
   
   \node[above] at (.15,.05) {$P_{3,12}$};
   \node[below] at (.8,.3) {$P_{12,13}$};
   \node[above] at (.63,.4) {$P_{3,13}$};
   \node[right] at (1,.66) {$C$};

  \begin{scope}[shift={(1.8,0)}]
    \draw (1,2/3)--(1,0);
    \draw [thick] [->] (0,0)--(1.2,0) node[right, above] {$v_1$};
    \draw [thick] [->] (0,0)--(0,.8) node[above, left] {$v_2$};

    \fill [fill opacity=.3, blue!50, domain=0:1, variable=\x]
      (0, 0)
      -- plot ({\x}, {\x*(\x+.1)*.6})
      -- (1, 0)
      -- cycle;

    \fill [fill opacity=.7,line width=1mm,pattern=north west lines, domain=0:1, path fading = north]
      (0,.7)--(0,-.1)--(.25,.05)--(.5,.7)--cycle;
   \node[above] at (.25,.6) {$H^{<0}_{12}\cap H^{<0}_{13}$};
         	
    \fill [fill opacity=.7,line width=1mm,pattern=north west lines, domain=0:1]
      (0.4,0)--(1,0)--(1,.66)--cycle;
   \node[above] at (.8,.1) {$H^{>0}_3$};
	
    \draw (0,-.1)--(1,.5);
    \node[right] at (1,.5) {$\ell_{12}$};

    \draw (.3,-.1)--(1,.66);
    \node[below] at (.3,-.1) {$\ell_{3}$};

    \draw (0.2,-.1)--(.5,.7);
    \node[below] at (.13,-.1) {$\ell_{13}$};

    \node[below] at (1,0) {$B=(w'_1,0)$};
   \node[below] at (0,0) {$O$};
   
   \node[above] at (.15,.05) {$P_{12,13}$};
   \node[left] at (.7,.35) {$P_{3,12}$};
   \node[right] at (1,.66) {$C$};
  
  \end{scope}

   \end{tikzpicture}
  \caption{Case 2c, $w_1'\le w_2b$. Left: $P_{3,12}^x\le P_{12,13}^x$. Right: $P_{3,12}^x> P_{12,13}^x$.}
  \label{fig:Q<0,case2c}
  \end{figure}

\medskip

If $w_1'> w_2b$, then $A_3> A_2$. The point $C$ satisfies $A_{13}<0$, the origin $O$ satisfies $A_{12}>0$. See Figure \ref{fig:Q<0,case2c part2}.

(i) If $P_{12,13}$ is to the southwest of $P_{2,13}$, then $H^{<0}_{12}\cap H^{<0}_{13}\cap H^{>0}_2\cap R_{f\ge0}=\emptyset$. 

(ii) If $P_{12,13}$ is to the northeast of $P_{2,13}$, then we even have $H^{<0}_{12}\cap H^{<0}_{13}\cap H^{>0}_2=\emptyset$. 

Thus $D_1\ge0$. 

\begin{figure}[h]
\begin{tikzpicture}[scale=4]
    \draw (1,2/3)--(1,0);
    \draw [thick] [->] (0,0)--(1.2,0) node[right, above] {$v_1$};
    \draw [thick] [->] (0,0)--(0,.8) node[above, left] {$v_2$};

    \fill [fill opacity=.3, blue!50, domain=0:1, variable=\x]
      (0, 0)
      -- plot ({\x}, {\x*(\x+.1)*.6})
      -- (1, 0)
      -- cycle;
	
    \fill [fill opacity=.7,line width=1mm,pattern=north west lines, domain=0:1]
      (.5,.18)--(.81,.45)--(.2,.1)--cycle;
    \draw[<-] (.5,.22)--(.4,.53) node[right, above] {$H^{<0}_{12}\cap H^{<0}_{13}\cap H^{>0}_2$};
	
    \draw (0,.05)--(1,.3);
    \node[right] at (1,.3) {$\ell_{12}$};

    \draw (0,0)--(1,.55);
    \node[above] at (.4,.25) {$\ell_{2}$};

    \draw (0.2,-.1)--(1.1,.7);
    \node[below] at (.2,-.1) {$\ell_{13}$};

    \node[below] at (1,0) {$B=(w'_1,0)$};
   \node[below] at (0,0) {$O$};
   
   \node[below] at (.65,.2) {$P_{12,13}$};
   \node[above] at (.7,.4) {$P_{2,13}$};
   \node[above] at (1,.66) {$C$};

  \begin{scope}[shift={(1.8,0)}]
    \draw (1,2/3)--(1,0);
    \draw [thick] [->] (0,0)--(1.2,0) node[right, above] {$v_1$};
    \draw [thick] [->] (0,0)--(0,.8) node[above, left] {$v_2$};

    \fill [fill opacity=.3, blue!50, domain=0:1, variable=\x]
      (0, 0)
      -- plot ({\x}, {\x*(\x+.1)*.6})
      -- (1, 0)
      -- cycle;

    \fill [fill opacity=.7,line width=1mm,pattern=north west lines, domain=0:1, path fading = north]
      (0,.8)--(0,.2)--(.8,.45)--(1,.56)--(1,.8)--cycle;
   \node[above] at (.3,.5) {$H^{<0}_{12}\cap H^{<0}_{13}$};
         	
    \fill [fill opacity=.7,line width=1mm,pattern=north west lines, domain=0:1]
      (0,0)--(1,0)--(1,.2)--cycle;
   \node[above] at (.7,0) {$H^{>0}_2$};

    \draw (0,.2)--(1,.5);
    \node[right] at (1,.5) {$\ell_{12}$};

    \draw (0,0)--(1,.2);
    \node[right] at (1,.2) {$\ell_{2}$};

    \draw (0.05,-.05)--(1.1,.63);
    \node[below] at (.1,-.05) {$\ell_{13}$};

    \node[below] at (1,0) {$B=(w'_1,0)$};
   \node[below] at (0,0) {$O$};
   
   \node[above] at (.2,0.05) {$P_{2,13}$};
   \node[above] at (.7,.42) {$P_{12,13}$};
   \node[above] at (1,.66) {$C$};
  \end{scope}
    
   \end{tikzpicture}
  \caption{Case 2c, $w_1'> w_2b$. Left: $P_{12,13}^x\le P_{2,13}^x$. Right: $P_{12,13}^x> P_{2,13}^x$.}
  \label{fig:Q<0,case2c part2}
  \end{figure}
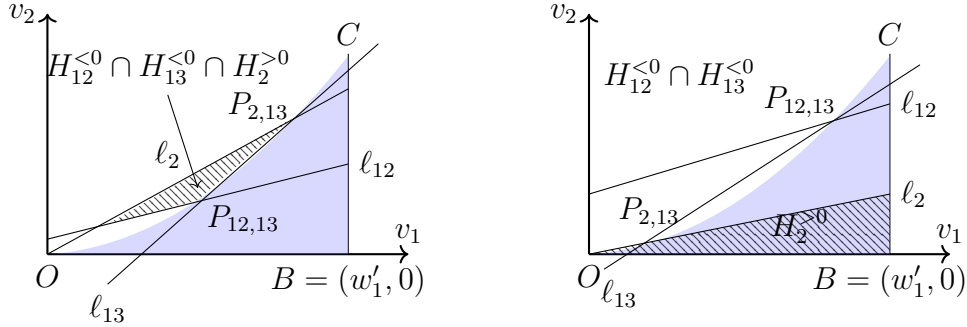
\medskip

 \noindent{\bf Case 2d.} If $w'_1< rv^0_2$. We assert that $D_1\ge0$ or $D_4>0$. Indeed, \eqref{eq:v_1-v^0_1} forces $v_1^0=v_1$, so the assertion follows from the following Claim.

\noindent {\bf Claim.}
 If $r_{\rm small}\le w_2/w'_1\le r_{\rm big}$,  $0\le v^0_2\le v_2$, $(v^0,w)$ is $l$-dominant (that is, $rv^0_2\ge v^0_1$, $v^0_1\le w'_1$, $v^0_2\le w_2$, $rv^0_1\ge v^0_2$), $v^0_1=v_1$, $f(v,w)\ge0$, then either $D_1\ge0$ or $D_4>0$.

\noindent Proof of Claim. 
Let $\bar{f}:[0,w'_1]\to\mathbb{R}_{\ge0}$ be the function with graph $C_f$. 
For each fixed $v_1$, if we increase $v_2$ then both $D_1$ and $D_2$ will decrease. 
So we can assume $v_2=\bar{f}(v_1)$. 

Denote $x=v_1,y=v^0_2$. Then
$$f(v,w)=-x^2+rx\bar{f}(x)-\bar{f}(x)^2+w'_1x-w_2\bar{f}(x)=0$$
$$\aligned
D_1&=(x^2-rxy+y^2)+(w'_1-2x)x+(2rx-2\bar{f}(x)-w_2)y\\
&=-x^2+rxy+y^2+w'_1x-2\bar{f}(x)y-w_2y\\
&=f(v,w)+(w_2-rx)(\bar{f}(x)-y)+(\bar{f}(x)-y)^2=(w_2-rx+\bar{f}(x)-y)(\bar{f}(x)-y)
\endaligned
$$
So $D_1\ge0$ if $y\le w_2-rx+\bar{f}(x)$.
On the other hand, if $y>w_2-rx+\bar{f}(x)$, then 
$$
D_4/2=x^2+y^2-x^2+(rx-\bar{f}(x)-w_2)y=y(y+rx-\bar{f}(x)-w_2)\ge0, 
$$
and the equality holds if and only if $y=0$, that is, $v_2^0=0$. But then $v_1^0\le rv^0_2=0$ forces $v^0_1=0$. This contradicts the assumption that $v^0=(v^0_1,v^0_2)\neq{\bf 0}$.  
So either $D_1\ge0$ or $D_4>0$. The claim is proved.

\smallskip

(II) Assume $r=2$. Then $(w'_1,w_2)\in \Phi^{im}_+$ implies $a_1=a_2=w'_1=w_2>0$; $f(v,w)\ge0$ together with other conditions on $v$ imply that $v$ lies in the triangle $OBC$ where $O=(0,0)$, $B=(a_1,0)$, $C=(a_1,a_1)$. In particular, $v_1\ge v_2$.

\begin{figure}[h]
\def\tikzscale{1}
\begin{tikzpicture}[scale=\tikzscale]
   \draw [thick] [->] (0,0)--(3,0) node[right, above] {$v_1$};
   \draw [thick] [->] (0,0)--(0,3) node[above, left] {$v_2$};
    \draw [dashed]  (2,0)--(2,2)--(0,2);
    \draw[thick,blue]  (-.5,-.5)--(2.5,2.5);
   \draw[thick,blue] (1,-1)--(4,2);
    \node[below right] at (2,0) {\tiny$B=(a_1,0)$};
   \node[right] at (2,2) {\tiny$C=(a_1,a_1)$};
   \node[left] at (0,2) {$rw'_1-w_2=a_1$};
   \node[right] at (2.5,2.5) {\tiny$f(v,w)\ge0$};
   \node[left] at (0,0) {$O$};
  \end{tikzpicture}
\label{fig:f=0,r=2}
  \end{figure}
  
Like (I), it suffices to prove one of the following: $D_1\ge0$, $D_2>0$, $D_3>0$, or $D_4>0$. Define $v^0=(at,bt)$ (with $a=1$), $Q_0,A_2,A_3,A_4$ as in (I). 

If $b=a(=1)$, then $Q_0=(a-b)^2=0$, $A_2=-v_1+(2v_1-v_2)=v_1-v_2$, $A_3=(a_1-v_1)+(2v_1-v_2-a_1)=v_1-v_2$, $A_1=A_2+A_3=2(v_1-v_2)\ge0$.  So $D_1=A_1t\ge0$, we are done.

If $b\neq a(=1)$, then $Q_0=(a-b)^2>0$, it suffices to prove that $A_2\ge0$ or $A_3\ge0$. The proof of Case 1 of (I) applies: we consider two subcases.

Case 1a: $b<1$. We claim that $A_3\ge0$. It suffices to show that the slope of $A_3=0$ is less than the slope of the diagonal, that is, $2-1/b<1$. This is obviously true.

Case 1b: $b>1$. We claim that $A_2\ge0$. It suffices to show that the slope of $A_2=0$ is greater than the slope of the diagonal, that is,
$2-1/b>1$. This is again true.
\end{proof}

Now we can complete the proof of the main theorem. 
\begin{proof}[Proof of Theorem \ref{main theorem}]
Let $w=(0,a_1,ra_1-a_2,0)$ (so $w={}^\phi w$ and satisfies $w'_1-w_1=a_1$, $w'_2-w_2=a_2-r[a_1]_+$). Let $v=(a_1-q,\, p)$.

Apply Lemma \ref{BBDG property} to the morphism $\pi: \mytildeF_{v,w}\to \myE_{v,w}$. Note that
$$
2\pi^{-1}(0)-\dim\mytildeF_{v,w}=2d(v,w)-\tilde{d}(v,w)=f(v,w)
$$
Note that $a_{v,0;w}(t)=\sum s_{0,b}t^b$ using the notation in \eqref{decomposition 0}. By Lemma \ref{BBDG property},  if $f(v,w)<0$, 
then $a_{v,0;w}(t)=0$, and there is nothing to prove. 

For the rest of the proof, we assume $f(v,w)\ge 0$. Then the degrees of the Laurent monomials in $a_{v,0;w}(t)$ must be in the interval $[-f(v,w),f(v,w)]$. Therefore, the degrees of the Laurent monomials in corresponding $e(p,q)=a_{v,0;w}(\v^r)$ must lie in the set $$\{-rf(v,w),-rf(v,w)+r,-rf(v,w)+2r,..., rf(v,w)\}.$$

Compare the max-degrees of both sides of \eqref{eq:sumAP-}, we see that the degree of right side is
$-v_1w_1-v_2w_2+v_2(w'_2+rv_1-v_2)+v_1(w'_1-v_1)
=-v_1^2+rv_1v_2-v_2^2+v_1(w'_1-w_1)+v_2(w'_2-w_2)=f(v,w)
$. 
On the other hand, by Lemma \ref{lemma:deg aP}, the summand $a_{v,v';w}P_-(v',w)$ on the left side of \eqref{eq:sumAP-} has degree less than $f(v,w)$ for all $v'\in \D(w)\setminus\{{\bf0}\}$. This forces $\deg(a_{v,{\bf0};w}P_-({\bf0},w))=f(v,w)$. Since $P_-({\bf0},w)=1$, we conclude that  $\deg a_{v,{\bf0};w}=f(v,w)$. All conclusions then follow from Lemma \ref{BBDG property}.
\end{proof}

\begin{proof}[Proof of Corollary \ref{cor:precise condition for e(p,q)neq0}]
If $(a_1,a_2)\notin\Phi_+^{im}$, the statement follows from Theorem \ref{theorem: real root}.
If $(a_1,a_2)\in\Phi_+^{im}$, the statement follows from Theorem \ref{main theorem} and the definition of $R(a_1,a_2)$. 
\end{proof}

\begin{proof}[Proof of Corollary \ref{cor:BBDG support Nakajima}]
By \eqref{eq:a=a}, $a_{v,v';w}=a_{v^\perp,0;w^\perp}$ where
 $v^\perp=(v_1-v'_1,v_2-v'_2)$, $w^\perp=(-v'_1+rv'_2,w_1'-v'_1,w_2-v'_2,-v'_2+rv'_1)$. Correspondingly, 
$a_1=(w^\perp)'_1-w^\perp_1 = w'_1-rv'_2$, 
$a_2=r[a_1]_+ +(w^\perp)'_2-w^\perp_2=r[a_1]_++rv'_1-w_2$,
$p=v_2^\perp=v_2-v'_2$, $q=[a_1]_+-v_1^\perp=[a_1]_+-v_1+v'_1$. 
Then $\mathbf{E}_{v',w}$ appears in the decomposition if 
the condition in Corollary \ref{cor:precise condition for e(p,q)neq0} is satisfied.
\end{proof}


\section{Analogue of Kazhdan-Lusztig polynomials}

We define $P(v,w)$ similar to the definition of Kazhdan-Lusztig polynomial:
$$P(v,w)=t^{\tilde{d}(v,w)}P_-(v,w)=\sum_k t^{\tilde{d}(v,w)+k}\dim H^k(i_0^* IC_w(v))\in\mathbb{Z}[t^\pm]$$

\begin{proposition}
Assume $(v,w)$ is $l$-dominant. Then $P(v,w)$ is a polynomial in $\mathbb{Z}[t]$ such that:

{\rm(i)} all coefficients are nonnegative; 

{\rm(ii)} the constant term is $1$; 

{\rm(iii)} $\deg P(v,w)\le \max(0,\tilde{d}(v,w)-1)$. More precisely, either $P(v,w)=1$, or $\deg P(v,w)\le \tilde{d}(v,w)+\min\big(-1,f(v,w),f^{swap}(v,w), g(v,w)\big)$.
\end{proposition}
\begin{proof}
(i) follows from the definition. (ii) and (iii) follow immediately from Lemma \ref{lem:upperbound for P}.
\end{proof}

Use \eqref{eq:variation of h=something-h} we can determine $P_-(v,w)$, thus $P(v,w)$, as follows: since
$$P_-(v,w)=
t^{d_{v,w}-\tilde{d}_{v,w}}{w'_2+rv_1 \brack v_2}_{t}{w'_1\brack v_1}_{t}
-\sum_{v^0\in \D(w)\setminus\{v\}} a_{v,v^0;w} P_-(v^0,w)
$$
we can compute $P_-(v,w)$ recursively on $v$, if we can compute $a_{v,v^0;w}$. On the other hand,  $a_{v,v^0;w}=a_{v^\perp,0;w^\perp}=a_{v^\perp,0; {}^\phi(w^\perp)}$ (the first equality follows from the slice theorem, the second equality follows from Lemma \ref{chiL(w)=C}), which are coefficients of certain triangular basis elements according to Lemma \ref{lem:e=a}, and the triangular basis can be recursively computed as described in \cite{BZ2}.

\begin{example}
Assume $w=(0,w_1',w_2,0)\in\mathbb{Z}_{\ge0}^4$. Using the recursive method described in the previous paragraph, we compute a few examples of $P_-(v,w)$ and $P(v,w)$:

{\tiny
For $r=2$: 

\begin{center}
 \begin{tabular}{|c|c|c|c|} 
 \hline
 $v$ &$w$&$P_-(v,w)$ & $P(v,w)$ \\ 
 \hline
 $(1,1)$ & any & $1$ & $1$\\
 \hline
 $(1,1)$&$(0,1,1,0)$ &$t^{-2}$ &$1$\\
 \hline
 $(1,1)$&$(0,1,i,0), i\ge2$&$t^{-i-1}+t^{-i+1}$ &$1+t^2$\\
 \hline
 $(1,1)$&$(0,2,2,0)$&$t^{-4}+2t^{-2}$ &$1+2t^2$\\
 \hline
 $(1,1)$&$(0,2,i,0), i\ge 3$&$t^{-i-2}+2t^{-i}+t^{-i+2}$ & $1+2t^2+t^4$\\
 \hline
 $(1,1)$&$(0,3,3,0)$&$t^{-6}+2t^{-4}+2t^{-2}$ &$1+2t^2+2t^4$\\
 \hline
 $(1,1)$&$(0,3,i,0), i\ge 4$&$t^{-i-3}+2t^{-i-1}+2t^{-i+1}+t^{-i+3}$ & $1+2t^2+2t^4+t^6$\\
  \hline
\end{tabular}
\end{center}

For $r=3$: 

\begin{center}
 \begin{tabular}{|c|c|c|c|} 
 \hline
 $v$ &$w$&$P_-(v,w)$ & $P(v,w)$ \\ 
 \hline
 $(1,1)$&$(0,2,3,0)$ &$t^{-6}+2t^{-4}+2t^{-2}$ &$1+2t^2+2t^4$\\
 \hline
 $(1,1)$&$(0,3,3,0)$ &$t^{-7}+2t^{-5}+3t^{-3}+t^{-1}$ &$1+2t^2+3t^4+t^6$\\
 \hline
 $(1,1)$&$(0,4,4,0)$&$t^{-9}+2t^{-7}+3t^{-5}+3t^{-3}+t^{-1}$ &$1+2t^2+3t^4+3t^6+t^8$\\
 \hline
 $(1,2)$&$(0,4,4,0)$&$t^{-13}+2t^{-11}+3t^{-9}+3t^{-7}+2t^{-5}+t^{-3}$ &$1+2t^2+3t^4+3t^6+2t^8+t^{10}$\\
 \hline
 $(2,2)$&$(0,4,4,0)$&\makecell{$t^{-20}+2t^{-18}+5t^{-16}+6t^{-14}+8t^{-12}+6t^{-10}$\\ $+5t^{-8}+2t^{-6}+t^{-4}$}&\makecell{$1+2t^2+5t^4+6t^6+8t^8+6t^{10}$\\ $+5t^{12}+2t^{14}+t^{16}$}\\
  \hline
\end{tabular}
\end{center}
}

We expect that $P(v,w)$ to always be unimodal. Note that the last example shows it is not necessarily symmetric or log-concave. It would be interesting to look for a combinatorial model for $P(v,w)$. 
\end{example}

\section{Examples}
The examples in \S\ref{examp1}  are the triangular basis element $C[3,4]$ and $C[2,8]$ in $\mathcal{A}_{\v}(3,3)$. In \S\ref{subsection:geometry0110} we use the geometry of the Nakajima's quiver varieties to compute $\chi(L(w))$ for $w=(0,1,1,0)$ and $r\ge2$ to show that it coincides with a triangular basis element. In \S\ref{eg:main thm} we use an example to illustrate the idea of the proof of the main theorem. In \S\ref{example for 1.5}, we give an example of Corollary \ref{cor:BBDG support Nakajima}.

\subsection{Two examples of triangular basis elements}\label{examp1}
Consider the skew-symmetric cluster algebra $\mathcal{A}_{\v}(3,3)$.
 
(a) $C[3,4]$. Using computer, we get
{\tiny
\noindent$C[3,4]=X^{\left(-3, -4\right)} + \left(\v^{6} + 1+\v^{-6}\right)X^{\left(-3, -1\right)} +  \left(\v^{6} + 1+\v^{-6}\right)X^{\left(-3, 2\right)} +  X^{\left(-3, 5\right)} +  \left(\v^{9} + \v^3+  \v^{-3}+\v^{-9}\right)X^{\left(0, -4\right)} +  \left(\v^{6} + 1+\v^{-6}\right)X^{\left(0, -1\right)} + \left(\v^{12} + \v^{6} + 2 + \v^{-6} + \v^{-12} \right)X^{\left(3, -4\right)} + X^{\left(3, -1\right)}\\ + \left(\v^{9} + \v^3+\v^{-3}+\v^{-9}\right)X^{\left(6, -4\right)} + X^{\left(9, -4\right)}$. 
}

Below the table for the coefficients $e(p,q)$. We see that the max-degrees of the coefficients match the corresponding values of $D(p,q)$ in Figure \ref{fig:C34}. 

{\tiny
\begin{center}
 \begin{tabular}{|c|ccccc|} 
 \hline
 \tikz{\node[below left, inner sep=1pt] (def) {$q$};%
      \node[above right,inner sep=1pt] (abc) {$p$};%
      \draw (def.north west|-abc.north west) -- (def.south east-|abc.south east);}
&
$0$ &$1$&$2$ &$3$& $4$\\ 
 \hline
3& $1$ &  & & &\\
2& $\left(\v^{6} + 1+\v^{-6}\right)$ &  &  & &\\
1& $\left(\v^{6} + 1+\v^{-6}\right)$ & $ \left(\v^{6} + 1+\v^{-6}\right)$ & $1$ & &\\
0& $1$ & $\left(\v^{9} + \v^3+\v^{-3}+\v^{-9}\right)$ & $ \left(\v^{12} + \v^{6} + 2 + \v^{-6} + \v^{-12} \right)$ & $ \left(\v^{9} + \v^3+\v^{-3}+\v^{-9}\right)$&$1$\\
  \hline
\end{tabular}
\end{center}
}

\begin{figure}[h]
\centering
\begin{minipage}{.3\textwidth}
\begin{tikzpicture}
    [
        dot/.style={circle,draw=black, fill,inner sep=1pt},
        scale=.7
    ]
\draw [black!70] {(0,3) to [out=308, in=160] (4,0)};
\fill [blue!10] {(0,3) to [out=308, in=160] (4,0) to (0,0)};

\foreach \x in {0,...,5}{
    \node[dot] at (\x,0){ };
    \node[dot] at (\x,1){ };
    \node[dot] at (\x,2){ };
    \node[dot] at (\x,3){ };
    \node[dot] at (\x,4){ };
}

\node[anchor=south west] at (0,0) {\tiny0};
\node[anchor=south ] at (1,0) {\tiny9};
\node[anchor=south ] at (2,0) {\tiny12};
\node[anchor=south ] at (3,0) {\tiny9};
\node[anchor=south ] at (4,0) {\tiny0};
\node[anchor=south ] at (5,0) {\tiny$-15$};

\node[anchor=south west] at (0,1) {\tiny6};
\node[anchor=south ] at (1,1) {\tiny6};
\node[anchor=south ] at (2,1) {\tiny0};
\node[anchor=south ] at (3,1) {\tiny$-12$};

\node[anchor=south west] at (0,2) {\tiny6};
\node[anchor=south ] at (1,2) {\tiny$-3$};
\node[anchor=south ] at (2,2) {\tiny$-18$};

\node[anchor=south west] at (0,3) {\tiny0};
\node[anchor=south] at (1,3) {\tiny$-18$};

\node[anchor=south west] at (-.2,4) {\tiny$-12$};

\draw[->,thick,-latex] (0,0) -- (0,5);
\node[anchor=west] at (0,5) {\tiny$q$};
\draw[->,thick,-latex] (0,0) -- (6,0);
\node[anchor=south] at (6,0) {\tiny$p$};  
\end{tikzpicture}
\end{minipage}
\begin{minipage}{.3\textwidth}
$\phantom{X}$\\
$\phantom{X}$\\
\includegraphics[width=5cm]{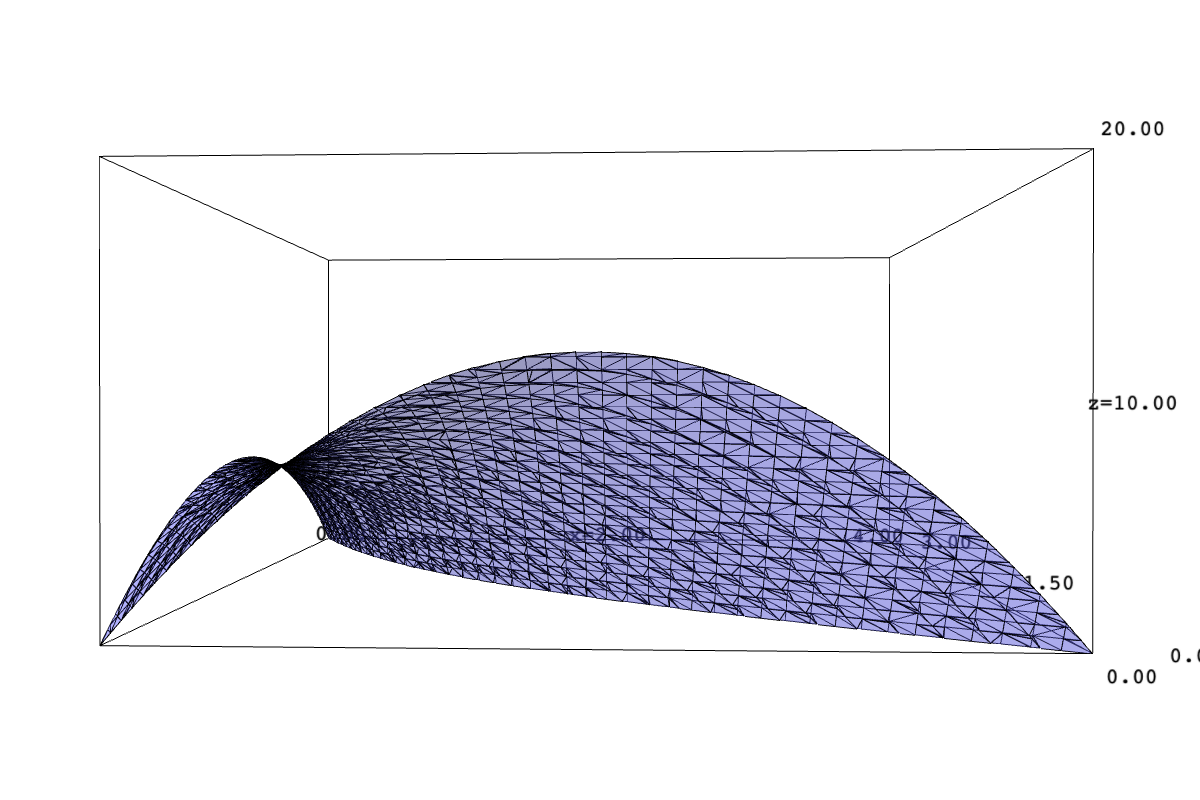}
\end{minipage}
\begin{minipage}{.3\textwidth}
$\phantom{X}$\\
\includegraphics[width=5cm]{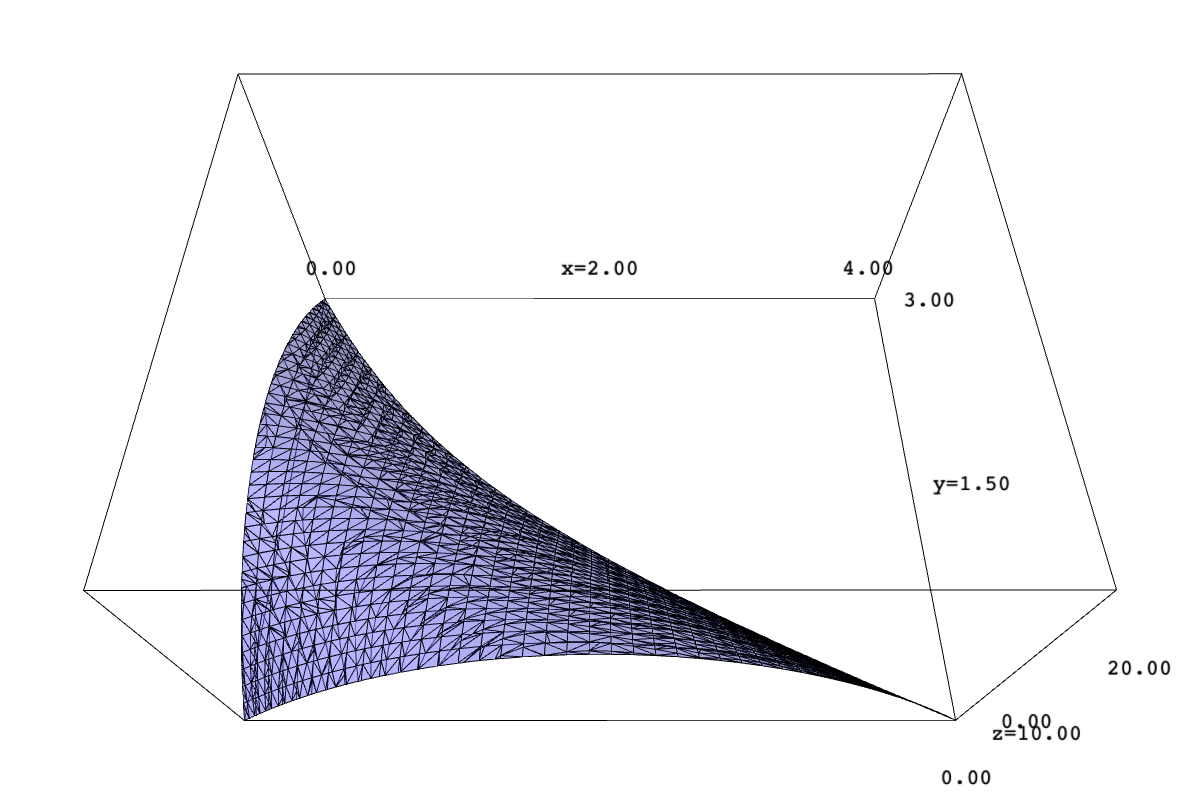}
\end{minipage}
\caption{Left: some values of $D(p,q)$; the shaded region is $\{(p,q)\in[0,4]\times [0,3]\ | \ D(p,q)\ge0\}$. Middle: the front view of $D(p,q)$ for $(p,q)\in R(3,4)$. Right: the top view of the the graph of $D(p,q)$ for $(p,q)\in R(3,4)$. }
\label{fig:C34}
\end{figure}
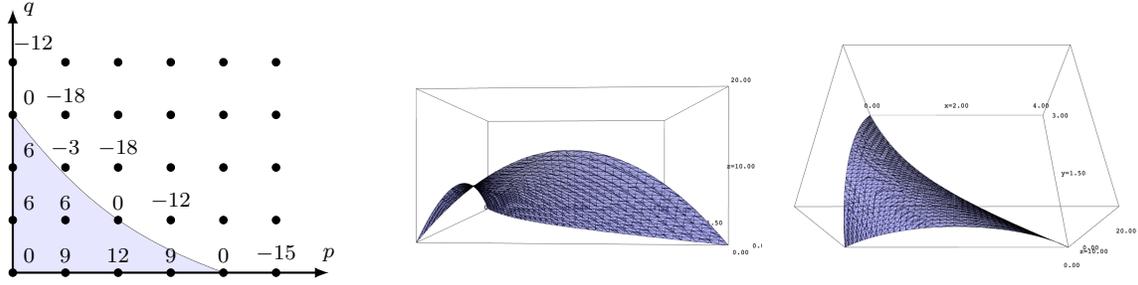

(b) $C[2,8]$. Using computer, we get
{\tiny
\noindent$C[2,8]=X^{\left(-2, -8\right)} + \left(\v^{3} +\v^{-3}\right)X^{\left(-2, -5\right)}  + \cdots+X^{\left(22, -8\right)}$
}.

Again, the max-degrees of the coefficients match the corresponding values of $D(p,q)$; see Figure \ref{fig:C28}. 

\begin{figure}[h]
\centering
\begin{minipage}{.35\textwidth}
\begin{tikzpicture}
    [
        dot/.style={circle,draw=black, fill,inner sep=1pt},
        scale=.6
    ]
\draw [black!70] {(0,2) to (2,2) to (8,0)};
\fill [blue!10] {(0,0) to (0,2) to (2,2) to (8,0) to (0,0)};

\foreach \x in {0,...,8}{
    \node[dot] at (\x,0){ };
    \node[dot] at (\x,1){ };
    \node[dot] at (\x,2){ };
}

\node[anchor=south west] at (0,0) {\tiny0};
\node[anchor=south ] at (1,0) {\tiny21};
\node[anchor=south ] at (2,0) {\tiny36};
\node[anchor=south ] at (3,0) {\tiny45};
\node[anchor=south ] at (4,0) {\tiny48};
\node[anchor=south ] at (5,0) {\tiny45};
\node[anchor=south ] at (6,0) {\tiny36};
\node[anchor=south ] at (7,0) {\tiny21};
\node[anchor=south ] at (8,0) {\tiny0};

\node[anchor=south west] at (0,1) {\tiny3};
\node[anchor=south ] at (1,1) {\tiny15};
\node[anchor=south ] at (2,1) {\tiny21};
\node[anchor=south ] at (3,1) {\tiny21};
\node[anchor=south ] at (4,1) {\tiny15};
\node[anchor=south ] at (5,1) {\tiny3};

\node[anchor=south west] at (0,2) {\tiny0};
\node[anchor=south ] at (1,2) {\tiny$3$};
\node[anchor=south ] at (2,2) {\tiny$0$};

\draw[->,thick,-latex] (0,0) -- (0,3);
\node[anchor=west] at (0,3) {\tiny$q$};
\draw[->,thick,-latex] (0,0) -- (9,0);
\node[anchor=south] at (9,0) {\tiny$p$};  
\end{tikzpicture}
\end{minipage}
\begin{minipage}{.3\textwidth}
$\phantom{X}$\\
\includegraphics[width=5cm]{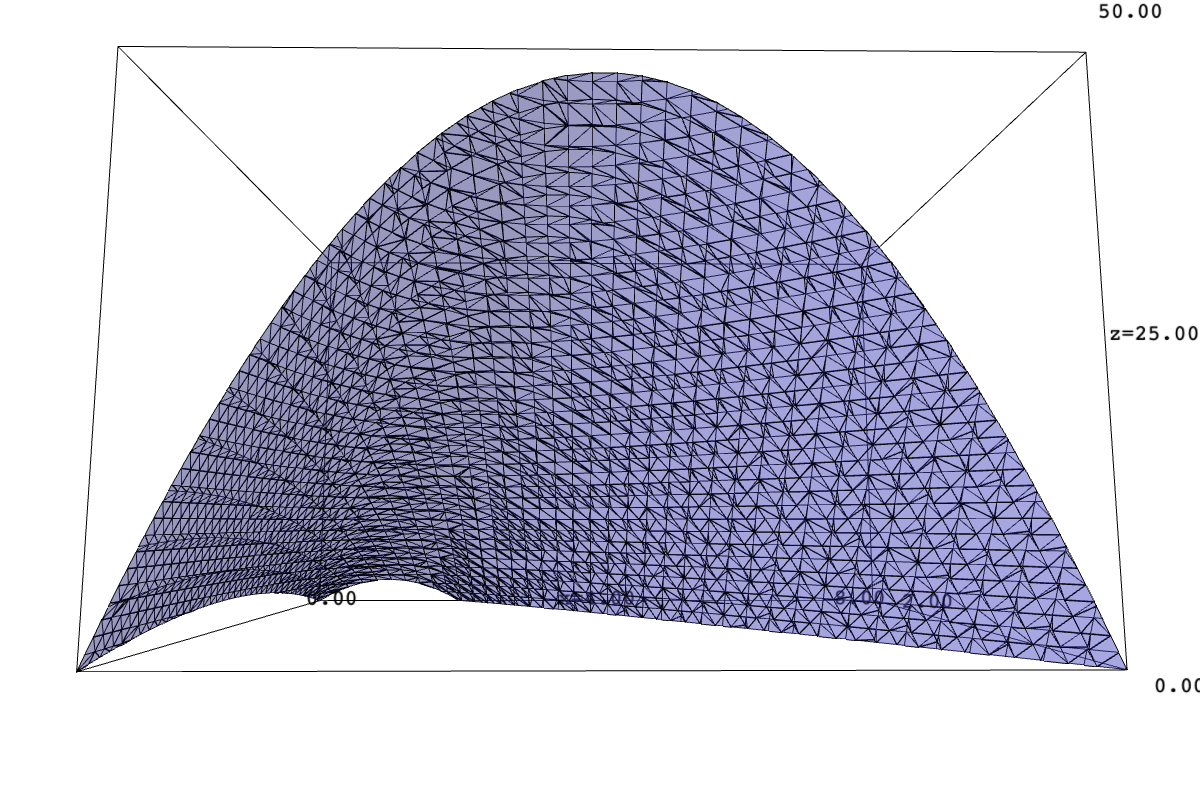}
\end{minipage}
\begin{minipage}{.3\textwidth}
\includegraphics[width=6cm]{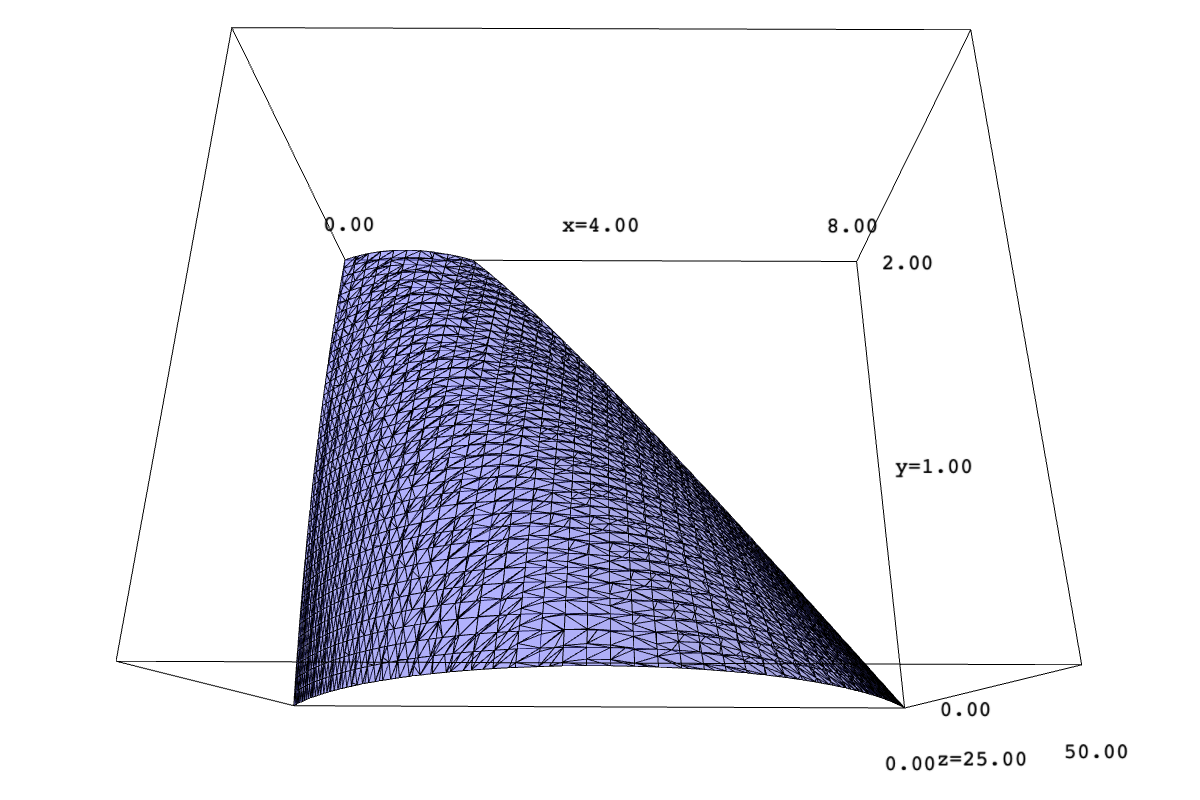}
\end{minipage}
\caption{Left: some values of $D(p,q)$; the shaded region is $\{(p,q)\in[0,8]\times [0,2]\ | \ D(p,q)\ge0\}$. Middle: the front view of $D(p,q)$ for $(p,q)\in R(2,8)$. Right: the top view of the the graph of $D(p,q)$ for $(p,q)\in R(2,8)$. }
\label{fig:C28}
\end{figure}
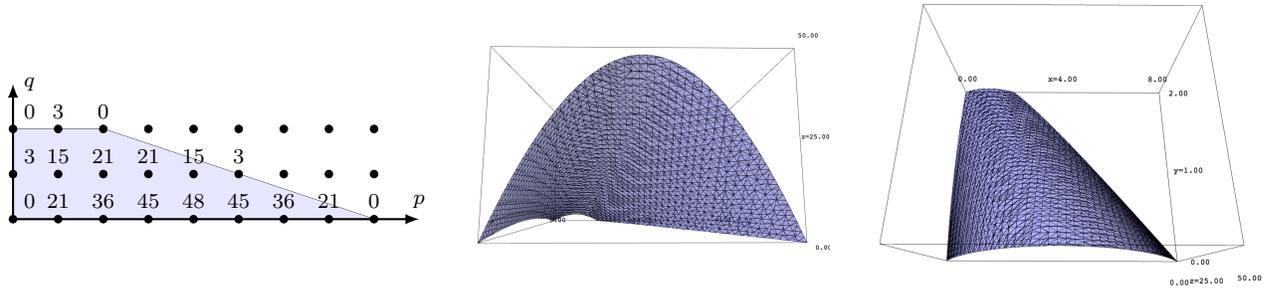

\subsection{An example to illustrate that $\{\chi(L(w))\}$ is identical with the triangular basis}\label{subsection:geometry0110} 
This example is to support Lemma \ref{chiL(w)=C}.
Let  $r\ge2$, $w=(0,1,1,0)$. For this $w$, there are only two possibilities for $v^0$, namely $(0,0)$ and $(1,1)$, such that $(v^0,w)$ is $l$-dominant. The stratification of $\mathbf{E}_{v,w}$ is:
$$\mathbf{E}_{v,w}=\mathbf{E}^\circ_{(0,0),w}\cup \mathbf{E}^\circ_{(1,1),w}=\{0\}\cup(\mathbb{C}^r\setminus\{0\}).$$ 
Since both strata are simply connected (if $r\ge 2$), the local systems appeared in the BBDG decomposition must be trivial.
We can easily check that $a_{v,0;w}$ is nonempty for 
$$v=(0,0),(1,0),(1,1),(1,2),\dots,(1,r),$$
and the corresponding $\tilde{\mathcal{F}}(v,w)$ are 
$$\textrm{a point, a point, $\widetilde{\CC^r}$, the tautological bundle  ${\rm Taut}(Gr_2(\CC^r))$,$\dots$, ${\rm Taut}(Gr_r(\CC^r))$. }$$

Consider $v=(1,i)$. Then $w-C_qv=(-1+ri,0,1-i,-i+r)$.

For $i\ge 1$, the BBDG decomposition is of the following form (note that $\dim {\rm Taut}(Gr_i(\CC^r))=\dim (Gr_i(\CC^r))+r=i(r-i)+i$): 
$$\pi_*(\QQ_{{\rm Taut}(Gr_i(\CC^r))}[-i(r-i)-i])=(A)\QQ_{\CC^r}[-r]\oplus (B)\QQ_{0}.$$ 
Since $\pi^{-1}(x)\cong Gr_{i-1}(\CC^{r-1})$ for $x\neq 0$, we get $A={r-1 \brack i-1}_\v$. Take cohomology over the point $0$, we get 
$${r \brack i}_\v \v^{i}={r-1 \brack i-1}_\v\v^r+B.$$
Therefore
$$a_{(1,i),0;w}=B= {r \brack i}_\v \v^{i} - {r-1 \brack i-1}_\v\v^r = {r-1\brack i}_\v.$$

Then
$$\chi(L(w))=X^{(-1,1)}+X^{(-1,-r+1)}+{r-1\brack 1}_{\v^r} X^{(r-1,r+1)}+\cdots+{r-1\brack r-1}_{\v^r} X^{((r-1)r-1,-r+1)}$$
As we expected, this coincides with the triangular basis element $C[1,r-1]$, which also equals to the greedy element $X[1,r-1]$. It specializes to $\frac{x_2^r+(1+x_1^r)^{r-1}}{x_1x_2^{r-1}}$ when setting $q=1$.

\subsection{An example to illustrate the proof of the main theorem}\label{eg:main thm}

Let $r=3$, we illustrate that for $(a_1,a_2)=(4,8)$, $(p,q)=(2,2)$, our strategy to show that the maximal degree of $e(2,2)$ is $\v^D$ where $D=D(p,q)=r(a_1q+a_2p-p^2-rpq-q^2)=3(4\times2+8\times2-2^2-3\times2\times2-2^2)=12$. 

By Lemma \ref{lem:e=a}, we have $w'_1-w_1=a_1=4$ and $w'_2-w_2=a_2-r[a_1]_+=-4$. Letting $w_1=w'_2=0$, we get $w=(0,4,4,0)$; also, $v=([a_1]_+-q,p)=(2,2)$. Note that $(v,w)$ is $l$-dominant and $w\in\Phi^{im}_+$. 

The following $v'$ satisfies that $(v',w)$ is $l$-dominant and $v'\le v$: $(0,0), (1,1),(1,2),(2,1),(2,2)$. We compute

 $d_{v,w}=v_2(w_2'+rv_1-v_2)+v_1(w_1'-v_1)=2(0+3\times 2-2)+2(4-2)=12$, 

$\tilde{d}_{v,w}=v_2(w_2+w_2'+rv_1-v_2)+v_1(w_1+w_1'-v_1)=2(4+0+3\times2-2)+2(0+4-2)=20$.

So \eqref{eq:sumAP-} becomes
\begin{equation}\label{eq:example4822}
\sum_{v'}a_{v,v';w} P_-(v',w)
=t^{-8}{6\brack 2}_{t}{4\brack 2}_{t}=t^4+\textrm{lower degree terms}
\end{equation}

Below is a table for the each term. Note that all except the first one have max-degree $<4$. Then \eqref{eq:example4822} forces the max-degree term of  $a_{v,0;w}$ to be $t^4$. 

\begin{center}
 \begin{tabular}{|c|c|c|c|} 
 \hline
 $v'$ &$a_{v,v';w}$&$P_-(v,w)$ & max-deg of $a_{v,v';w}P_-(v,w)$ \\ 
 \hline
 $(0,0)$ & $t^4+2t^2+4+2t^{-2}+t^{-4}$ & $1$ &$4$\\
 $(1,1)$ & $t+t^{-1}$ & $t^{-1}+3t^{-3}+\cdots+t^{-9}$& $0$\\
 $(1,2)$ & $0$ & $t^{-3}+2t^{-5}+\cdots+t^{-13}$ & $-\infty$\\
 $(2,1)$ & $t+t^{-1}$ & $t^{-3}+2t^{-5}+\cdots+t^{-13}$ & $-2$\\
 $(2,2)$ & $1$ & $t^{-4}+2t^{-6}+\cdots+t^{-20}$ & $-4$\\
  \hline
\end{tabular}
\end{center}

\subsection{An example to Corollary \ref{cor:BBDG support Nakajima}}\label{example for 1.5}

We illustrate Corollary \ref{cor:BBDG support Nakajima} by a special example where the morphism $\pi: {\mytildeF_{v,w}}\to \mathbf{E}_{v,w}$ is not necessarily birational, and we shall examine the generic fiber. 

Consider $w=(0,w'_1,w_2,0)$ with $rw_2\ge w'_1$ and $rw'_1\ge w_2$ (that is, Case 1 in Figure \ref{figure:l-dominant}). We assume $v=(v_1,v_2)$ satisfying $0\le v_1\le w'_1$ and $w_2\le v_2\le rv_1$. Let $v'=\bar{v}=(v_1,w_2)$ (Lemma \ref{lemma:vbar}). By \eqref{a1a2pq}, $a_1=w'_1-rw_2\le0$, $a_2=rv_1-w_2\ge0$, $p=v_2-w_2$, $q=0$. The triangular basis element $C[a_1,a_2]$ is a cluster monomial, so we have the following (for example, by \cite[Corollary 3.4]{LLRZ2}):
$$e(p,0)={a_2\brack p}_{\v^r}={rv_1-w_2\brack v_2-w_2}_{\v^r}$$ 

On the other hand, we compute $a^d_{v,v';w}$ directly by considering the geometry. The morphism $\pi: {\mytildeF_{v,w}}\to \mathbf{E}_{v,w}=\mathbf{E}_{v',w}$ is a $\mathcal{M}$-bundle when restricting over the stratum $\mathbf{E}^\circ_{v',w}$,  
where the generic fiber $\mathcal{M}$ is a 
$Gr(v_2-v'_2,w_2'+rv_1-v'_2)$-bundle over $Gr(v_1-v'_1,w_1'-v'_1)$ by Proposition \ref{prop:fiber} (b). 
These two Grassmannians are $Gr(v_2-w_2,rv_1-w_2)$ and $Gr(0,w'_1-v_1)$ respectively. So $\mathcal{M}$ is isomorphic to $Gr(v_2-w_2,rv_1-w_2)$. 
Since we know that the local systems appeared in the BBDG decomposition must be trivial,  $a^d_{v,v';w}$ are determined by the homology of the fiber $Gr(v_2-w_2,rv_1-w_2)$, more precisely,
$$\sum_d a^d_{v,v';w}t^d={rv_1-w_2\brack v_2-w_2}_t.$$

Thus $e(p,0)={rv_1-w_2\brack v_2-w_2}_{\v^r}=\sum_d a^d_{v,v';w}\v^{rd}$, as predicted by Corollary  \ref{cor:BBDG support Nakajima}.

\section{Some technical proofs}
To ensure that the main theorem's argument is easily readable, the proofs for several statements in sections \S4 and \S5 have been deferred to this section.

\subsection{Proof of Lemma \ref{lem:ab}}
(a) $\myE^\circ_{v',w}\subseteq \myE_w$ is locally closed because it is the difference of closed subsets:
$$\aligned
\myE^\circ_{v',w} 
=&\{({\bf x},{\bf y})\in\mathbf{E}_w\, |\, 
\mathrm{rank}A\le v'_1, \mathrm{rank} B\le v_2' \}
-\{({\bf x},{\bf y})\in\mathbf{E}_w\, |\, 
\mathrm{rank}A\le v'_1-1, \mathrm{rank} B\le v_2' \}\\
&-\{({\bf x},{\bf y})\in\mathbf{E}_w\, |\, 
\mathrm{rank}A\le v'_1, \mathrm{rank} B\le v_2'-1 \}
\endaligned
$$

To show $\myE^\circ_{v',w}$ is nonsingular, we consider the following open covering
\begin{equation}\label{UIJ}
\myE^\circ_{v',w}=\bigcup U_{I,J}
\end{equation}
where $I\subseteq\{1,\dots,w_1'\}$ and $|I|=v_1'$, $J\subseteq\{1,\dots,w_2\}$ and $|J|=v_2'$, $U_{I,J}\subseteq \myE^\circ_{v',w}$ consists of those $({\bf x},{\bf y})$ such that the $v_1'$ rows of $A$ indexed by $I$ are linearly independent, and the $v_2'$ columns of $B({\bf x},{\bf y})$ indexed by $J$ are linearly independent. 
We will show that 
$U_{I,J}$ is isomorphic to an open subset $\tilde{\mathbb{A}}^\circ$ of an affine space $\tilde{\mathbb{A}}$,
hence is nonsingular, which implies that $\myE^\circ_{v',w}$ is nonsingular. 

Without loss of generality, we can assume $I=\{1,\dots,v_1'\}$ and $J=\{1,\dots,v_2'\}$. 
Denote $$\tilde{\mathbb{A}}=\mathbb{A}^{(w'_1-v_1')\times v_1'}\times \mathbb{A}^{v_2'\times (w_2-v_2')}\times\mathbb{A}^{v_1'\times w_1}\times \mathbb{A}^{w_2'\times v_2'}\times (\mathbb{A}^{v_1'\times v_2'})^{\oplus r}$$ and write elements in this affine space as $(P,Q,Z_1',Z_2',C_1,\dots,C_r)$ where each component is a matrix of indicated size. Then define a morphism
$$\aligned
&\varphi: \mathbb{A}^{(w'_1-v_1')\times v_1'}\times \mathbb{A}^{v_2'\times (w_2-v_2')}\times\mathbb{A}^{v_1'\times w_1}\times \mathbb{A}^{w_2'\times v_2'}\times (\mathbb{A}^{v_1'\times v_2'})^{\oplus r} \to \myE_w\\
&(P,Q,Z_1',Z_2',C_1,\dots,C_r)\mapsto 
({\bf x},{\bf y})=\Big(
x_1=\begin{bmatrix}Z_1'\\PZ_1'\end{bmatrix},
x_2=\begin{bmatrix}Z_2'&Z_2'Q\end{bmatrix},
y_i=\begin{bmatrix}C_i&C_iQ\\ PC_i&PC_iQ\end{bmatrix}
\Big).
\endaligned
$$
Note that ${\rm im}\varphi\subseteq\myE_{v',w}$. Indeed, define
$$A'=[Z_1',C_1,\dots,C_r]_\text{hor},\quad B'=[Z'_2,C_1,\dots,C_r]_\text{vert}$$
Then letting
$$x_1=\begin{bmatrix}Z'_1\\PZ'_1\end{bmatrix}, \quad 
x_2=\begin{bmatrix}Z'_2 & Z'_2Q\end{bmatrix}, \quad
y_i=\begin{bmatrix}C_i&C_iQ\\ PC_i& PC_iQ\end{bmatrix} \quad (1\le i\le r),
$$
the corresponding matrices
\begin{equation}\label{AB in ZCMN}
A=[x_1\ y_1\ \cdots y_r]=\begin{bmatrix}Z'_1&C_1&C_1Q&\cdots&C_r&C_rQ\\ PZ'_1&PC_1&PC_1 Q&\cdots&PC_r&PC_rQ
\end{bmatrix},
\quad
B=\begin{bmatrix}x_2\\ y_1\\ \vdots \\y_r \end{bmatrix}
=\begin{bmatrix}Z'_2& Z'_2Q\\ C_1&C_1Q \\ PC_1& PC_1Q \\ \vdots &\vdots\\ C_r& C_rQ \\ PC_r & PC_rQ
\end{bmatrix} 
\end{equation}
satisfy 
${\rm rank}A={\rm rank} A'\le v'_1$, ${\rm rank}B={\rm rank} B'\le v'_2$. 
Therefore ${\rm im}\varphi\subseteq\myE_{v',w}$.

Define $\tilde{\mathbb{A}}^\circ$ to be the following open subset of $\tilde{\mathbb{A}}$:
$$\aligned
\tilde{\mathbb{A}}^\circ=\Bigg\{(P,Q,Z'_1,Z'_2,C_1,\dots,C_r)\in \tilde{\mathbb{A}} \; \Bigg|\;  
&\textrm{$A'=[Z'_1\ C_1\  \cdots \ C_r]_\text{hor}$ has full row rank $v'_1$, and}\\
&\textrm{$B'=[Z'_2,C_1,\cdots, C_r]_\text{vert}$ has full column rank $v'_2$.}\Bigg\}
\endaligned
$$
Then $\varphi$ restricts to a morphism $\varphi|_{\tilde{\mathbb{A}}^\circ}: \ \tilde{\mathbb{A}}^\circ\to U_{I,J}$.

We claim that $\varphi|_{\tilde{\mathbb{A}}^\circ}$ is an isomorphism. To prove this, it suffices to construct the inverse morphism $\varphi|_{\tilde{\mathbb{A}}^\circ}^{-1}:U_{I,J}\to \tilde{\mathbb{A}}^\circ$ by sending $(x_1,x_2,y_1,\dots,y_r)$ to $(P,Q,Z'_1,Z'_2,C_1,\dots,C_r)$ defined as follows. For a matrix $X$, denote $X_{[1,\dots,a;1,\dots,b]}$ the submatrix of $X$ by selecting the first $a$ rows and first $b$ columns. Define $Z'_1={x_1}_{[1,\dots,v'_1;1,\dots,w_1]}$, $Z'_2={x_2}_{[1,\dots,w'_2;1,\dots,v'_2]}$, $C_i={y_i}_{[1,\dots,v'_1;1,\dots,v'_2]}$. By the definition of $U_{I,J}$, the top $v'_1$ rows of $A$ are linearly independent, thus $A'$ is of full rank $v'_1$. Denote $\begin{bmatrix}A' \\ A'' \end{bmatrix}$ the $w'_1\times (w_1+rv'_2)$-submatrix of $A$ with $A'$ as the top block.  Define $P$ to be the unique matrix satisfying $PA'=A''$. Since $A'$ has full row rank, it has a right inverse $A'^T(A'A'^T)^{-1}$. Thus, $P=A''A'^T(A'A'^T)^{-1}$ whose entries are rational functions that are regular on $U_{I,J}$. 
The matrix $Q$ is defined similarly and has similar properties. Therefore $\varphi|_{\tilde{\mathbb{A}}^\circ}^{-1}$ is a morphism. It is easy to see that  $\varphi|_{\tilde{\mathbb{A}}^\circ}^{-1}$ is indeed the inverse of $\varphi|_{\tilde{\mathbb{A}}^\circ}$.
This implies that $\myE^\circ_{v',w}$ is nonsingular. 

\smallskip
The nonemptyness statement is proved in \cite[Proposition 4.23]{Qin}. Here we give an alternative proof. 
First, assume $\myE^\circ_{v',w}$ is nonempty. Then 
\begin{itemize}
\item $w_1'-v'_1\ge0$ follows from  $v'_1={\rm rank}A\le\textrm{(number of rows of $A$)}=w_1'$. 

\item $w_1-v_1'+rv_2'\ge0$ follows from $v_1'={\rm rank}A ={\rm rank} [x_1 \ y_1\ \cdots \ y_r]\le  \textrm{rank}x_1+\sum_{i=1}^r\textrm{rank} y_i\le w_1+\sum_{i=1}^r{\rm rank } B=w_1+rv_2'$.

\item $w_2-v_2'\ge0$ follows from  $v_2'={\rm rank}B\le\textrm{(number of columns of $B$)}=w_2$. 

\item $w_2'-v_2'+rv_1'\ge0$ follows from $v_2'={\rm rank}B\le \textrm{rank}x_2+\sum_{i=1}^r\textrm{rank} y_i\le w'_2+\sum_{i=1}^r{\rm rank } A=w'_2+rv_1'$.
\end{itemize}
Therefore $v'\in\D(w)$. 
Conversely, assume $v'\in\D(w)$. Let $I=\{1,\dots,v_1'\}$ and $J=\{1,\dots,v_2'\}$, we shall show that $U_{I,J}\neq \emptyset$.  A generic choice of $Z_1', Z_2', C_1,\dots,C_r$  will make the matrices $A'$ and $B'$ full-rank (here we used the fact that the base field $\mathbb{C}$ is infinite),   so   $\rank A'=v_1'$ and $\rank B'=v_2'$ thanks to the $l$-dominant assumption. After choosing arbitrary $P,Q$, we obtain a point in $\tilde{\mathbb{A}}^\circ$. Since  $\tilde{\mathbb{A}}^\circ\cong U_{I,J}\subseteq\myE^\circ_{v',w}$, we conclude that $\myE^\circ_{v',w}$ is nonempty.

Now assume that $\myE^\circ_{v',w}$ is nonempty, and we shall prove that this nonsingular variety is irreducible. For this, it suffices to find a common point $({\bf x},{\bf y})$ in all $U_{I,J}$. Indeed, a generic choice of $(P,Q,Z_1', Z_2', C_1,\dots,C_r)$  will give such a point. To see this, let $A_{[I;-]}$ be the submatrix of $A$ by selecting the rows with indices in $I$. Let $\tilde{P}=[I,P]_\text{vert}$. Then
$$A_{[I;-]}=\tilde{P}_{[I;-]}\; [Z'_1, C_1,C_1Q,\dots,C_r,C_rQ]_\text{hor}$$
where the right side is the product of two matrices.
Since $P$ is generic, we can assume $\tilde{P}_{[I;-]}$ is of full rank $v'_1$ (which equals its number of columns). Also $[Z'_1, C_1,C_1Q,\dots,C_r,C_rQ]_\text{hor}$ is of full rank $v'_1$ (which equals its number of rows). It follows that $A_{[I;-]}$ has rank $v'_1$. Similarly, $B_{[-;J]}$, the submatrix of $B$ by selecting the columns with indices in $J$, has rank $v'_2$. Therefore $({\bf x},{\bf y})$ is in $U_{I,J}$.

Since $\myE^\circ_{v',w}$ is covered by open subsets $U_{I,J}$ of affine spaces and is irreducible, it is rational. 
\medskip

(b) To prove that  $\pi^{-1}(\myE^\circ_{v',w})\to \myE^\circ_{v',w}$ is a locally trivial bundle,  we show that it is a pullback of another locally trivial bundle $\mathcal{X}\to \mathcal{F}_{v',w}$ that is easier to deal with. Here is the diagram:
$$\xymatrix{\pi^{-1}(\myE^\circ_{v',w}) \ar[d]^\pi\ar[r]& \mathcal{X}\ar[d]^q\\ \myE^\circ_{v',w}\ar[r]^f & \mathcal{F}_{v',w}}$$
where 
$$\aligned
&\mathcal{X}=\{(Y_1,Y_2,X_1,X_2) | (Y_1,Y_2)\in \mathcal{F}_{v',w}, (X_1,X_2)\in\mathcal{F}_{v,w}, Y_1\subseteq X_1, Y_2\subseteq X_2\}\\
& f: ({\bf x},{\bf y})\mapsto (Y_1,Y_2)=(\text{im}A({\bf x},{\bf y}),\text{im}B({\bf x},{\bf y}))\\
& q: (Y_1,Y_2,X_1,X_2) \mapsto (Y_1,Y_2)
\endaligned
$$ 

We claim that the natural projection $q:\mathcal{X}\to\mathcal{F}_{v',w}$ gives a locally trivial fiber bundle with fibers isomorphic to $\mathcal{M}$. 
To prove the claim, note that $\mathcal{F}_{v',w}$ can be covered by open subsets over which the direct sum of universal subbundles $S_{Gr(v'_1,W'_1)}\oplus S_{Gr(v'_2,W'_2\oplus {W'_1}^{\oplus r})}$ is trivial, so we only need to discuss fiberwisely. For fixed $Y_1,Y_2$, we show that the set $(X_1,X_2)$ satisfying $Y_1\subseteq X_1\subseteq W'_1$ and $Y_2\subseteq X_2\subseteq W'_2\oplus X_1^{\oplus r}$ is the variety $\mathcal{M}$. Indeed, choose  a complement $Y_1^c\subseteq W'_1$ to $Y_1$, then $X_1\mapsto X'_1=X_1\cap Y_1^c$ gives an isomorphism between $\{X_1| Y_1\subseteq X_1\subseteq W'_1\}$ and $Gr(v_1-v'_1,Y_1^c)\cong Gr(v_1-v'_1,w'_1-v'_1)$. Choose a complement $Y_2^c\subseteq W'_2\oplus Y_1^{\oplus r}$ to $Y_2$, then $\widetilde{Y}_2^c=Y_2^c+(0\oplus(X'_1)^{\oplus r})$ is a complement to $Y_2$ as subspaces of $W'_2\oplus X_1^{\oplus r}$, and
$X_2\mapsto X'_2=X_2\cap \widetilde{Y}_2^c$ gives an isomorphism between $\{X_2|Y_2\subseteq X_2\subseteq W'_2\oplus X_1^{\oplus r}\}$ and 
$$Gr(v_2-v'_2,\widetilde{Y}_2^c)\cong Gr(v_2-v'_2, Y_2^c+(0\oplus (X'_1)^{\oplus r}))\cong Gr(v_2-v'_2,\mathbb{C}^{w'_2+rv'_1-v'_2}\oplus (X'_1)^{\oplus r})$$

Now observe that  $\pi^{-1}(\myE^\circ_{v',w})$ is isomorphic to the pullback $f^*(\mathcal{X})$. Indeed, the fiber of $f^*(\mathcal{X})$ over a point $({\bf x},{\bf y})\in\myE^\circ_{v',w}$ is isomorphic to the fiber of $q:\mathcal{X}\to\mathcal{F}_{v',w}$ over $f({\bf x},{\bf y})=(\text{im}A({\bf x},{\bf y}),\text{im}B({\bf x},{\bf y}))$, which by definition of $\mathcal{X}$ is the set of $(X_1,X_2)$ satisfying $(X_1,X_2)\in\mathcal{F}_{v,w}$, $\text{im}A({\bf x},{\bf y})\subseteq X_1,\text{im}B({\bf x},{\bf y})\subseteq X_2$. This coincides with the conditions in the definition of $\mytildeF_{v,w}$.

Since the pullback of a locally trivial fiber bundle is again a locally trivial fiber bundle with the same fiber, (b) is proved.

This completes the proof of Lemma \ref{lem:ab}.

\subsection{Proof of Lemma \ref{lem:transversal slice}}
(1) We will define $\psi$. To fix notations, for a point $q=({\bf x},{\bf y})$ in the domain $U$ of $\psi$, denote $(q^0,q^\perp)=\psi(q)$, and denote 
$q^0=({\bf x}^0,{\bf y}^0)=(x_1^0,x_2^0,y_1^0,\dots,y_r^0)\in U^0$,  
$q^\perp=({\bf x}^\perp,{\bf y}^\perp)=(x_1^\perp,x_2^\perp,y_1^\perp,\dots,y_r^\perp)\in U^\perp$.
Recall that the point $p$ determines two matrices as in \eqref{AB}, denoted $A_p$ and $B_p$. Without loss of generality, assume the first $v_1^0$ row vectors of $A_p$ (resp.~the first $v_2^0$ column vectors of $B_p$) are linearly independent. Then  $A_p$ and $B_p$ must be in the form of
\eqref{AB in ZCMN}, where the the sizes of the blocks are
\begin{center}
 \begin{tabular}{|c|c|c|c|c|c|} 
 \hline
block & $Z_1'$ & $Z_2'$ & $C_i$ & $P$ & $Q$\\
 \hline
size   & $v_1^0\times w_1$ & $w'_2\times v_2^0$ & $v_1^0\times v_2^0$ & $(w_1'-v_1^0)\times v_1^0$ & 
$v_2^0\times (w_2-v_2^0)$\\
 \hline
\end{tabular}
\end{center}

First, we define the open set $U$ and $U^0$. Define the set  $$M=\{1,\dots,w_1\}\cup\bigcup_{i=1}^r \{w_1+(i-1)w_2+1, \dots, w_1+(i-1)w_2+v_2^0\}.$$ 
The submatrix $A'_p=[Z'_1,C_1,  \cdots , C_r]_\text{hor}=(A_p)_{[1,\dots,v_1^0;M]}$ satisfies $\rank A'_p=\rank A_p=v_1^0$, thus there exists a subset  $J\subseteq M$ of cardinality $v_1^0$ and $(A_p)_{[1,\dots,v_1^0;J]}$ is an invertible $v_1^0\times v_1^0$-matrix. 
Similarly, 
define the set $$N=\{1,\dots,w_2'\}\cup\bigcup_{i=1}^r \{w_2'+(i-1)w_1'+1,   \dots, w_2'+(i-1)w_1'+v_1^0\}.$$
There is a subset
$I\subseteq N$
of cardinality $v_2^0$ such that $(B_p)_{[I; 1,\dots,v_2^0]}$ is an invertible $v_2^0\times v_2^0$ -matrix. 
 Define an open subset $U$ of  $\myE_{v,w}$ as
 $$U=\{ q\in \myE_{v,w}\ | \ \textrm{ $(A_q)_{[1,\dots,v_1^0;J]}$ and  $(B_q)_{[I; 1,\dots,v_2^0]}$ are invertible square matrices}\}.$$ 
 Define $U^0=U\cap  \myE_{v^0,w}$ which is an open subset of  $\myE^\circ_{v^0,w}$ because $q\in U$ implies $\rank A_q\ge v_1^0$ and $\rank B_q\ge v_2^0$, and furthermore the equalities must hold by the condition $q\in\myE_{v^0,w}$.
 \medskip
 
Second,  we construct $q^0$ as follows. 
Denote
$$A_q=[x_1\ y_1\ \cdots y_r]=\begin{bmatrix}Z'_1&C_1&C_1'&\cdots&C_r&C_r'\\ Z_1''&C_1''&C_1''' &\cdots&C_r''&C_r'''
\end{bmatrix},
\quad
B_q=\begin{bmatrix}x_2\\ y_1\\ \vdots \\y_r \end{bmatrix}
=\begin{bmatrix}Z'_2& Z''_2\\ C_1&C_1' \\ C_1''& C_1''' \\ \vdots &\vdots\\ C_r& C_r' \\ C_r'' & C_r'''
\end{bmatrix} 
$$
where the sizes of the blocks are:
\begin{center}
 \begin{tabular}{|c|c|c|c|c|} 
 \hline
block & $Z_1'$ & $Z_1''$ &$Z_2'$ &$Z_2''$ \\
 \hline
size   & $v_1^0\times w_1$ & $(w_1'-v_1^0)\times w_1$ 
&$w'_2\times v_2^0$
&$w'_2\times(w_2-v_2^0)$
\\
 \hline
\end{tabular}
\vspace{10pt}

 \begin{tabular}{|c|c|c|c|c|} 
 \hline
block &  $C_i$ &$C_i'$ & $C_i''$ & $C_i'''$\\
 \hline
size 
& $v_1^0\times v_2^0$ & $v_1^0\times(w_2-v_2^0)$ & 
$(w_1'-v_1^0)\times v_2^0 $ &$(w_1'-v_1^0)\times (w_2-v_2^0)$\\
 \hline
\end{tabular}
\end{center}

Since $(A_q)_{[1,\dots,v_1^0;J]}$ and  $(B_q)_{[I; 1,\dots,v_2^0]}$ are invertible, there are unique $(w_1'-v_1^0)\times v_1^0$ -matrix $P$ and $v_2^0\times (w_2-v_2^0)$ -matrix $Q$ such that 
\begin{equation}\label{AJBI}
(A_q)_{[-;J]}=\begin{bmatrix}(A_q)_{[1,\dots,v_1^0;J]}\\P (A_q)_{[1,\dots,v_1^0;J]}\end{bmatrix}, \quad 
(B_q)_{[I;-]}=\bigg[ (B_q)_{[I;1,\dots,v_2^0]}\;\; (B_q)_{[I;1,\dots,v_2^0]}Q\bigg].
\end{equation}
Define $q^0$ such that 
$$A_{q^0}=\begin{bmatrix}
Z'_1&C_1&C_1Q&\cdots&C_r&C_rQ\\ 
PZ_1'&PC_1& PC_1Q &\cdots&PC_r&PC_rQ
\end{bmatrix},
\quad
B_{q^0}=\begin{bmatrix}
Z_2'&Z_2'Q\\
C_1&C_1Q\\
PC_1&PC_1Q\\
\vdots&\vdots\\
C_r&C_rQ\\
PC_r&PC_rQ\\
\end{bmatrix}.$$
That is, 
$x_1^0=\begin{bmatrix}Z'_1\\PZ'_1\end{bmatrix}$,
$x_2^0=\begin{bmatrix}Z'_2&Z'_2Q\end{bmatrix}$,
$y_i^0=\begin{bmatrix}C_i&C_iQ\\PC_i&PC_iQ\end{bmatrix}$ ($1\le i\le r$).
It is easy to see that $q^0$ is in $U^0$. (To check that $q_0\in \myE_{v^0,w}$, note that 
$$\rank A_{q^0}=
\rank( \begin{bmatrix}{\bf I}\\ P\end{bmatrix} \begin{bmatrix}Z_1'& C_1& \cdots & C_r\end{bmatrix} \begin{bmatrix} {\bf I} &0&0&\dots&0&0\\ 0&{\bf I}&Q&\dots&{\bf I}&Q\end{bmatrix})
=\rank [Z_1', C_1, \cdots, C_r]_{\rm hor}\le v_1^0,$$ 
thus the equality holds since the first $v_1^0$ rows of $A_q$ are linearly independent; 
similarly $\rank B_{q^0}=v_2^0$.)

Third,  we construct $q^\perp$ as follows. Let
$$\widetilde{A_q}=[\widetilde{P}x_1,  \widetilde{P}y_1\widetilde{Q}, \cdots, \widetilde{P}y_r\widetilde{Q}]_{\rm hor},
\text{ where }
\widetilde{P}=\begin{bmatrix}{\bf I}_{v_1^0}&0\\ -P&{\bf I}_{w_1'-v_1^0}\end{bmatrix},\quad  
\widetilde{Q}=\begin{bmatrix}{\bf I}_{v_2^0}&-Q\\ 0&{\bf I}_{w_2-v_2^0}\end{bmatrix}. 
$$
Then $\rank \widetilde{A_q}=\rank A_q=u_1$, and the submatrix $(\widetilde{A_q})_{[1,\dots,v_1^0; J]}= (A_q)_{[1,\dots,v_1^0;J]}$, while the submatrix $(\widetilde{A_q})_{[v_1^0+1,\dots,w'_1; J]}=0$. (To see this: for all $j\in J$, let ${\bf v}$ be the $j$-th column of the matrix. If $j\in \{1,\dots,w_1\}$, then ${\bf v}$ is the $j$-column of $Px_1$, and by \eqref{AJBI}, the $j$-column of $x_1$ is of the form $\begin{bmatrix}{\bf u}\\P{\bf u}\end{bmatrix}$, so 
${\bf v}=\widetilde{P}\begin{bmatrix}{\bf u}\\P{\bf u}\end{bmatrix}
=\begin{bmatrix}{\bf I}&0\\ -P&{\bf I}\end{bmatrix}\begin{bmatrix}{\bf u}\\ P{\bf u}\end{bmatrix}
=\begin{bmatrix}{\bf u}\\ 0\end{bmatrix}
$. In the other case, assume $j\in \{w_1+(i-1)w_2+1,\dots,w_1+(i-1)w_2+v_2^0\}$ for some $1\le i\le r$, note that  
$\widetilde{P}y_i\widetilde{Q}
=\begin{bmatrix}{\bf I}&0\\ -P&{\bf I}\end{bmatrix}
\begin{bmatrix} C_i&C_i'\\ C_i''&C_i'''\end{bmatrix}
\begin{bmatrix}{\bf I}&-Q\\ 0&{\bf I}\end{bmatrix}
=\begin{bmatrix}C_i&*\\ -PC_i+C_i''&*\end{bmatrix}
$. Since the $j$-th column of $A_q$ is of the form 
$\begin{bmatrix}{\bf u}\\ P{\bf u}\end{bmatrix}$, 
the vector ${\bf v}$ is of the form 
$\begin{bmatrix}{\bf u}\\ -P{\bf u}+P{\bf u}\end{bmatrix} = \begin{bmatrix}{\bf u}\\ 0\end{bmatrix}$.)

Note that the space spanned by the top $v_1^0$ rows of $\widetilde{A_q}$ has zero intersection with the space spanned by the rest rows by considering the columns indexed by $J$.
Since  $\rank\widetilde{A_q}=\rank A_q\le v_1$, and that 
$\rank (\widetilde{A_q})_{[1,\dots,v_1^0;-]}=v_1^0$ (because its $J$ columns is equal to the invertible matrix $(A_q)_{[1,\dots,v_1^0;J]}$),
we have
$$\rank (\widetilde{A_q})_{[v_1^0+1,\dots,w'_1; -]} = \rank \widetilde{A_q}-\rank (\widetilde{A_q})_{[1,\dots,v_1^0;-]}
=u_1^\perp\le v_1^\perp.$$
 We define $A_q^\perp$ to be 
$(\widetilde{A_q})_{[v_1^0+1,\dots,w_1'; \{1,\dots,w_1+rw_2\}\setminus J]}$ with an appropriate rearrangement of the columns; to be more precise:
$$A_q^\perp=\Big[  (\widetilde{A_q})_{[v_1^0+1,\dots,w_1'; M\setminus J]} \ \vline \  PC_1Q-PC_1'-C_1''Q+C_1''' \ \vline \  \cdots \Big]$$
That is, if we denote $J'=\iota_M^{-1}(J)$  where $\iota_M:\{1,\dots, w_1+rv_2^0\}\to M$ is the order preserving bijection, then 
$$\aligned
&x_1^\perp= \text{obtained from $[-PZ_1'+Z_1'', -PC_1+C_1'', \cdots, -PC_r+C_r'']_{\rm hor}$ }\\
&\hspace{30pt} \text{by deleting  columns $J'$ (which are zero)},\\ 
&y_i^\perp=PC_iQ-PC_i'-C_i''Q+C_i''' \quad (1\le i\le r). 
\endaligned
$$
Note that the column spaces of $A_q^\perp$ and $(\widetilde{A_q})_{[v_1^0+1,\dots,w'_1; -]}$ are the same, which implies
\begin{equation}\label{eq:rank Aqperp}
\rank A_q^\perp= u_1^\perp
\end{equation}

Similarly, we define 
$$\aligned
&x_2^\perp = \text{obtained from $[-Z_2'Q+Z_2'', -C_1Q+C_1', \cdots, -C_rQ+C_r']_{\rm vert}$}\\
&\hspace{30pt} \text{by deleting rows $I'$ (which are zero)},
\endaligned
$$ 
where  $I'=\iota_N^{-1}(I)$  where $\iota_N:\{1,\dots, w_2'+rv_1^0\}\to N$ is the order preserving bijection. 
Like \eqref{eq:rank Aqperp}, we have
\begin{equation}\label{eq:rank Bqperp}
\rank B_q^\perp= u_2^\perp
\end{equation}

This complete the construction of $q^\perp=(x_1^\perp,x_2^\perp,y_1^\perp,\dots,y_r^\perp)$.

Note that $\psi$ is a morphism because it can be expressed in terms of matrix additions, multiplications, and inverses, so can be expressed as rational functions. 

\medskip
\noindent (2) 
We explain that $\psi$ is an isomorphism by constructing its inverse $\psi^{-1}$. Assume $q^0,q^\perp$ are given. 
From $q^0$ we can recover $Z_1', Z_2',C_i, P, Q$ by rational functions. 
From $q^\perp$ we can recover $Z_1'', Z_2'', C_i', C_i'', C_i'''$ by rational functions. 
Thus $\psi^{-1}$ is a morphism and it is easy to see that it is indeed the two-sided inverse of $\psi$. Thus $\psi$ is an isomorphism.

\medskip
\noindent (3) 
We define $\varphi$. Given $(q=({\bf x},{\bf y}),X_1,X_2)\in \pi^{-1}U$, define 
$\varphi(q,X_1,X_2)=(q^0,(q^\perp,X_1^\perp,X_2^\perp))$, where $q^0$ and $q^\perp$ are defined in (1), and $X_1^\perp,X_2^\perp$ are defined as follows:

For ${\bf u}\in X_1$, define
${\bf u}^0 = \begin{bmatrix}{\bf I} &0\\ P&0\end{bmatrix} {\bf u}$  
(where $P$ is determined by $q$),
${\bf u}^\perp = {\bf u}-{\bf u}^0= \begin{bmatrix}0 &0\\ -P&{\bf I}\end{bmatrix} {\bf u} $ has all its top $v_1^0$ entries being 0. Define
$$X_1^0=\{{\bf u}^0\ | \ {\bf u}\in X_1\}, \quad 
X_1^\perp=\{{\bf u}^\perp\ | \ {\bf u}\in X_1\}
=\Big\{\begin{bmatrix} 0\\-P\alpha+\beta\end{bmatrix}\; \Big| \; \begin{bmatrix}{\alpha}\\{\beta}\end{bmatrix}\in X_1\Big\}$$
Then $\dim X_1^0=v_1^0$ (since the top $v_1^0$ rows of $A_q$ are linearly independent which implies $\dim X_1^0\ge v_1^0$; meanwhile,  $\dim X_1^0\le v_1^0$ because $\rank \begin{bmatrix}{\bf I} &0\\ P&0\end{bmatrix}=v_1^0$) , $\dim X_1^\perp=\dim X_1-\dim X_1^0=v_1-v_1^0=v_1^\perp$. Note that $X_1^\perp=X_1\cap (0\oplus \mathbb{C}^{w_1-v_1^0})$ is the set of vectors in $X_1$ with the first $v_1^0$ entries being 0.

We define $X_2^\perp$ as follows. 
$$
X_2^\perp=
\Bigg\{\gamma=\begin{bmatrix}\gamma_1\\ \gamma_2 \\ \vdots \\ \gamma_{w_2'+rw_1'}\end{bmatrix}
\in 
\widetilde{X}_2= P^*X_2,  
\ {\textrm {such that} } \
\gamma_i=0 \; {\rm for }\;  i\in I
\Bigg\}, 
\;
P^*=\begin{bmatrix}
{\bf I}&0&\cdots&0\\0&\widetilde{P}&\cdots& 0\\ \vdots&\vdots&\ddots&\vdots\\0&0&\cdots&\widetilde{P}
\end{bmatrix}
$$
We now show that $\dim X_2^\perp=v_2^\perp$. Note that $\dim \widetilde{X}_2=\dim X_2=v_2$. The $I$-rows of $\widetilde{X}_2$ is the same as the $I$-rows of $X_2$, so they have the full row rank $v_2^0$. So there is a subspace $V$ of $\widetilde{X}_2$  of dimension $v_2^0$ whose $I$-rows are full rank. Then $V\cap X_2^\perp=0$ and $V+X_2^\perp=\widetilde{X}_2$, so $\dim X_2^\perp= \dim \widetilde{X}_2-\dim V=v_2-v_2^0=v_2^\perp$. 

Next we shall show that $X_2^\perp\subseteq W_2'^\perp\oplus (X_1^\perp)^{\oplus r}$, where $W_2'^\perp\cong \mathbb{C}^{w_2'+rv_1^0-v_2^0}$ consists of vectors whose $i$-th coordinates are 0 except for $i\in N\setminus I$. 
Assume $P^*u\in X_2^\perp$ for $u=\begin{bmatrix} u_0\\ u_1\\ \vdots \\u_r\end{bmatrix}\in X_2$ where $u_0\in W'_1$, $u_i\in X_1$ for $i=1,\dots,r$, we have 
$$
P^*u
=\begin{bmatrix}u_0\\ \widetilde{P}u_1\\ \vdots \\ \widetilde{P}u_r \end{bmatrix}
=\begin{bmatrix}u_0\\ \begin{bmatrix} {\bf I} &0\\ 0&0\end{bmatrix} u_1\\ \vdots \\ \begin{bmatrix} {\bf I} &0\\ 0&0\end{bmatrix}  u_r \end{bmatrix}
+
\begin{bmatrix}u_0\\ \begin{bmatrix} 0 &0\\ -P&{\bf I}\end{bmatrix} u_1\\ \vdots \\ \begin{bmatrix} 0 &0\\ -P&{\bf I}\end{bmatrix}  u_r \end{bmatrix}
$$ 
where the second vector is in $(X_1^\perp)^{\oplus r}$; the first vector has zero entries at coordinates except $N\setminus I$, so is in $W_2'^{\perp}$. This proves $X_2^\perp\subseteq W_2'^\perp\oplus (X_1^\perp)^{\oplus r}$.

\medskip
\noindent (4) To show $\varphi$ is an isomorphism, we observe that its inverse $\varphi^{-1}$ can be constructed as follows: 
given $(q^0,(q^\perp,X_1^\perp,X_2^\perp))$, to construct $(q,X_1,X_2)$, we get $q$ by $\psi^{-1}$, and define
$$\aligned
&X_1
=(\text{ column space of } A_{q^0})+X_1^\perp
=(\text{ column space of } \begin{bmatrix}{\bf I}\\ P\end{bmatrix})+X_1^\perp\\
&X_2
=(\text{ column space of } B_{q^0})+(P^*)^{-1}X_2^\perp\\
\endaligned
$$
In the construction of $X_1$ (resp. $X_2$), the sum is an internal direct sum since the first $v_1^0$ entries of vectors in $X_1^\perp$ are 0 (resp. the $I$-entries of vectors in $(P^*)^{-1}X_2^\perp$ are 0). 
From the construction of $X_1^\perp$ and $X_2^\perp$ it is easy to see that $\phi$ and $\phi^{-1}$ are indeed inverse to each other.

\medskip
\noindent (5) To check that the diagram commutes: 
$$\psi\circ\pi(q,X_1,X_2)=\psi(q)=(q^0,q^\perp)=(1\times\pi)(q^0,(q^\perp,X_1^\perp,X_2^\perp))=(1\times\pi)\circ\varphi(q,X_1,X_2)$$

\medskip
\noindent (6) To check $\psi(p)=(p,0)$: let $q=p$, and use the same assumption as above, we need to show $q^0=q$ and $q^\perp=0$. Since $A_q$ has rank $v_1^0$, all other rows are linear combinations of its first $v_1^0$ rows, that is, $A_q=\begin{bmatrix} (A_q)_{[1,\dots,v_1^0;-]}\\  P' (A_q)_{[1,\dots,v_1^0;-]}\\  \end{bmatrix}$ for some matrix $P'$. Comparing with \eqref{AJBI} we see that $P'=P$. Thus $Z_1''=PZ_1'$, $C_i''=PC_i$, $C_i'''=PC_i'$. Argue similarly for $B_q$. We then conclude that $q=q_0$. Next, note that $\widetilde{A_q}=\begin{bmatrix} Z'_1&C_1&0&\cdots &C_r&0\\ 0&0&0&\cdots&0&0\end{bmatrix}$, so $A_q^\perp=0$, and similarly $B_q^\perp=0$. Thus $q^\perp=0$.

The fact that $\psi({\bf E}^\circ_{u,w}\cap U)=U_0\times ({\bf E}^\circ_{u^\perp,w^\perp}\cap U^\perp)$, follows from \eqref{eq:rank Aqperp} and \eqref{eq:rank Bqperp}.

This completes the proof of Lemma \ref{lem:transversal slice}.

\subsection{The second proof of trivial local systems}\label{subsection:2nd proof}
In this subsection we give another proof of  Theorem \ref{thm:trivial local system}. We first study simply connectedness of some varieties.

\begin{lemma}\label{lem:E simply connected}
Let $(v,w)$ be $l$-dominant. Then the quasi-affine variety ${\bf E}^\circ_{v,w}$ is simply connected unless $v=(w'_1,w_2)$, $(w'_1,(w'_1-w_1)/r)$, or $((w_2-w'_2)/r,w_2)$. 
\end{lemma}
\begin{proof}
If $v=(0,0)$ then  ${\bf E}^\circ_{v,w}$ is a point, which is simply connected. In the rest, we assume $v\neq (0,0)$.

Consider the following diagram, where $\pi'$ is an isomorphism by Proposition \ref{prop:fiber} (e).
$$
\begin{tikzcd}
\pi^{-1}(\myE^\circ_{v,w})\ar[r, hook,"j"]\ar[d,"\pi' "', "\cong"] & \widetilde{\mathcal{F}}_{v,w} \ar[d,"\pi"]\\
\myE^\circ_{v,w} \ar[r, hook,"i"] & \myE_{v,w} 
\end{tikzcd}
$$
Our idea to study the simply connectedness of ${\bf E}^\circ_{v,w}$ is: first show that $\widetilde{\mathcal{F}}_{v,w}$ is simply connected; next, show that $\pi^{-1}(\myE^\circ_{v,w})$ is simply connected because the complement has codimension greater than 1; last, show that $\myE^\circ_{v,w}$ is simply connected.

We first show that $\widetilde{\mathcal{F}}_{v,w}$ is simply connected. By Proposition \ref{prop:fiber} (b),  ${\mathcal{F}}_{v,w}=\pi^{-1}(0)$ is a fiber bundle over a Grassmannian whose fibers are Grassmannian. It is well known that a complex Grassmannian is simply connected. It follows from the long exact sequence of homotopy groups that ${\mathcal{F}}_{v,w}$ is simply connected.
By Lemma \ref{EF},  $\widetilde{\mathcal{F}}_{v,w}$ is a vector bundle over ${\mathcal{F}}_{v,w}$, so $\widetilde{\mathcal{F}}_{v,w}$ is also simply connected.

Next, we compute the codimension of $\widetilde{\mathcal{F}}_{v,w} \setminus \pi^{-1}(\myE^\circ_{v,w})$ in $\widetilde{\mathcal{F}}_{v,w}$. 
Since 
$$\widetilde{\mathcal{F}}_{v,w} \setminus \pi^{-1}(\myE^\circ_{v,w})=\bigcup_{v'<v, v'\in\D(w)}\pi^{-1}({\bf E}^\circ_{v',w}).$$ 
So the codimension of the complement of $\pi^{-1}(\myE^\circ_{v,w})$ is 
$$\mathrm{Codim}=\min\{\dim\pi^{-1}({\bf E}^\circ_{v',w}) - \dim\pi^{-1}({\bf E}^\circ_{v',w})\}$$

Note the following (where the equalities follow from Proposition \ref{prop:fiber}):
$$\aligned
\dim\pi^{-1}({\bf E}^\circ_{v',w})&=\dim({\bf E}^\circ_{v',w})+\dim Gr(v_2-v'_2,w'_2+rv_1-v'_2) + \dim Gr(v_1-v'_1,w'_1-v'_1)\\
&=\big(v'_2(w_2+w_2'+rv'_1-v'_2)+v'_1(w_1+w_1'-v'_1)\big)\\
&\quad +(v_2-v'_2)(w'_2+rv_1-v_2)
+(v_1-v'_1)(w'_1-v_1)
\\
\endaligned
$$
We define the following function, which satisfies $f(v')=\dim\pi^{-1}({\bf E}^\circ_{v',w})$:
$$\aligned
f(x,y)&=\big(y(w_2+w_2'+rx-y)+x(w_1+w_1'-x)\big)\\
&\quad +(v_2-y)(w'_2+rv_1-v_2)
+(v_1-x)(w'_1-v_1)\\
&=-x^2+rxy-y^2+(w_1+v_1)x+(w_2-rv_1+v_2)y+v_1w'_1+v_2w'_2-v_1^2+rv_1v_2-v_2^2
\endaligned
$$ 
with domain being the closed region 
$$S=\{(x,y)\in\mathbb{R}^2 \ | \  0\le x\le v_1, 0\le y\le v_2, w_1-x+ry\ge 0\}$$ 
Thus 
$$\mathrm{Codim}=\min\{f(v)-f(v')\ |\  v'<v, v'\in \D(w)\}$$

Compute the partial derivatives:
$$f_x=(w_1-x+ry)+(v_1-x)\ge0, \quad  f_y=(w_2-y)+(v_2-y)-r(v_1-x),$$
and note that $w_2\ge v_2\ge y$, and the directional derivative of $f$ along the direction $\langle r,1\rangle$ is 
$$\langle f_x,f_y\rangle \cdot \langle r,1\rangle =r(w_1-x+ry)+(w_2-y)+(v_2-y)\ge0,$$
it is then easy to check that $f$ is strictly increasing in $S$ along each line of slope $0$ or $1/r$, an along the line $x=v_1$.

Next we give a lower bound for $f(v)-f(v')$. Define a path $C$ from $v'$ to $v$ as follows: start $v'$, walk along direction $(r,1)$; if it hits the point $v$ we say the path is of type $1$; if hits the horizontal line $y=v_2$ then walk east to $v$ and we say $C$ is of type 2; if hits the vertical line $x=v_1$ then walk north to $v$ and we say $C$ is of type 3. The fundamental theorem for line integrals asserts
$f(v)-f(v')=\int_C \langle f_x,f_y\rangle \cdot d{\bf r}$.

For a type 2 path $C$, it must contains the point $v''=v-(1,0)$. 
$$f(v)-f(v')\ge f(v)-f(v'')=\int_{v_1-1}^{v_1} f_xdx=w_1-v_1+rv_2+1\ge 1$$
A necessary condition for the equality $f(v)-f(v')=1$ to hold is $w_1-v_1+rv_2=0$ and $(v_1-1,v_2)$ is $l$-dominant. 

For a path $C$ of type 1 or 3, let $v''$ be the point where it hits the vertical line $x=v_1$. Then
$$
\aligned
f(v)-f(v')&=(f(v)-f(v''))+(f(v'')-f(v'))\\
&=\int_{v''_2}^{v_2} (w_2+v_2-2y)dy +\int_{v'_2}^{v''_2} r(w_1-v'_1+rv'_2)+(w_2+v_2-2y) dy\\
&\ge \int_{v'_2}^{v_2} (w_2+v_2-2y) dy\ge \int_{v_2-1}^{v_2} (w_2+v_2-2y) dy=w_2-v_2+1\ge 1\\
\endaligned
$$
For the equality $f(v)-f(v')=1$ to hold, we need $v_2=w_2$; in other cases, $f(v)-f(v')\ge 2$. 





So, $\pi^{-1}({\bf E}^\circ_{v,w})$ is simply connected unless either ``$w_1-v_1+rv_2=0$ and $(v_1-1,v_2)$ is $l$-dominant'' or $v_2=w_2$.

Now we make an important observation: if we simultaneously swap the role of $v_1$ with $v_2$, $w_1$ with $w'_2$, $w_2$ with $w'_1$,  the same argument above would imply that 
$\pi^{-1}({\bf E}^\circ_{v,w})$ is simply connected unless either ``$w'_2-v_2+rv_1=0$ and $(v_1,v_2-1)$ is $l$-dominant'' or $v_1=w'_1$.

Therefore, $\pi^{-1}({\bf E}^\circ_{v,w})$ is simply connected except the following four cases:

Case 1: $w_1-v_1+rv_2=w'_2-v_2+rv_1=0$, and $(v_1-1,v_2)$, $(v_1,v_2-1)$ are $l$-dominant. Then $w_1=w'_2=0$, $r=1$ and $v_1=v_2$. In this case, the $l$-dominant lattices lies in a line segment, a degeneration of the hexagon in Figure \ref{figure:l-dominant} (Left). So neither $(v_1-1,v_2)$ nor $(v_1,v_2-1)$ is $l$-dominant. Thus this case is impossible. 

Case 2: $w_1-v_1+rv_2=0$ and $v_1=w'_1$. In this case $v=(w'_1,(w'_1-w_1)/r)$.

Case 3: $w'_2-v_2+rv_1=0$ and $v_2=w_2$. In this case $v=((w_2-w'_2)/r,w_2)$. 

Case 4: $v=(w'_1,w_2)$. 
\smallskip

Last, assume $\pi^{-1}({\bf E}^\circ_{v,w})$ is simply connected; since it is a fiber bundle over ${\bf E}^\circ_{v,w}$, by the long exact sequence of homotopy groups we conclude that ${\bf E}^\circ_{v,w}$ is also simply connected.
This proves the lemma. 
\end{proof}

Next, we study the simply-connectedness of open determinantal varieties. Let $\mathbb{C}^{m\times n}=\Hom(\mathbb{C}^n,\mathbb{C}^m)$ denote the space of $m\times n$ matrices over $\mathbb{C}$. For nonnegative integers $m,n,t$ satisfying $0\le t\le \min(m,n)$, the open determinantal variety $Y^\circ_t$ (or denoted $Y^\circ_{t;m\times n}$ to specify $m,n$) is defined as
$$Y^\circ_t=Y^\circ_{t;m\times n}=\{ A\in \mathbb{C}^{m\times n} \,|\, {\rm rank }A=t\}.$$  
\begin{lemma}\label{determinantal variety simply connected}
The open determinantal variety $Y^\circ_{t;m\times n}$ is simply connected unless $t=m=n$. 
\end{lemma}
\begin{proof}
For simplicity we denote $Y^{m\times n}_t$ by $Y_t$. 
The closure $Y_t$ of $Y^\circ_t\subseteq \mathbb{C}^{m\times n}$ is usually called the determinantal variety, and is equal to 
$$Y_t=\{ A\in \mathbb{C}^{m\times n} \,|\, {\rm rank }A\le t\}=\bigcup_{i=0}^t Y^\circ_i.$$
Consider the well-known desingularization $p: \tilde{Y}_t\to Y_t$ (see \cite[(6.1.1)]{Weyman})
$$\tilde{Y}_t=\{(A,X)\in Y_t\times Gr(t,m) \, | \, {\im}A\subseteq X\},$$
where $p: (A,X)\mapsto A$. Since the projection $p': \tilde{Y}_t\to Gr(t,m), \, (A,X)\mapsto X$ is a vector bundle of rank $nt$, we see that $\tilde{Y}_t$ is nonsingular, simply connected since $Gr(t,m)$ is simply connected. Moreover,
$$\dim Y_t=\dim Y^\circ_t=\dim \tilde{Y}_t=\dim Gr(t,m)+nt=t(m-t)+nt=t(m+n-t).$$ 

The projection $p$ is stratified by $Y_t=\bigcup_{i=0}^t Y^\circ_i$. The fibers over $Y^\circ_i$ are isomorphic to
$\{X\in Gr(t,m) \, | \, \mathbb{C}^r\subseteq X\}\cong Gr(t-i,m-i)$. So
$$\dim p^{-1}(Y^\circ_i)=\dim Y^\circ_i+\dim Gr(t-i,m-i)=i(m+n-i)+(t-i)(m-t).$$
Then the codimension of $p^{-1}(Y^\circ_i)$ in $\tilde{Y}_t$ is 
$$\dim \tilde{Y}_t-\dim p^{-1}(Y^\circ_i)=t(m+n-t)-i(m+n-i)-(t-i)(m-t)=(t-i)(n-i),$$
which is $\ge 2$ for all $i<t$, unless $t-i=n-i=1$. So we conclude that $Y^\circ_t$ is simply connected unless $t=n$. 

Now swap $m$ with $n$ (i.e., transpose $A$) and make a similar argument, we conclude that $Y^\circ_t$ is simply connected unless $t=m$. This completes the proof. 
\end{proof}

For $s$ satisfying $t\le s\le m$, define $$\tilde{Y}_t^s=\{(A,X)\in Y_t\times Gr(s,m) \, | \, {\im}A\subseteq X\}.$$
\begin{lemma}\label{lem:det variety trivial local system}
For $s$ satisfying $t\le s\le m$, the local systems appeared in the decomposition of the natural projection $\pi:\tilde{Y}_t^s\to Y_t$ are all trivial. 
\end{lemma}
\begin{proof}
Note that $\pi$ is the same as the natural map $\widetilde{\mathcal{F}}_{v,w}\to {\bf E}_{v,w}$ for $r=1$, $w=(0,m,n,0)$, $v=(t,m)$.   Using the analytic version of the Transversal Slice Theorem we see that each summand appeared in the decomposition of $\pi_*(IC_{\tilde{Y}_t^s})$ is of the form $IC_{Y_{t'}}(L)[i]$ where $0\le t'\le t$. If $t'=m=n$ does not hold then $Y^\circ_{t'}$ is simply connected, thus $L$ is trivial. So we only need to consider the case $t'=m=n$, which forces $t'=t=s=m=n$. Since $\pi$ restricts to an isomorphism $\pi^{-1}(Y^\circ_m)\stackrel{\cong}{\to} Y^\circ_m$, the only summand that is supported on $Y_m$ must be $IC_{Y_m}$ with a trivial local system.
\end{proof}

Now we are ready to give the second proof of trivial local systems: 
\begin{proof}[The Second Proof of Theorem \ref{thm:trivial local system}]
Here we only use the analytic version of the Transversal Slice Theorem.  Unlike the First Proof, now we can only conclude that each summand appeared in the decomposition theorem is of the form 
\begin{equation}\label{ICv0}
\textrm{ $IC_{{\bf E}_{v^0,w}}(L)[n]$ for some local system $L$ on ${\bf E}^\circ_{v^0,w}$ }
\end{equation}
because the open neighborhood $U$ in the analytic Transversal Slice Theorem is not necessarily open in Zariski topology. 

To prove $L$ is trivial, by Lemma \ref{lem:E simply connected} we only need to consider the cases when $v^0$ equals $w$, $(w'_1,(w'_1-w_1)/r)$, or  $((w_2-w'_2)/r,w_2)$. We consider two cases:

Case 1: $v^0_1=w'_1$. Since $v^0\le  \bar{v}\le (w'_1,w_2)$, $\bar{v}\le v$,  and $v_1\le w'_1$ (by \eqref{nonempty F}), we must have $v^0_1=\bar{v}_1=v_1=w'_1$, $v^0_2\le \bar{v}_2\le w_2$. Let  $U=\bigcup_{v'} {\bf E}^\circ_{v',w}$ where $v'$ satisfies $v^0\le v'\le \bar{v}$. Consider the following diagram:
$$
\begin{tikzcd}
W=\{({\bf x},{\bf y}, X_2)| \dim X_2=v_2, {\rm im} B({\bf x},{\bf y})\subseteq X_2\}\ar[d,"\pi'' "]&\pi^{-1}(U)\ar[r, hook,"j"]\ar[l,hook',"j' "']\ar[d,"\pi' "] & \widetilde{\mathcal{F}}_{v,w} \ar[d,"\pi"]\\
V=\{({\bf x},{\bf y})| \rank B({\bf x},{\bf y})\le \bar{v}_2\}& U \ar[r, hook,"i"]\ar[l, hook',"i' "'] & \myE_{\bar{v},w} 
\end{tikzcd}
$$
Note that $i'$ is a Zariski open embedding because $V=\bigcup_{v'} {\bf E}^\circ_{v',w}$  where $v'$ satisfies $v'_2\le \bar{v}_2$ (so $V\setminus U=\bigcup_{v'} {\bf E}^\circ_{v',w}$ where $v'$ satisfies $v'_1\le v_1-1$ or $v'\le  v^0-(0,1)$; thus $V\setminus U$ is Zariski closed in $V$). Also note that $j'$ is a Zariski open embedding because for any point $({\bf x},{\bf y},X_1,X_2)\in \pi^{-1}(U)$, we must have $X_1=W'_1$ by the fact that $X_1\subset W'_1$ and $\dim X_1=v_1=w'_1=\dim W'_1$. 

Now $V=\{x_1\in\mathbb{C}^{w_1'w_1}\}\times Y_{\bar{v}_2;(w'_2+rw'_1)\times w_2}$, $W= \{x_1\in\mathbb{C}^{w_1'w_1}\}\times \tilde{Y}^{v_2}_{\bar{v}_2;(w'_2+rw'_1)\times w_2}$. By Lemma \ref{lem:det variety trivial local system}, the locally system appeared in the decomposition of $\pi''_*(IC_W)$ must be trivial, thus the same is true for $\pi'_*(IC_{\pi^{-1}(U)})$. This implies that $L$ is trivial.

Case 2: $v^0_2=w_2$. Since $v^0\le \bar{v}\le v$ and $\bar{v}\le (w'_1,w_2)$, we must have $v^0_2=\bar{v}_2=w_2$, $v_1=\bar{v}_1$, so $\bar{v}=(v_1,w_2)$.  The point $v$ is in the triangle which is the intersection of three half-planes 
\begin{equation}\label{eq:triangle}
\{(x,y)\in\mathbb{R}^2 \ | \ x\le w'_1, y\ge w_2, w'_2-y+rx\ge 0\}
\end{equation}

As seen in \eqref{diagram:FGtoE}, $\pi$ factors as $p_1p_2$ as shown in the following diagram, where all the hooked arrows are Zariski open embedding:
$$
\begin{tikzcd}
\pi^{-1}(U)\ar[d, "p_2|_U "]\ar[r,hook]&\widetilde{\mathcal{F}}_{v,w} \ar[d,"p_2 "]\\
p_1^{-1}(U)\ar[d,"p_1|_U "]\ar[r,hook]&\tilde{\mathcal{H}}_{v,w} \ar[d,"p_1"]\\
U=\{({\bf x},{\bf y}) \in \myE_{\bar{v},w}\  | \  \rank B({\bf x},{\bf y})=w_2\} \ar[r,hook] &\myE_{\bar{v},w} 
\end{tikzcd}
$$
Note that 

$\bullet$ ${\bf E}^\circ_{v^0,w}\subseteq U$ because $v^0_2=w_2$; 

$\bullet$ $p_2|_U$ is a fiber bundle with fiber 
$$\{X_2|\dim X_2=v_2, \im B({\bf x},{\bf y})\subseteq X_2\subseteq W'_2\oplus X_1^{\oplus r}\}\cong Gr(v_2-w_2,w'_2+rv_1-w_2) ;$$

$\bullet$
 $p_1^{-1}(U)\to Gr(v_1,w'_1)$ defined by $({\bf x},{\bf y},X_1)\mapsto X_1$ is a fiber bundle whose fiber is 
$$\aligned
&\{({\bf x},{\bf y})\ | \ \im A({\bf x},{\bf y})\subseteq  X_1, \rank B({\bf x},{\bf y})=w_2\}\\
&\cong \{x_1\in\mathbb{C}^{v_1w_1}\}\times \{B\textrm{ is a $(w'_2+rv_1)\times w_2$ matrix}\ | \ \rank B=w_2\}
\endaligned
$$ so is simply connected unless the exceptional case $w'_2+rv_1=w_2$, by Lemma \ref{determinantal variety simply connected}. Now consider the exceptional case  $w'_2+rv_1=w_2$. In this case, $\bar{v}$ is on the line $w'_2-y+rx=0$. Thus $\bar{v}$ is the left-most point on the horizontal line $y=w_2$ such that $(v^0,w)$ is $l$-dominant, which implies that $\bar{v}=v^0$. Since $v$ is in the triangle \eqref{eq:triangle} and $v_1=\bar{v}_1$, we have $v=\bar{v}=v^0$, and $p_2|_U$ is an isomorphism. So in either situation, the direct summand appeared in the decomposition of $(p_2|_U)_*IC_{\pi^{-1}(U)}$ is (a shift) of $IC_{p_1^{-1}(U)}$, with the trivial local system. 

Next consider $p_1$. Up to an identify map on the affine space factor $\{{\bf x_2}\in\mathbb{C}^{w'_2w_2}\}$, $p_1$ is the same as the map $\tilde{Y}_t\to Y_t$ with $(m,n)=(w'_1, w_1+rw_2)$ and  $t=v_1$. So all direct summand of $(p_1)_*IC_{p_1^{-1}(U)}$ has trivial local system by Lemma \ref{lem:det variety trivial local system}. After restricting to the open subset $U$, we have a similar conclusion for $p_1|_U$. Now composing $p_1|_U$ with $p_2|_U$, we see that all local system appeared in the decomposition of $(\pi|_U)_*IC_{\pi^{-1}(U)}$ are trivial. In particular, the local system $L$ in \eqref{ICv0} must be trivial since ${\bf E}^\circ_{v^0,w}\subseteq U$. 
This completes the proof.
\end{proof}

\end{document}